\documentclass[a4paper, 12pt]{article}

\usepackage{amsmath} 
\usepackage{theorem}
\usepackage{amssymb}
\usepackage{amsfonts}
\usepackage{latexsym}
\usepackage{amscd}
\usepackage{diagramb}
\input GrCalc3.sty
\usepackage{graphicx}
\usepackage{enumerate}

\DeclareMathAlphabet{\mathpzc}{OT1}{pzc}{m}{it}

\newcommand\Rep{\operatorname{Rep}}
\newcommand{\mo}{{\mathcal M}}
\newcommand{\ca}{{\mathcal C}}
\newcommand{\otb}{{\overline{\otimes}}}
\newcommand{\Do}{{\mathcal D}}

\usepackage{xspace,colortbl}
\usepackage{color}

\definecolor{verde}{rgb}{0.,0.7,0.}
\definecolor{indigo}{rgb}{.18, .34, .78}
\definecolor{indigo1}{rgb}{.18, .24, .78}
\definecolor{indigo2}{rgb}{.18, .14, .78}
\definecolor{indigo3}{rgb}{.18, 0., .78}
\definecolor{rojo}{rgb}{1,0,0}
\definecolor{negro}{rgb}{0,0,0}
\definecolor{lila}{rgb}{.46, .16, .78}
\definecolor{lila1}{rgb}{.46, .16, .86}
\definecolor{lila2}{rgb}{.56, .16, .86}
	\definecolor{lila3}{rgb}{.63, .16, .78}
\definecolor{lila4}{rgb}{.7, .16, .78}
\definecolor{lila5}{rgb}{.78, .26, .78}
\definecolor{lila6}{rgb}{.6, 0., .78}

\theoremstyle{plain}
\theoremheaderfont{\normalfont\bfseries}

\newtheorem{thm}{Theorem}[section]

\newtheorem{lma}[thm]{Lemma}
\newtheorem{cor}[thm]{Corollary}
\newtheorem{defn}[thm]{Definition}

\newtheorem{defnlma}[thm]{Definition and Lemma}

\theorembodyfont{\rmfamily}

\newtheorem{rem}[thm]{Remark}
\newtheorem{prop}[thm]{Proposition}
\newtheorem{ex}[thm]{Example}
\newcommand{\qed}{\hfill\quad\fbox{\rule[0mm]{0,0cm}{0,0mm}}  \par\bigskip}

\newcommand{\Ee}{\mathcal{E}}

\newcommand{\x}{\mbox{-}}
\newcommand{\R}{{\mathcal R}}

\newcommand{\w}{\hspace{-0,12cm}}
\newcommand{\Br}{{\rm Br}}
\newcommand{\Cc}{{\mathfrak C}}

\newcommand{\Del}{\boxtimes}
\newcommand{\comp}{\circ}
\newcommand{\iso}{\cong}
\newcommand{\ot}{\otimes}

\newcommand{\C}{{\mathcal C}}
\newcommand{\M}{{\mathcal M}}
\newcommand{\D}{{\mathcal D}}
\newcommand{\F}{{\mathcal F}}
\newcommand{\G}{{\mathcal G}}

\newcommand{\HH}{{\mathcal H}}
\newcommand{\N}{{\mathcal N}}
\newcommand{\A}{{\mathcal A}}

\newcommand{\E}{{\mathcal E}}

\def\dul#1{\underline{\underline{#1}}}
\def\u{\underline}
\def\o{\overline}
\newcommand{\coev}{\rm coev}

\newcommand{\Ll}{{\mathcal L}}
\newcommand{\Pp}{{\mathcal P}}

\newcommand{\crta}{\overline}
\newcommand{\ev}{\rm ev}

\newcommand{\Id}{\operatorname {Id}}
\newcommand{\id}{\operatorname {id}}
\newcommand{\alfa}{\alpha}

\newcommand{\Ker}{\operatorname {Ker}}
\newcommand{\Hom}{\operatorname {Hom}}
\newcommand{\End}{\operatorname {End}}
\newcommand{\Fun}{\operatorname {Fun}}

\newcommand{\Can}{\operatorname {Can}}

\def\colim{{\rm colim}}

\newcommand{\Pic}{\operatorname{Pic}}
\newcommand{\BrPic}{\operatorname{BrPic}}

\newcommand{\Mod}{\operatorname{Mod}}
\newcommand{\Bimod}{\operatorname{Bimod}}
\newcommand{\Inv}{\operatorname{Inv}}

\newcommand{\im}{{\rm Im}\,}

\newcommand{\cref}[1]{Cor.~\ref{c:#1}}

\newcommand{\exlabel}[1]{\label{ex:#1}}
\newcommand{\exref}[1]{Example~\ref{ex:#1}}
\newcommand{\lelabel}[1]{\label{le:#1}}
\newcommand{\leref}[1]{Lemma~\ref{le:#1}}
\newcommand{\eqlabel}[1]{\label{eq:#1}}
\newcommand{\equref}[1]{(\ref{eq:#1})}
\newcommand{\thlabel}[1]{\label{th:#1}}

\newcommand{\prlabel}[1]{\label{pr:#1}}
\newcommand{\prref}[1]{Proposition~\ref{pr:#1}}
\newcommand{\colabel}[1]{\label{co:#1}}
\newcommand{\coref}[1]{Corollary~\ref{co:#1}}

\newcommand{\rmlabel}[1]{\label{rm:#1}}
\newcommand{\rmref}[1]{Remark~\ref{rm:#1}}
\newcommand{\selabel}[1]{\label{se:#1}}
\newcommand{\seref}[1]{Section~\ref{se:#1}}
\newcommand{\sslabel}[1]{\label{ss:#1}}
\newcommand{\ssref}[1]{Subsection~\ref{ss:#1}}

\begin{document}

\title{Coring categories and Villamayor-Zelinsky sequence for symmetric finite tensor categories}
\author{Bojana Femi\'c \vspace{6pt} \\
{\small Facultad de Ingenier\'ia, \vspace{-2pt}}\\
{\small  Universidad de la Rep\'ublica} \vspace{-2pt}\\
{\small  Julio Herrera y Reissig 565} \vspace{-2pt}\\
{\small  11 300 Montevideo, Uruguay}}

\date{}

\maketitle
\begin{abstract}
In the preceeding paper we constructed an infinite exact sequence a la Villamayor-Zelinsky for a symmetric finite tensor category. It consists of
cohomology groups evaluated at three types of coefficients which repeat periodically. In the present paper we interpret the middle cohomology group
in the second level of the sequence. We introduce the notion of coring categories and we obtain that the mentioned middle cohomology group is
isomorphic to the group of Azumaya quasi coring categories. This result is a categorical generalization of the classical Crossed Product Theorem,
which relates the relative Brauer group and the second Galois cohomology group with respect to a Galois field extension. We construct the colimit
over symmetric finite tensor categories of the relative groups of Azumaya quasi coring categories and the full group of Azumaya quasi coring categories
over $vec$. We prove that the latter two groups are isomorphic.

\bigbreak

\end{abstract}

\section{Introduction}

In its origins, the Brauer group was defined for a field: it consists of equivalence classes of
central simple algebras over a field. Its first generalization is over a commutative ring $R$: central simple algebras are replaced by Azumaya algebras $A$, satisfying
$A\ot_RA^{op}\iso\End_R(A)$. Further generalizations are crowned by the Brauer group of a braided monoidal category $\C$ constructed in \cite{VZ1}, it consists of equivalence classes of
Azumaya algebras $A\in\C$ such that $A\ot A^{op}\iso[A,A]$ and $A^{op}\ot A\iso[A,A]^{op}$, where $[A,A]$ denotes the inner hom object.

By the classical Crossed Product Theorem the relative	Brauer group $\Br(l/k)$ is isomorphic to the second Galois cohomology group with respect to the
Galois field extension $l/k$. Since every central simple algebra can be	split by a Galois field extension, the full Brauer group $\Br(k)$ of the field $k$ is a limit
of the relative Brauer groups $\Br(l/k)$ for Galois field extensions $l/k$, and as such it is isomorphic to the second Galois cohomology group with respect to the separable closure of the field.
If we change the setting and we consider Galois extensions of commutative rings, we have that not every Azumaya algebra can be split by a Galois ring extension.
On the other hand, for the relative Brauer group 
instead of the Crossed Product Theorem we have an infinite exact sequence by Villamayor and Zelinsky \cite{ViZ}. It
involves cohomology groups which are evaluated at three types of coefficients so that the three types of cohomology groups appear periodically in
the sequence. As they observed, when the ring extension is faithfully flat, the relative Brauer group embeds into the middle term in the second level of the sequence.
If the extension is faithfully projective, the embedding is indeed an isomorphism, and the Crossed Product Theorem is recovered.

In \cite{CF} we introduced the Brauer group of Azumaya corings and we proved that it is isomorphic to the mentioned middle term cohomology group. In this
setting a commutative ring extension $R\to S$ is faithfully flat and $S\ot_R S$ is the basic Azumaya $S/R$-coring. We also have that the full Brauer group $\Br(R)$,
which is the limit of the relative Brauer groups $\Br(S/R)$ over faithfully flat ring extension $R\to S$, is isomorphic to the full group of Azumaya corings over $R$.
This resolves the deficiency in the cohomological interpretation of the Brauer group of a commutative ring.
Corings were introduced in \cite{Sw4} and a later interest for them was revived by the appearance of \cite{Brz1}, where it was shown that
various structure theorems for entwining structures in the context of gauge theory on non-commutative spaces,
Doi-Koppinen Hopf modules, just to mention some, are special cases of structure theorems for the category of comodules of a coring.


Replacing a commutative ring $S$ by a symmetric finite tensor category $\C$, we proved in \cite{Femic2} that there exists an analogous infinite exact sequence to
that of Villamayor and Zelinsky. On the other hand, replacing Azumaya algebras in a braided monoidal category from the mentioned construction in \cite{VZ1}
by invertible (exact) bimodule categories over a finite tensor category, in \cite{ENO} the authors introduced the Brauer-Picard group of a finite tensor category.
It consists of equivalence classes of exact invertible $\C$-bimodule categories $\M$, where $\C$ is a finite tensor category, satisfying: $\M\Del_{\C}\M^{op}\simeq\C$
(and equivalently $\M^{op}\Del_{\C}\M\simeq\C$). This group is of great importance in the classification of extensions of a given tensor category by a finite group,
and it is also related to mathematical physics, e.g. rational Conformal Field Theory and 3-dimensional Topological Field Theory, \cite{fsv, KK}.
In particular, it is also shown that $\BrPic(Vec_k)\iso\Br(k)$, that is, the Brauer-Picard group of the category of $k$-vector spaces is isomorphic to the Brauer
group of the field $k$. When $\C$ is braided one may consider one-sided $\C$-module categories in which the action on the other side is induced via the braiding.
Their equivalence classes form a group, called the Picard group of $\C$ and denoted by $\Pic(\C)$, and it is a subgroup of $\BrPic(\C)$. It is also known, \cite{DN},
that $\Pic(\C)$ is isomorphic to the Brauer group of {\em exact} Azumaya algebras. In particular, for the braided monoidal category of modules over a finite-dimensional
quasitriangular Hopf algebra $(H, \R)$, it is proved in \cite{DZ} that the Picard and the Brauer group of $\C={}_H\M$ are isomorphic.

In the present paper we introduce the notion of coring categories and of Azumaya quasi coring categories. We prove that the middle cohomology group in the second
level of the sequence from \cite{ Femic2} is isomorphic to the group of Azumaya quasi coring categories. These are invertible categories coming from $\Pic(\C\Del\C)$
equipped with an additional structure. The reason that we obtain {\em quasi} coring categories (coring categories without the counit functor) is that we do not have a
suitable notion of ``faithful flatness'' in the setting of module categories (we lack of examples for a formal definition). This notion was crucial in recovering the counit
for the Azumaya corings in \cite{CF}. Nevertheless, \prref{coh maps dul} is valid even without the ``faithful flatness'' condition. This result is fundamental in \seref{Full group}
for few reasons. Firstly, to construct the colimit of the relative groups of Azumaya quasi coring categories, secondly to introduce the corresponding full group
and finally to prove that the two groups are isomorphic. The interpretation of the middle term in the first level of our infinite exact sequence encounters
the same difficulty as the one mentioned above concerning the ``faithful flatness'' condition, whereas the interpretation in the third level is left for a further investigation.

\medskip

The paper is organized as follows. In the next section we recall some definitions as well as some results from \cite{Femic2} that we will be using here, and we pursue
with some basic results. In \seref{adjunctions} we develop some results based on certain adjoint pairs of functors which will help us to obtain the 1-1 correspondence in
\leref{alfa-delta}. In the following section we introduce coring categories over finite tensor categories, which will be used in our cohomological interpretation.
\seref{Amitsur coh} is dedicated to recall our Amitsur cohomology over symmetric finite tensor categories from \cite{Femic2} and we consider what we call extended cocycles.
We also prove here some new results, among others \prref{Knus} and \prref{coh maps dul} by which any two functors $F, G: \C\to\D$ between symmetric finite tensor categories
induce one and the same map between the n-th Amitsur cohomology groups: $F_*=G_*:\ H^n(\C,P)\to H^n(\D,P)$, where $P$ is any of the three types of coefficients considered
in our sequence. This result enables one to construct the colimit over symmetric finite tensor categories of $H^n(\bullet/vec,P)$.
The subject of \seref{Interpret} is the interpretation of the middle cohomology group in the second level of the sequence, where we get the group of Azumaya
quasi $\C$-coring categories, for a symmetric finite tensor category $\C$. We also construct the connecting map in the second level of the sequence that ends up in the novel group.
In the last section we prove that the colimit over symmetric finite tensor categories of the relative groups of Azumaya quasi coring categories is isomorphic to the
full group of Azumaya quasi coring categories over $vec$.

\section{Preliminaries and notation}

Throughout $k$ will be an algebraically closed field of characteristic zero, $I$ will denote the unit object in a monoidal category $\C$.
When there is no confusion we will denote the identity functor on a category $\M$ by $\M$. We recall some definitions and basic properties.

A {\em finite category} over $k$ is a $k$-linear abelian category equivalent to a category of finite-dimensional representations of a finite-dimensional $k$-algebra.
A {\em tensor category} over $k$ is a $k$-linear abelian rigid monoidal category such that the unit object is simple. An object $X$ is said to be {\em simple}
if $\End(X)=k\Id_X$.
A {\em finite tensor category} is a tensor category such that the underlying category is finite. For a finite-dimensional Hopf algebra $H$ (or, more generally,
a finite-dimensional quasi-Hopf algebra) the category of its representations $\Rep H$ is a finite tensor category.

All categories will be finite, all tensor categories will be over $k$, and all functors will be $k$-linear.

We assume the reader is familiar with the notions of a left, right and bimodule categories over a tensor category, (bi)module functors,
Deligne tensor product of finite abelian categories, tensor product of bimodule categories and exact module categories.
For the respective definitions we refer to \cite{EO}, \cite{ENO}, \cite{Gr}, \cite{EGNO}.
The Deligne tensor product bifunctor $-\Del-$ and the action bifunctor for module categories $-\crta\ot-$ are biexact in both variables \cite[Proposition 1.46.2]{EGNO}.

For finite tensor categories $\C, \D, \E$, a $\C\x\D$-bimodule category $\M$ and a $(\D,\E)$-bimodule category $\N$, the tensor product over $\D$:
$\M\boxtimes_{\D}\N$ is a $(\C,\E)$-bimodule category. Given an $\E\x\F$-bimodule category $\Pp$, there is a canonical equivalence of $\C\x\F$-bimodule
categories: $(\M\Del_{\D}\N)\Del_{\E}\Pp\simeq\M\Del_{\D}(\N\Del_{\E}\Pp)$, \cite[Remark 3.6]{ENO}.

\medbreak

Given a tensor functor $\eta: \C\to\E$, then $\E$ is a left (and similarly a right) $\C$-module category.
The action bifunctor $\C\times\E\to\E$ is given by $\ot(\eta\times\Id_{\E})$, where $\ot$ is the tensor product in $\E$, and the associator functor is
\begin{equation} \eqlabel{associator}
m_{X,Y,F}=\alpha_{\eta(X), \eta(Y), E}(\xi_{X,Y}\ot E)
\end{equation}
for every $X, Y\in\C$ and $E\in\E$, where $\xi_{X,Y}: \eta(X\ot Y)\to\eta(X)\ot\eta(Y)$
determines the monoidal structure of the functor $\eta$ and $\alpha$ is the associativity constraint for $\E$. The constraint for the action of the unit
is defined in the obvious manner. Moreover, $\E$ is a $\C\x\E$-bimodule category with the bimodule constraint
$\gamma_{X, E, F}: (X\crta\ot E)\ot F\to X\crta\ot (E\ot F)$ for $X\in\C, E,F\in\E$, given via $\gamma_{X, E, F}=\alpha_{\eta(X), E,F}$.

\medskip

{\bf $\C$-balanced functors}. Let $\M$ be a right $\C$-module and $\N$ a left $\C$-module category.
For any abelian category $\A$ a bifunctor $F: \M\times\N\to\A$ additive in every argument is called {\em $\C$-balanced}
if there are natural isomorphisms $b_{M,X,N}: F(M\crta\ot X, N) \stackrel{\iso}{\to} F(M, X\crta\ot N)$ for all $M\in\M, X\in\C, N\in\N$  s.t.
\begin{equation} \eqlabel{C-balanced}
\scalebox{0.84}{
\bfig \hspace{-1cm}
\putmorphism(-200,500)(1,0)[F((M\crta\ot X)\crta\ot Y, N)` F(M \crta\ot(X\ot Y), N)`F(m_{M,X,Y}^r, N)]{1500}{-1}a
\putmorphism(1300,500)(1,0)[\phantom{(X \ot (Y \ot U)) \ot W}`F(M,(X\ot Y)\crta\ot N)` b_{M,X \ot Y,N}]{1350}1a
\putmorphism(2870,500)(0,-1)[``F(M, m_{X,Y,N}^l)]{500}1l
\putmorphism(-160,500)(0,-1)[``b_{M\crta\ot X,Y,N}]{500}1r
\putmorphism(-200,0)(1,0)[F((M\crta\ot X), Y\crta\ot N)` F(M, X \crta\ot(Y \crta\ot N))` b_{M,X,Y \crta\ot N}]{2960}1b
\efig}
\end{equation}
commutes.

\medskip

A {\bf $\C$-balanced natural transformation} $\Psi: F\to G$ between two $\C$-balanced functors $F, G: \M\times\N\to\A$ with their respective balancing isomorphisms $f_X$ and $g_X$,
is a natural transformation such that the following diagram commutes:
\begin{equation}  \eqlabel{C balanced nat tr}
\scalebox{0.88}{\bfig
 \putmorphism(0,400)(1,0)[F((M\crta\ot X), N)` F(M, (X\crta\ot N))`f_{M,X,N}]{1500}1a
 \putmorphism(0,0)(1,0)[G((M\crta\ot X), N) ` G(M, (X\crta\ot N)).` g_{M,X,N}]{1500}1a
\putmorphism(0,400)(0,-1)[\phantom{B\ot B}``\Psi(M\crta\ot X, N)]{380}1l
\putmorphism(1500,400)(0,-1)[\phantom{B\ot B}``\Psi(M, X\crta\ot N)]{380}1r
\efig}
\end{equation}

\medskip



\medbreak

A right $\ca$-module category $\M$ gives rise to a left $\C$-module category $\mo^{op}$ with the action
given by \equref{left op} and associativity isomorphisms $m^{op}_{X,Y,M}= m_{M,{}^*Y, {}^*X}$ for all $X, Y\in \ca, M\in \mo$.
Similarly, a left $\ca$-module category $\mo$ gives rise to a right $\ca$-module category $\mo^{op}$ with the action 
given via \equref{right op}. Here ${}^*X$ denotes
the left dual object and $X^*$ the right dual object for $X\in\C$. If $\mo$ is a $(\ca,\Do)$-bimodule category then $\mo^{op}$ is a
$(\Do,\ca)$-bimodule category and $(\mo^{op})^{op}\iso\M$ as $(\ca,\Do)$-bimodule categories. \vspace{-0,7cm}
\begin{center}
\begin{tabular}{p{4.8cm}p{1,2cm}p{5.4cm}}
\begin{eqnarray}  \eqlabel{left op}
X\crta\ot^{op}M=M\otb {}^*X
\end{eqnarray}  & &
\begin{eqnarray} \eqlabel{right op}
M\crta\ot^{op}X=X^*\crta\ot M
\end{eqnarray}
\end{tabular}
\end{center} \vspace{-0,7cm}

For a $\C\x\D$-bimodule functor $\F:\M\to\N$ the $\D\x\C$-bimodule functor ${}^{op}\F: {}^{op}\N\to{}^{op}\M$ is given by
${}^{op}\F=\sigma_{\M, \C}^{-1}\comp\F^*\comp\sigma_{\N, \C}$. Here $\F^*: \Fun_{\C}(\N, \C)\to\Fun_{\C}(\M, \C)$ is given by
$\F^*(G)=G\comp\F$ and $\sigma$ is the equivalence proved in \cite[Lemma 4.3]{Femic2} and given by
$\sigma_{\M,\N}: \M\Del_{\C}{}^{op}\N\to\Fun(\N,\M)_{\C}, \sigma(M\Del_{\C}N)=M\crta\ot \o\Hom_{\N}(-, N)$.
If $\F$ is an equivalence, it is $(\F^*)^{-1}=(\F^{-1})^*=-\comp\F^{-1}$ and consequently: $({}^{op}\F)^{-1}={}^{op}(\F^{-1})$.

\medskip

A $(\ca, \Do)$-bimodule category $\mo$ is called \emph{invertible} \cite{ENO} if there are  equivalences of bimodule categories
$$\mo^{op}\boxtimes_{\ca} \mo\simeq \Do, \quad \mo\boxtimes_{\Do} \mo^{op}\simeq \ca.$$
The group of equivalence classes of exact invertible module categories over a finite tensor category $\C$ is called the Brauer-Picard group. It was
introduced in \cite{ENO} and it is denoted by $\BrPic(\C)$.

\medskip

{\bf One-sided $\C$-bimodule categories.}
When $\C$ is braided with a braiding $\Phi$, then every left $\C$-module category is a right and a $\C$-bimodule category:
$M\crta\ot X=X\crta\ot M$ with the isomorphism functors
$m_{M,X,Y}^r: M\crta\ot (X\ot Y)\to(M\crta\ot X)\crta\ot Y$ defined via:
\begin{equation} \eqlabel{right associator}
\scalebox{0.84}{
\bfig 
\putmorphism(0,500)(1,0)[M \crta\ot (X\ot Y)	`(M\crta\ot X)\crta\ot Y` m_{M,X,Y}^r]{2660}1a
\putmorphism(0,0)(1,0)[(X \w\ot\w Y) \crta\ot M` (Y\ot X)\crta\ot M` \Phi_{X,Y}\crta\ot M]{1450}1b
\putmorphism(1400,0)(1,0)[\phantom{(X \ot (Y \ot U))}`Y \crta\ot (X\crta\ot M),` m_{Y,X,M}^l]{1250}1b
\putmorphism(60,500)(0,1)[`` =]{500}1l 
\putmorphism(2670,500)(0,1)[`` =]{500}1r 
\efig}
\end{equation}
see \cite[Section 2.8]{DN}. 
Moreover, the $\C$-bimodule associativity constraint is given by:
\begin{equation} \eqlabel{mixed assoc}
\scalebox{0.84}{
\bfig
\putmorphism(0,500)(1,0)[(X\crta\ot M) \crta\ot Y	`X\crta\ot (M\crta\ot Y)` a_{X,M,Y}]{3100}1a
\putmorphism(0,0)(1,0)[Y\crta\ot (X\crta\ot M)` (Y\w\ot X)\crta\ot M`(m^l)^{-1}_{Y,X,M}]{1050}1b
\putmorphism(1000,0)(1,0)[\phantom{(X \ot (Y \ot U))}`(X\w\ot Y)\crta\ot M` \Phi_{Y,X}\ot M]{1100}1b
\putmorphism(2100,0)(1,0)[\phantom{(X \ot (Y \ot U))}`X \crta\ot (Y\crta\ot M)` m_{X,Y,M}^l]{1000}1b
\putmorphism(60,500)(0,1)[``=]{500}1l 
\putmorphism(3100,500)(0,1)[``=]{500}{-1}r 
\efig}
\end{equation}
for all $X,Y\in\C, M\in\M$. The $\C$-bimodule categories obtained in this way are called {\em one-sided $\C$-bimodule categories.} We denote their category by $\C^{br}\x\dul\Mod$.
We have that $(\C^{br}\x\dul\Mod, \Del_{\C}, \C)$ is monoidal.

\medskip

For a braided finite tensor category $\C$ we denote by $\dul\Pic(\C)$ the monoidal category $(\dul\Pic(\C), \Del_{\C}, \C)$ of exact invertible
one-sided $\C$-bimodule categories. The Grothendieck group of $\dul\Pic(\C)$ -- the Picard group of equivalence classes of exact invertible
one-sided $\C$-bimodule categories -- is denoted by $\Pic(\C)$. It is a subgroup of $\BrPic(\C)$, \cite[Section 4.4]{ENO}, \cite[Section 2.8]{DN}.
Moreover, $\Inv(\C)$ will denote the group of isomorphism classes of
invertible objects in $\C$ (objects $X$ such that there exists an object $Y\in\C$ such that $X\ot Y\iso I\iso Y\ot X$). The inverse object of $X$ we will
denote by $X^{-1}$. The product is induced by the tensor product in $\C$. If $\C$ is braided $\Inv(\C)$ is an abelian group (in two ways).

\medskip

In \cite[Proposition 4.2 and Proposition 4.7]{Femic2} we proved that given a $\C$-bimodule category $\M$ its left and right dual objects in the
monoidal category of $\C$-bimodule categories
$(\C\x\Bimod, \C, \Del_{\C})$ are $\M^{op}$ and ${}^{op}\M$, respectively. The corresponding dual basis functors are given as follows.
$$\ev: \M\Del_{\C}\M^{op}\to\C\quad\textnormal{and}\quad\coev: \C\to \M^{op}\Del_{\C}\M$$
where $\ev(M\Del_{\C}N)=\u\Hom_{\M}(M,N)$ and $\coev(I)=\oplus_{i\in J} W_i\Del_{\C}V_i$ is such an object in $\M^{op}\Del_{\C} \M$
that $\Id_{\M}= \oplus_{i\in J} \u\Hom_{\M}(-, W_i)\crta\ot V_i$. Similarly, for
$$\crta\ev: {}^{op}\M\Del_{\C}\M\to\C\quad\textnormal{and}\quad\crta\coev: \C\to \M\Del_{\C}{}^{op}\M$$
one has: $\crta\ev(M\Del_{\C}N)=\o\Hom_{\M}(N,M)$ and $\crta\coev(I)=\oplus_{i\in J} V_i\Del_{\C}W_i$ is such an object in $\M\Del_{\C} {}^{op}\M$
that
\begin{equation} \eqlabel{crta coev eq}
\Id_{\M}= \oplus_{i\in J} V_i\crta\ot\o\Hom_{\M}(-, W_i).
\end{equation}
Here $\u\Hom_{\M}(M,N)$ and $\o\Hom_{\M}(M,N)$ are inner hom objects in $\C$ determined as follows: $\u\Hom_{\M}(M,-):\M\to\C$ is a right adjoint functor
to $-\crta\ot M:\C\to\M$, and $\o\Hom_{\M}(M,-)$ is a right adjoint functor to $M\crta\ot-$.
We also proved that if $\M$ is an invertible $\C$-bimodule category, the evaluation and coevaluation functors are equivalences. Moreover, if $\C$ is symmetric,
the left and the right dual objects coincide: $\M^{op}={}^{op}\M$. We also have:

\begin{cor} \cite[Corollary 4.11]{Femic2}  \colabel{coev inverse ev}
For a symmetric finite tensor category $\C$ and $\M\in\dul{\Pic}(\C)$, it is $\ev^{-1}\simeq\coev$.
In particular, it is $\oplus_{i\in J}\u\Hom_{\M}(V_i, W_i)\iso I$.
\end{cor}

\begin{cor} \cite[Corollary 3.3]{Femic2} \colabel{symmetric}
If $\C$ is a symmetric finite tensor category, then every $\C$-bimodule category $\M$ is a one-sided $\C\Del\C$-bimodule category.
\end{cor}

As we commented after \cite[Corollary 3.3]{Femic2}, for a symmetric finite tensor category $\C$ we have:
\begin{equation} \eqlabel{summ one sided bimod}
\M\in\C\Bimod\quad \Leftrightarrow\quad \M\in\C\Del\C\x\Mod \quad\Leftrightarrow\quad \M\in\Mod\x\C\Del\C
\end{equation}
and moreover $\M$ is a one-sided $\C\Del\C$-bimodule category.

\medskip

For any finite tensor category $\C$, in order to simplify the notation we will often write: $X_1\cdots X_n$ for the tensor product $X_1\ot \dots\ot X_n$ in $\C$.
We will write an object $X\in \C^{\Del n}$ as $X=X^1\Del X^2\Del \cdots \Del X^n$, where the direct summation we understand implicitly.
The category $\C^{\Del n}$ is a finite tensor category with unit object $I^{\Del n}$ and the componentwise tensor product: $(X^1\Del X^2\Del \cdots \Del X^n)\odot
(Y^1\Del Y^2\Del \cdots \Del Y^n)=X_1Y_1\Del\cdots\Del X_nY_n$.

\medskip

For a braided (symmetric) finite tensor category $\C$ the category $\C^{\Del n}$ is a braided (symmetric) finite tensor category for every $n\in\mathbb N$.
If $\C$ is braided, every left $\C^{\Del n}$-module category is a $\C^{\Del n}$-bimodule category so we may consider the Picard group $\Pic(\C^{\Del n})$.
We proved:

\begin{prop} \cite[Proposition 3.4]{Femic2}  \prlabel{sim bimod}
For a symmetric finite tensor category $\C$ the category $(\C^{br}\x\dul\Mod, \Del_{\C}, \C)$ is symmetric monoidal.
\end{prop}

Consequently, for a symmetric finite tensor category $\C$ the Picard group $\Pic(\C^{\Del n})$ is abelian for every $n\in\mathbb N$. The main result of our previous
paper is:

\begin{thm} \cite[Theorem 8.2]{Femic2} \thlabel{VZ}
Let $\C$ be a symmetric finite tensor category. There is a long exact sequence
\begin{eqnarray}\eqlabel{VZ seq}
1&\longrightarrow&H^2(\C,\Inv)\stackrel{\alpha_2}{\longrightarrow}H^1(\C,\dul{\Pic})\stackrel{\beta_2}{\longrightarrow}H^1(\C,{\Pic})\\
&\stackrel{\gamma_2}{\longrightarrow}&H^3(\C,\Inv)\stackrel{\alpha_3}{\longrightarrow}H^2(\C,\dul{\Pic})\stackrel{\beta_3}{\longrightarrow}H^2(\C,{\Pic})\nonumber\\
&\stackrel{\gamma_3}{\longrightarrow}&H^4(\C,\Inv)\stackrel{\alpha_4}{\longrightarrow}H^3(\C,\dul{\Pic})\stackrel{\beta_4}{\longrightarrow}H^3(\C,{\Pic})\nonumber\\
&\stackrel{\gamma_4}{\longrightarrow}&\cdots \nonumber
\end{eqnarray}
\end{thm}

In \seref{Amitsur coh} we will recall the construction of the category $\dul{\Pic}(\C)$.
The following will be useful throughout the paper:

\begin{lma} \cite[Lemma 2.1]{Femic2} \lelabel{Weq}
Let $\C$ and $\D$ be finite tensor categories. Given a right $\C$-module category $\M$, a left $\C$-module category $\M'$,
a right $\D$-module category $\N$ and a left $\D$-module category $\N'$, there is an isomorphism of categories:
$$(\M\Del_{\C}\M')\Del(\N\Del_{\D}\N')\iso(\M\Del\N)\Del_{\C\Del\D}(\M'\Del\N').$$
\end{lma}

\begin{lma} \cite[Lemma 4.9]{Femic2} \lelabel{basic lema}
Let $\F, \G: \M\to\N$ be $\D\x\C$-bimodule functors and let $\Pp$ be an invertible $\C$-bimodule category.
\begin{enumerate}
\item It is $\F\Del_{\C}\Id_{\Pp}=\G\Del_{\C}\Id_{\Pp}$ if and only if $\F=\G$.
\item If $\HH: \Pp\to\Ll$ is a $\C$-bimodule equivalence functor, it is $\F\Del_{\C}\HH=\G\Del_{\C}\HH$ if and only if $\F=\G$.
\end{enumerate}
\end{lma}

Now we prove:

\begin{lma} \lelabel{alfa dagger}
Let $\C$ be a finite tensor category and $\M$ an exact $\C$-bimodule category and let $\alpha: \M\to\C$ be a $\C$-bimodule equivalence. Define
$\alpha^{\dagger}=\crta\ev(id_{{}^{op}\M}\Del_{\C}\alpha^{-1}): {}^{op}\M\to\C$.
Then:
\begin{enumerate}
\item $\crta\ev=\alpha^{\dagger} \Del_{\C}\alpha$ as right $\C$-linear functors.
\item $\alpha^{\dagger}=({}^{op}\alpha)^{-1}$ as $\C$-bimodule functors.
\end{enumerate}
\end{lma}

\begin{proof}
Take $M\in {}^{op}\M$ and $N\in\M$.
$1)$. Observe that $\alpha^{\dagger}$ is the composition:
$$ \bfig
\putmorphism(0, 0)(1, 0)[\alpha^{\dagger}=({}^{op}\M\simeq {}^{op}\M\Del_{\C}\C` {}^{op}\M\Del_{\C}\M` {}^{op}\M\Del_{\C}\alpha^{-1}]{1420}1a
\putmorphism(1420, 0)(1, 0)[\phantom{ {}^{op}\M\Del_{\C}\M}` \C).`\crta\ev]{560}1a
\efig
$$
We compute:
\begin{eqnarray*}
(\alpha^{\dagger} \Del_{\C}\alpha)(M\Del_{\C}N)&=&
\alpha^{\dagger} (M)\Del_{\C} \alpha(N)\\
&=& \crta\ev(M\Del_{\C}\alpha^{-1}(I))\Del_{\C}\alpha(N)\\
&=& \crta\ev(M\Del_{\C}\alpha^{-1}(I)\crta\ot\alpha(N))\\
&=& \crta\ev(M\Del_{\C}\alpha^{-1}(I\ot\alpha(N)))\\
&=& \crta\ev(M\Del_{\C}N).
\end{eqnarray*}
In the third and the fourth equality we used that $\crta\ev$, respectively $\alpha^{-1}$, is right $\C$-linear.
The second claim follows by \leref{basic lema} from the parts 1) and 2).
\qed\end{proof}

\subsection{Tensor product of bimodule categories as a braided monoidal category} \sslabel{product cats}

In \cite{Femic2} we introduced Amitsur cohomology associated to Deligne tensor powers $\C^{\Del n}$, which is a braided monoidal category if so is $\C$. In
\ssref{Extended cocycles} we will want to consider Amitsur complex assotiated to Deligne tensor powers of the form $(\C\Del\C)^{\Del_{\C} n}$. For this purpose
let us see how the braided monoidal structure of such a category looks like.


\begin{lma} \lelabel{braided: product over braided}
Let $\D$ be a $\C$-bimodule category which is a left and a right $\C$-module category through two tensor functors $\sigma, \tau:\C\to\D$, respectively.
Suppose that $\D$ is braided.  
Then $\D\Del_{\C}\D$ is a braided finite tensor category.
\end{lma}

\begin{proof}
We define the tensor product $\boxdot$ between two objects $D\Del_{\C}D'$ and $E\Del_{\C}E'$ in $\D\Del_{\C}\D$ as follows:
$$(D\Del_{\C}D')\boxdot(E\Del_{\C}E'):=DE\boxdot D'E'$$
(or, what is the same, so that the following diagram commutes:
\begin{equation*} 
\scalebox{0.88}{\bfig
 \putmorphism(-30,400)(1,0)[\D\Del\D\times\D\Del\D` \D\Del_{\C}\D\times\D\Del_{\C}\D `\Pi\times\Pi]{1630}1a
 \putmorphism(-50,0)(1,0)[\D\Del\D` \D\Del_{\C}\D.)` \Pi]{1700}1a
\putmorphism(-60,400)(0,-1)[\phantom{B\ot B}``\odot]{380}1l
\putmorphism(1630,400)(0,-1)[\phantom{B\ot B}``\boxdot]{380}1r
\efig}
\end{equation*}
In order to see that this is a well-defined operation, we need to check that it is $\C$-balanced at two places, so that
$$ ((D\crta\ot X)\Del_{\C}D')\boxdot(E\Del_{\C}E')\iso(D\Del_{\C}(X\crta\ot D'))\boxdot(E\Del_{\C}E')$$
and
$$(D\Del_{\C}D')\boxdot((E\crta\ot X)\Del_{\C}E')\iso(D\Del_{\C}D')\boxdot(E\Del_{\C}(X\crta\ot E'))$$
for all $X\in\C$. We will consider the corresponding isomorphisms defined as follows:
\begin{equation} \hspace{-2cm}  \eqlabel{left C balance}
\scalebox{0.88}{\bfig
 \putmorphism(-130,400)(1,0)[((D\crta\ot X)\Del_{\C}D')\boxdot(E\Del_{\C}E')` (D\Del_{\C}(X\crta\ot D'))\boxdot(E\Del_{\C}E') `b_{D,X,D'}]{2530}1a
 \putmorphism(-50,0)(1,0)[D\tau(X)E\Del_{\C}D'E' ` DE\tau(X)\Del_{\C}D'E'=DE\Del_{\C}\sigma(X)D'E' ` D\ot\tilde\Phi_{\tau(X),E}\Del_{\C}D'E']{2300}1a
\putmorphism(0,400)(0,-1)[\phantom{B\ot B}``=]{380}1l
\putmorphism(2530,400)(0,-1)[\phantom{B\ot B}``=]{380}1r
\efig}
\end{equation}
and
\begin{equation} \hspace{-2cm}  \eqlabel{right C balance}
\scalebox{0.88}{\bfig
 \putmorphism(80,400)(1,0)[(D\Del_{\C}D')\boxdot((E\crta\ot X)\Del_{\C}E')` (D\Del_{\C}D')\boxdot(E\Del_{\C}(X\crta\ot E')) `\tilde b_{E,X,E'}]{2530}1a
 \putmorphism(350,0)(1,0)[DE\tau(X)\Del_{\C}D'E'=DE\Del_{\C}\sigma(X)D'E' ` DE\Del_{\C}D'\sigma(X)E'. ` DE\Del_{\C}\tilde\Phi_{\sigma(X),D'}\ot E']{2300}1a
\putmorphism(-40,400)(0,-1)[\phantom{B\ot B}``=]{380}1l
\putmorphism(2500,400)(0,-1)[\phantom{B\ot B}``=]{380}1r
\efig}
\end{equation}
We only show that the isomorphism $b_{D,X,D'}$ satisfies the coherence diagram \equref{C-balanced}, the check for $\tilde b_{E,X,E'}$ 
is done similarly. For this purpose we should prove that \vspace{1cm}
\begin{equation*}
\scalebox{0.76}{
\bfig \hspace{-2cm}
\putmorphism(850,760)(3,1)[ ` D\tau(XY)E\Del_{\C}D'E' ` ]{400}0l  
\putmorphism(200,660)(3,1)[ `  ` ]{400}{-1}l
\putmorphism(250,720)(3,1)[ `  ` D\ot\omega_{\tau}\ot E\Del_{\C}D'E']{400}0l 
\putmorphism(-300,500)(2,1)[D\tau(X)\tau(Y)E\Del_{\C}D'E' ` ` ]{400}0l
\putmorphism(1440,1000)(3,-1)[\phantom{(X \ot (Y \ot U)) \ot W}`  `  D\ot\tilde\Phi_{\tau(XY),E}\Del_{\C}D'E']{1300}0r 
\putmorphism(1500,950)(3,-1)[\phantom{(X \ot (Y \ot U)) \ot W}`  `  ]{1300}1r
\putmorphism(1300,500)(1,0)[\phantom{(X \ot (Y \ot U)) \ot W}` DE\Del_{\C}\sigma(XY)D'E' ` ]{1450}0a
\putmorphism(2970,500)(0,-1)[``DE\Del_{\C}\omega_{\sigma}\ot D'E']{500}1l
\putmorphism(-160,500)(0,-1)[``D\tau(X)\ot\tilde\Phi_{\tau(Y),E}\Del_{\C}D'E']{500}1r
\putmorphism(-200,0)(1,0)[D\tau(X)E\Del_{\C}\sigma(Y)D'E' ` DE\Del_{\C}\sigma(X)\sigma(Y)D'E'` D\ot\tilde\Phi_{\tau(X),E}\Del_{\C}\sigma(Y)D'E']{2960}1b
\efig}
\end{equation*}
commutes. This turns out to be fulfilled, as we have:
$$
\scalebox{0.9}[0.9]{
\gbeg{4}{6}
\got{2}{\tau(XY)} \got{1}{E} \gnl
\glmptb \gnot{\hspace{-0,4cm}\omega_{\tau}} \grmpb \gcl{1}  \gnl
\gcl{1} \gbr \gnl
\gbr \gmp{} \gnl
\gcl{1} \gmp{} \gcn{1}{1}{1}{3} \gnl
\gob{1}{\hspace{-0,4cm}E} \gob{1}{\sigma(X)} \gob{3}{\sigma(Y)}
\gend}\stackrel{nat. \tilde\Phi}{=}
\scalebox{0.9}[0.9]{
\gbeg{4}{6}
\got{1}{\hspace{-0,4cm}\tau(XY)} \got{1}{\hspace{0,2cm}E} \gnl
\gbr \gnl
\gcl{1} \gmp{} \gnl
\gcl{1} \glmptb \gnot{\hspace{-0,4cm}\omega_{\sigma}} \grmpb \gnl
\gcl{1} \gcl{1} \gcn{1}{1}{1}{2} \gnl
\gob{1}{\hspace{-0,4cm}E} \gob{1}{\sigma(X)} \gob{3}{\sigma(Y)}
\gend}
$$
(here the empty circle denotes that $\tau$ switches to $\sigma$ because the corresponding object from $\C$ passes to act from the other side of $\Del_{\C}$).
The associativity constraint on $\D\Del_{\C}\D$  is induced by $\alpha\Del_{\C}\alpha$, where $\alpha$ is the associativity
constraint for $\D$. Then it is clear that $\D\Del_{\C}\D$ is a monoidal category.

The braiding $\Psi$ in $\D\Del_{\C}\D$ between two objects $D\Del_{\C}D'$ and $E\Del_{\C}E'$ is given by the following morphism in $\D\Del_{\C}\D$:
\begin{equation} \eqlabel{big braiding Psi}
\scalebox{0.88}{\bfig
 \putmorphism(-90,350)(1,0)[(D\Del_{\C}D')\boxdot(E\Del_{\C}E')` (E\Del_{\C}E')\boxdot(D\Del_{\C}D') `\Psi]{1700}1a
 \putmorphism(-50,0)(1,0)[DE\Del_{\C}D'E'` ED\Del_{\C}E'D'.`\tilde\Phi\Del_{\C}\tilde\Phi]{1700}1b
\putmorphism(-90,350)(0,-1)[\phantom{B\ot B}``=]{350}1l
\putmorphism(1600,350)(0,-1)[\phantom{B\ot B}``=]{350}1r
\efig}
\end{equation}
We have to check that $\Psi$ is a $\C$-balanced natural transformation at two places. This means that two coherence diagrams like \equref{C balanced nat tr}
should commute, one of which we show and it is the following:
$$\scalebox{0.88}{\bfig
 \putmorphism(-20,400)(1,0)[((D\tau(X))E\Del_{\C}D'E'` DE\Del_{\C}(\sigma(X) D')E' ` b_{D,X,D'}]{1520}1a
 \putmorphism(0,0)(1,0)[E(D\tau(X))\Del_{\C}E'D' ` ED\Del_{\C}E'(\sigma(X)D'). ` \tilde b_{D,X,D'}]{1500}1a
\putmorphism(100,400)(0,-1)[\phantom{B\ot B}``\tilde\Phi_{D\tau(X), E}\Del_{\C}\tilde\Phi_{D', E'}]{380}1l
\putmorphism(1300,400)(0,-1)[\phantom{B\ot B}``\tilde\Phi_{D, E} \Del_{\C} \tilde\Phi_{\sigma(X)D', E'}]{380}1r
\efig}
$$
with $b$ and $\tilde b$ from \equref{left C balance} and \equref{right C balance}, respectively. It commutes indeed, since:
$$
\scalebox{0.9}[0.9]{
\gbeg{6}{5}
\got{2}{D} \got{1}{\tau(X)} \got{1}{E}  \got{1}{D'}  \got{1}{E'} \gnl
\gcn{1}{1}{2}{3} \gvac{1} \gbr \gbr \gnl
\gvac{1} \gbr \gmp{} \gcl{1} \gcl{1} \gnl
\gcn{1}{1}{3}{2} \gcn{2}{1}{3}{2} \gbr \gcn{1}{1}{1}{2} \gnl
\gob{2}{E} \gob{1}{\hspace{-0,4cm}D} \gob{1}{\hspace{-0,34cm}E'} \gob{1}{\sigma(X)} \gob{2}{D'}
\gend}=
\scalebox{0.9}[0.9]{
\gbeg{6}{5}
\got{2}{D} \got{1}{\tau(X)} \got{1}{E}  \got{1}{D'}  \got{1}{E'} \gnl
\gcn{1}{1}{2}{3} \gvac{1} \gbr \gcl{1} \gcl{1} \gnl
\gvac{1} \gbr \gmp{} \gbr \gnl
\gcn{1}{1}{3}{2} \gcn{2}{1}{3}{2} \gbr \gcn{1}{1}{1}{2} \gnl
\gob{2}{E} \gob{1}{\hspace{-0,4cm}D} \gob{1}{\hspace{-0,34cm}E'} \gob{1}{\sigma(X)} \gob{2}{D'}
\gend}
$$
by a braiding axiom.
The braiding axioms for $\Psi$ follow by those for $\tilde\Phi$ in $\D$.
\qed\end{proof}

\begin{rem}
Similar structures to the ones presented in the proof of \leref{braided: product over braided} were considered in \cite{DGNO1} and \cite{Gr1}.
Recall that given a braided monoidal category $\D$ the M\"uger's center category $\D'$ is a braided subcategory of $\D$ such that for all $M\in\D'$ and
all $D\in\D$ the braiding $\tilde\Phi$ of $\D$ is symmetric when acting between $M$ and $D$, that is: $\tilde\Phi_{M,X}=(\tilde\Phi_{M,X})^{-1}$, \cite[Definition 2.9]{M}.
Another term used in the literature for such objects (of $\D'$) is that they are {\em transparent}, (see \cite{Al1}). 
In \leref{braided: product over braided} we do not require that the functors $\sigma$ and $\tau$ factor through the M\"uger's center category $\D'$, that is that
$\tilde\Phi$ be symmetric when acting between $\tau(X)$, for every $X\in\C$, and any $D\in\D$ (and the same for $\sigma$).
\end{rem}

\bigskip

\section{Some equivalences of categories of functors} \selabel{adjunctions}

The next result can be seen as a special case of \cite[Proposition 4.5]{Femic2}. Suppose a finite tensor category $\E$ is a left $\C$-module category
via a monoidal functor $\eta: \C\to\E$, as in \equref{associator}. 
Then $-\Del_{\C} \E$ has a right adjoint $\Fun(\E, -)_{\E}\simeq\Id: \Mod\x\E\to\Mod\x\C$. Given a right $\E$-module category $\A$, the category
$\Fun(\E, \A)_{\E}\simeq\A$ is a right $\C$-module category via the left $\C$-module structure of $\E$ (by $(F\crta\ot X)(E)=F(X\crta\ot E)$ for
$F\in\Fun(\E, \A)_{\E}$ and $X\in\C, E\in\E$), that is, via $\eta: \C\to\E$. Otherwise stated, $\A\in\Mod\x\E$ is a right $\C$-module category via
$A\crta\ot C=A\crta\ot\eta(C)$, for $A\in\A, C\in\C$. We denote this $\C$-module by $\G(\A)$. This right adjoint functor $\Mod\x\E\to\Mod\x\C$ we will call the
{\em restriction of scalars functor}.

\medskip

By the notation $\F :\C\to\D:\G$ we will mean that a functor $\F$ is left adjoint to $\G$.

\begin{cor} \colabel{adj}
Given a functor of tensor categories $\eta: \C\to\E$, there is an adjoint pair of functors $\F=-\Del_{\C}\E: \Mod\x\C\to\Mod\x\E :\G$,
where $\F$ is the {\em induction functor} and $\G$ the {\em restriction of scalars functor}, so that there is an
isomorphism of $\F\x\D$-bimodule categories:
$$\Fun(\N\Del_{\C}\E, \A)_{\E}\iso\Fun(\N, \G(\A))_{\C}$$
for every $\N\in\D\x\C\x\Bimod$ and $\A\in\F\x\E\x\Bimod$. 

Given a right $\C$-module functor $\HH: \N\to\G(\A)$, a right $\E$-module functor $\tilde\HH: \N\Del_{\C}\E\to \A$ is given by
\begin{equation}\eqlabel{tilde H}
\tilde\HH(N\Del_{\C} E)=\HH(N)\crta\ot E=\HH(N)\Del_{\C} E
\end{equation}
for every $N\in\N, E\in\E$.
\end{cor}

\begin{defnlma} \cite[Definition and Lemma 5.2]{Femic2}
Let $\F: \D\to\C$ be a tensor functor.
\begin{itemize}
\item
Given another tensor functor $\G: \D'\to\C$ and a $\C$-bimodule category $\M$, the category ${}_{\F}\M_{\G}$ equal to $\M$ as an abelian category with actions:
$$X\crta\ot M\crta\ot X'=\F(X)\crta\ot M\crta\ot \G(X')$$
for all $X\in\D, X'\in\D'$ and $M\in\M$ is a $\D\x\D'$-bimodule category.
\item
If $\HH: \C\to\E$ is another tensor functor and $\M$ is a left $\E$-module category, then there is an obvious equivalence of left $\D$-modules categories:
\begin{equation}\eqlabel{functor change}
{}_{\F}\C\Del_{\C}{}_{\HH}\M\simeq {}_{\HH\F}\M.
\end{equation}
\end{itemize}
\end{defnlma}

\medskip

Let $e^n_i: \C^{\Del n}\to\C^{\Del (n+1)}$ with $i=1,\cdots, n+1$, denote the tensor functors that we introduced in \cite[Section 5]{Femic2} and which are given by:
\begin{equation} \eqlabel{e's}
e_i^n(X^1\Del\cdots\Del X^n) = X^1\Del\cdots \Del I\Del X^i\Del\cdots\Del X^n
\end{equation}

Let us now consider $\eta=e^2_i: \C\Del\C\to \C\Del\C\Del\C, i=1,2,3$ in \coref{adj} and let us write the corresponding adjoint pair of functors as $(\F_i, \G_i)$.
Denote by $\M_i=\F_i(\M)=\M\Del_{\C\Del\C} {}_{e^2_i}(\C\Del\C\Del\C)$ for $i=1,2,3$. In particular, an object in $\M_3$ is of the form:
$$M\Del_{\C\Del\C} {}_{e^2_3}(X^1\Del X^2\Del X^3)= M\crta\ot (X^1\Del X^2) \Del_{\C\Del\C} {}_{e^2_3}(I\Del I\Del X^3)\iso
  (M\crta\ot (X^1\Del X^2)) \Del X^3$$
for some $M\in\M, X^1, X^2, X^3\in\C$. Therefore, $\M_3\simeq \M\Del\C$. Similarly, an object in $\M_1$ is of the form: $X\Del M$ for some $X\in\C$,
and $\M_1\simeq\C\Del\M$.

We will also denote: $M_i=(\M\Del_{\C\Del\C} {}_{e^2_i})(M)$ for some $M\in\M$ and $i=1,2,3$. In particular,
$$M_3=M\Del_{\C\Del\C} {}_{e^2_3}(I\Del I\Del I)=M\Del I \quad\textnormal{and}\quad
M_1=M\Del_{\C\Del\C} {}_{e^2_1}(I\Del I\Del I)=I\Del M.$$

\medskip

To simplify the calculations in the following lemmas we introduce the Sweedler-type notation. For a functor $H: \M\to \M\Del_{\C}\M$ and $M\in\M$ we will write:
$$H(M)=M_{(1)}\Del_{\C} M_{(2)},$$
where the direct sums are understood implicitly. 

\begin{lma} \lelabel{Delta tilde}
Let $\C$ be a symmetric finite tensor category and let $\M\in\C\Del\C\x\Mod$. There is 
an isomorphism of categories
$$\Fun_{\C}(\M, \M\Del_{\C}\M)_{\C}\iso\Fun_{\C^{\Del 3}}(\M_2, \M_3\Del_{\C^{\Del 3}}\M_1).$$
The functor $\tilde{H}$ corresponding to the functor $H\in\Fun_{\C}(\M, \M\Del_{\C}\M)_{\C}$, with
$H(M)= M_{(1)}\Del_{\C} M_{(2)}$ is given by the formula:
\begin{equation}\eqlabel{3H}
\tilde{H}(M_2)=M_{(1)3}\Del_{\C^{\ot 3}} M_{(2)1}.
\end{equation}
\end{lma}

\begin{proof}
Observe that the property \equref{summ one sided bimod} applies.
First we will show the isomorphism: $\G_2(\M_3\Del_{\C^{\Del 3}}\M_1)\iso \M\Del_{\C}\M$. We consider $\M$ as a $\C$-bimodule category, with its
left and right $\C\Del\C$-module structures as in \cite[Proposition 3.1]{Femic2}. Then $\M_i$ are (one-sided) $\C\Del\C\Del\C$-bimodule categories.
Let $A: \M_3\Del\M_1\to \M\Del_{\C}\M$ be defined
by $A((M\Del X)\Del(Y\Del N)):=(Y\crta\ot M)\Del_{\C}(N\crta\ot X)$. Let us show that it induces a functor $\crta A: \M_3\Del_{\C^{\Del 3}}\M_1\to \M\Del_{\C}\M$.
For this, we need to check if $A$ is $\C^{\Del 3}$-balanced, that is, that there is a natural isomorphism between
$$\Sigma:=A\left(((M\Del X)\crta\ot(Z^1\Del Z^2\Del Z^3))\Del(Y\Del N)\right)$$
and
$$\Omega:=A\left((M\Del X)\Del((Z^1\Del Z^2\Del Z^3)\crta\ot(Y\Del N))\right)$$
satisfying condition \equref{C-balanced}.
We first note that $(M\Del X)\crta\ot(Z^1\Del Z^2\Del Z^3)\iso (M\crta\ot(Z^1\Del Z^2))\Del(X\ot Z^3)\iso (Z^1\crta\ot M\crta\ot Z^2)\Del(X\ot Z^3)$
and $(Z^1\Del Z^2\Del Z^3)\crta\ot(Y\Del N)\iso (Z^1\ot Y)\Del((Z^2\Del Z^3)\crta\ot N)\iso (Z^1\ot Y)\Del(Z^2\crta\ot N\crta\ot Z^3)$.
Now we have:
\begin{eqnarray*}
\Sigma&=&A((Z^1\crta\ot M\crta\ot Z^2)\Del(X\ot Z^3)) \Del (Y\Del N)) \\
&=&(Y\crta\ot (Z^1\crta\ot M\crta\ot Z^2)) \Del_{\C} (N\crta\ot (X\ot Z^3)) \\
&\iso& ((Y\ot Z^1)\crta\ot M) \Del_{\C} ((Z^2\crta\ot N)\crta\ot (X\ot Z^3))
\end{eqnarray*}
and
\begin{eqnarray*}
\Omega&=&A((M\Del X) \Del ((Z^1\ot Y)\Del(Z^2\crta\ot N\crta\ot Z^3))) \\
&=& ((Z^1\ot Y)\crta\ot M) \Del_{\C}  ((Z^2\crta\ot N\crta\ot Z^3)\crta\ot X)\\
&\iso& ((Z^1\ot Y)\crta\ot M) \Del_{\C}  (Z^2\crta\ot N\crta\ot (Z^3 \ot X)).
\end{eqnarray*}
We define the wanted isomorphism $b_{M\Del X,Z,Y\Del N}$ between $\Sigma$ and $\Omega$ as:
$$\scalebox{0.84}{
\bfig 
\putmorphism(-100,400)(1,0)[A(((M\Del X)\crta\ot Z)\Del(Y\Del N))	`A((M\Del X)\Del(Z\crta\ot(Y\Del N))) ` b_{M\Del X,Z,Y\Del N}]{2670}1a
\putmorphism(0,0)(1,0)[(Y\crta\ot (Z^1\crta\ot M\crta\ot Z^2)) \Del_{\C} (N\crta\ot (X\ot Z^3)) `((Z^1\ot Y)\crta\ot M) \Del_{\C}  ((Z^2\crta\ot N\crta\ot Z^3)\crta\ot X) `
   \Phi_{Y,Z^1}^{-1}\Del\Id\Del\Phi_{X,Z^3}]{2600}1b
\putmorphism(60,400)(0,1)[`` =]{400}1l 
\putmorphism(2510,400)(0,1)[`` =]{400}1r 
\efig}
$$
where $Z=Z^1\Del Z^2\Del Z^3$. In the coherence condition \equref{C-balanced} the left and right associator functors that are to apply are $m^L$ and $m^R$
defined in \cite[Proposition 3.1]{Femic2}.
Apart from the compatibilities of the left, right and bimodule associators, the coherence condition comes down to:
$$
\gbeg{7}{6}
\got{1}{Y} \got{1}{Z^1} \got{2}{W^1} \got{1}{X} \got{1}{Z^3} \got{2}{W^3} \gnl
\gcl{1} \gbr \gvac{1} \gcl{1} \gbr \gnl
\gibr \gcl{1} \gvac{1} \gbr \gcl{1} \gnl
\gcl{1} \gibr \gvac{1} \gcl{2} \gbr \gnl
\gibr \gcl{1} \gvac{2} \gcl{1} \gcl{1} \gnl
\gob{1}{Z^1} \gob{1}{W^1} \gob{1}{Y} \gob{1}{} \gob{1}{W^3} \gob{1}{Z^3} \gob{1}{X}
\gend=
\gbeg{6}{5}
\got{1}{Y} \got{1}{Z^1} \got{1}{W^1} \gvac{1} \got{1}{X} \got{1}{Z^3} \got{2}{W^3} \gnl
\gibr \gcl{1} \gvac{1} \gbr \gcl{1} \gnl
\gcl{1} \gibr \gvac{1} \gcl{1} \gbr \gnl
\gcl{1} \gcl{1} \gcl{1} \gvac{1} \gbr \gcl{1} \gnl
\gob{1}{Z^1} \gob{1}{W^1} \gob{1}{Y} \gob{1}{} \gob{1}{W^3} \gob{1}{Z^3} \gob{1}{X}
\gend
$$
where $W=W^1\Del W^2\Del W^3$ is another object from $\C^{\Del 3}$. This identity is fulfilled by naturality.

The functor $B: \M\Del\M\to \M_3\Del_{\C^{\Del 3}}\M_1$ given by 
$B(M\Del N)=M_3\Del_{\C^{\Del 3}}N_1$ induces
the functor $\crta B: \M\Del_{\C}\M\to \M_3\Del_{\C^{\Del 3}}\M_1$. To see that $B$ is $\C$-balanced, observe that the identity $B((M\crta\ot X)\Del N)=
M_3\crta\ot (I\Del X\Del I)\Del_{\C^{\Del 3}}N_1 = M_3\Del_{\C^{\Del 3}} (I\Del X\Del I)\crta\ot N_1
=B(M\ot (X\crta\ot N))$ satisfies condition \equref{C-balanced} as the identity functor on $\M_3\Del_{\C^{\Del 3}}\N_1$ is $\C^{\Del 3}$-balanced.
Now, we easily see that $\crta A((M\Del X)\Del_{\C^{\Del 3}}(Y\Del N))=(Y\crta\ot M)\Del_{\C}(N\crta\ot X)$ and $\crta B(M\Del_{\C} N)=(M\Del I)\Del_{\C^{\Del 3}}(I\Del N)$
are inverse to each other:
\begin{eqnarray*}
\crta B\crta A((M\Del X)\Del_{\C^{\Del 3}}(Y\Del N))&=&\crta B((Y\crta\ot M)\Del_{\C}(N\crta\ot X))\\
&=&((Y\crta\ot M)\Del I)\Del_{\C^{\Del 3}}(I\Del (N\crta\ot X))\\
&\iso& (M\Del X)\Del_{\C^{\Del 3}}(Y\Del N)
\end{eqnarray*}
and
\begin{eqnarray*}
\crta A\crta B((M\Del_{\C} N)&=&\crta A((M\Del I)\Del_{\C^{\Del 3}}(I\Del N))\\
&=& (I\crta\ot M)\Del_{\C}(N\crta\ot I)=M\Del_{\C}N.
\end{eqnarray*}

On morphisms we define $\crta A$ as follows. Let $f: (M\Del X)\Del_{\C^{\Del 3}}(Y\Del N)\to (M'\Del X')\Del_{\C^{\Del 3}}(Y'\Del N')$ be a morphism in
$\M_3\Del_{\C^{\Del 3}}\M_1$.
Then $\crta A(f)=\alpha\comp f\comp \beta: (Y\crta\ot M)\Del_{\C}(N\crta\ot X)\to (Y'\crta\ot M')\Del_{\C}(N'\crta\ot X')$ -
here $\alpha: (M'\Del X')\Del_{\C^{\Del 3}}(Y'\Del N')\to(Y'\crta\ot M')\Del_{\C}(N'\crta\ot X')$ is given by
$\alpha((m'\Del x')\Del_{\C^{\Del 3}}(y'\Del n'))=(y'\crta\ot m')\Del_{\C}(n'\crta\ot x')$ and $\beta: (Y\crta\ot M)\Del_{\C}(N\crta\ot X)\to
(M\Del X)\Del_{\C^{\Del 3}}(Y\Del N)$ by $\beta((y\crta\ot m)\Del_{\C}(n\crta\ot x))=(m\Del x)\Del_{\C^{\Del 3}}(y\Del n)$, with $x\in X, x'\in X',
y\in Y, y'\in Y', m\in M, m'\in M', n\in N, n'\in N'$.

On the other hand, $\crta B$ is defined on morphisms as follows. For $f:M\Del_{\C}N\to M'\Del_{\C}N'$ a morphism in $\M\Del_{\C}\M$ we define
$\crta B(f)$ as the composition: 
$$
\scalebox{0.84}{
\bfig 
\putmorphism(50,400)(1,0)[M_3\Del_{\C^{\Del 3}}N_1 ` M'_3\Del_{\C^{\Del 3}}N'_1` \crta B(f)]{3340}1a
\putmorphism(0,0)(1,0)[(I\crta\ot M)\Del_{\C}(N\crta\ot I)` M\Del_{\C}N`l_M\Del_{\C}r_N]{1200}1b
\putmorphism(1050,0)(1,0)[\phantom{(X \ot (Y \ot U))}`M'\Del_{\C}N'` f]{950}1b
\putmorphism(1900,0)(1,0)[\phantom{(X \ot (Y \ot U))}` (I\crta\ot M')\Del_{\C}(N'\crta\ot I)` l_{M'}^{-1}\Del_{\C}r_{N'}^{-1}]{1450}1b
\putmorphism(0,400)(0,1)[``\beta_I^{-1}]{400}1l
\putmorphism(3340,400)(0,1)[``\alpha_I^{-1}]{400}{-1}r
\efig}
$$
where $\alpha_I^{-1}$ and $\beta_I^{-1}$ are obvious maps arising from $\alpha$ and $\beta$ with $X=Y=I$.

\medskip

Then $\M\Del_{\C}\M$ and $\M_3\Del_{\C^{\Del 3}}\M_1$ are isomorphic as abelian categories. Let us see that $\crta A$ is an isomorphism
of $\C$-bimodule categories, where $\M_3\Del_{\C^{\Del 3}}\M_1$ is seen as a $\C$-bimodule category through $\F_2: \Mod\x\C^{\Del 2} \to \Mod\x\C^{\Del 3}$.
For this purpose observe the isomorphism $s_c$ defined via:
$$
\scalebox{0.84}{
\bfig
\putmorphism(150,800)(1,0)[\crta A(C\crta\ot(M\Del X)\Del_{\C^{\Del 3}}(Y\Del N)) ` C\crta\ot\crta A((M\Del X)\Del_{\C^{\Del 3}}(Y\Del N))` s_C]{3000}1a
\putmorphism(150,400)(0,1)[Y\crta\ot(C\crta\ot M)\Del_{\C}(N\crta\ot X)`  `]{400}0l
\putmorphism(200,400)(0,1)[`  `(m^l)^{-1}\Del_{\C}\Id]{400}1l
\putmorphism(80,0)(1,0)[((Y\ot C)\crta\ot M)\Del_{\C}(N\crta\ot X)` ((C\ot Y)\crta\ot M)\Del_{\C}(N\crta\ot X)` \Phi_{Y,C}\Del_{\C}\Id]{3050}1b
\putmorphism(210,800)(0,1)[``=]{400}1l
\putmorphism(3250,800)(0,1)[``=]{400}{-1}r
\putmorphism(3250,400)(0,1)[``m^l\Del_{\C}\Id]{400}{-1}r
\putmorphism(3150,400)(0,1)[(C\crta\ot (Y\crta\ot M))\Del_{\C}(N\crta\ot X)``]{400}0r
\efig}
$$
The check that it satisfies the coherence for a left $\C$-linear functor comes down to: $(C\ot\Phi_{Y,D})(\Phi_{Y,C}\ot D)=\Phi_{Y,C\ot D}$, which is
fulfilled by an axiom for the braiding. (The rest in the coherence check holds by naturality and the coherence of the left associator functor $m^l$.)
The proof that $\crta A$ is right $\C$-linear is similar.
This yields that 
$\G_2(\M_3\Del_{\C^{\Del 3}}\M_1)\iso \M\Del_{\C}\M$ as $\C$-bimodule categories.

\medskip

Finally, the adjunction $\F_2=-\Del_{\C\Del\C}{}_{e^2_2}(\C^{\Del 3}): \Mod\x\C^{\Del 2} \to \Mod\x\C^{\Del 3}:\G_2$ yields that there is an isomorphism:
\begin{eqnarray*}
\Fun_{\C}(\M, \M\Del_{\C}\M)_{\C}
&\iso &\Fun_{\C}(\M, \G_2(\M_3\Del_{\C^{\Del 3}}\M_1))_{\C}\\
&\iso& \Fun_{\C^{\Del 3}}(\F_2(\M), \M_3\Del_{\C^{\Del 3}}\M_1)\\
&=&\Fun_{\C^{\Del 3}}(\M_2, \M_3\Del_{\C^{\Del 3}}\M_1).
\end{eqnarray*}
In particular, given a $\C$-bimodule functor $H: \M\to \M\Del_{\C}\M$ and $M\in\M$, we have by \equref{tilde H} that
$\tilde H(M_2)=\crta B(H(M))=M_{(1)3}\Del_{\C^{\ot 3}} M_{(2)1}$. Since $\tilde H$ is  $\C^{\Del 3}$-linear, it extends to all $\M_2$
(an arbitrary object in $\M_2$ is of the form: $\crta M=M\Del_{\C\Del\C} {}_{e^2_2}(X^1\Del X^2\Del X^3)
=M_2\crta\ot(X^1\Del X^2\Del X^3)$, so $\tilde H(\crta M)=\tilde H(M_2)\crta\ot(X^1\Del X^2\Del X^3)$).
\qed\end{proof}

\begin{lma}
Let $\C$ be a symmetric finite tensor category and let $\M\in\C\Del\C\x\Mod$. There is an isomorphism of categories
$$\Fun_{\C}(\M, \M\Del_{\C}\M\Del_{\C}\M)_{\C}\iso\Fun_{\C^{\Del 4}}(\M_{23}, \M_{34}\Del_{\C^{\Del 4}}\M_{14}\Del_{\C^{\Del 4}}\M_{12}).$$
Given $H\in\Fun_{\C}(\M, \M\Del_{\C}\M\Del_{\C}\M)_{\C}$, with $H(M)= M_{(1)}\Del_{\C} M_{(2)}\Del_{\C} M_{(3)}$, its corresponding functor $\tilde{H}$ is given by:
\begin{equation}\eqlabel{4H}
\tilde{H}(M_{23})=M_{(1)34}\Del_{\C^{\ot 4}} M_{(2)14}\Del_{\C^{\ot 4}} M_{(3)12}.
\end{equation}
\end{lma}

\begin{proof}
As in the proof of \leref{Delta tilde}, one proves that
$\crta A: \M_{34}\Del_{\C^{\Del 4}}\M_{14}\Del_{\C^{\Del 4}}\M_{12}\to \M\Del_{\C}\M\Del_{\C}\M$ given by
\begin{eqnarray*}
\crta A((M\Del X\Del Y) &\Del_{\C^{\Del 4}}& (X'\Del N\Del Y')\Del_{\C^{\Del 4}} (X''\Del Y''\Del P))=\\
&&((X''\ot X')\crta\ot M)\Del_{\C}(Y''\crta\ot N\crta\ot X)\Del_{\C} (P\crta\ot(Y\ot Y'))
\end{eqnarray*}
and
$$\crta B: \M\Del_{\C}\M\Del_{\C}\M\to \M_{34}\Del_{\C^{\Del 4}}\M_{14}\Del_{\C^{\Del 4}}\M_{12}$$
given by
$$\crta B(M\Del_{\C}N\Del_{\C}P):=M_{34}\Del_{\C^{\Del 4}}N_{14}\Del_{\C^{\Del 4}}P_{12}$$
are well-defined $\C$-bilinear functors inverse to each other,
where $\M_{34}\Del_{\C^{\Del 4}}\M_{14}\Del_{\C^{\Del 4}}\M_{12}$ is a $\C$-bimodule category through $\F_{23}:
\C\Del\C\to\C\Del\C\Del\C\Del\C$. Observe that
$\F_{23}=-\Del_{\C\Del\C}{}_{e^2_2}(\C^{\Del 3})\Del_{\C^{\Del 3}}{}_{e^3_3}(\C^{\Del 4})=
-\Del_{\C\Del\C} {}_{(e^2_2\comp e^3_3)}(\C^{\Del 4}): \Mod\x\C^{\Del 2} \to \Mod\x\C^{\Del 4}$. Then
$\M\Del_{\C}\M\Del_{\C}\M
\iso\G_{23}(\M_{34}\Del_{\C^{\Del 4}}\M_{14}\Del_{\C^{\Del 4}}\M_{12})$  as $\C$-bimodule categories and by the adjunction
$(\F_{23}, \G_{23})$ we have the isomorphism:
\begin{eqnarray*}
\Fun_{\C}(\M, \M\Del_{\C}\M\Del_{\C}\M)_{\C}&\iso & \Fun_{\C}(\M, \G_{23}(\M_{34}\Del_{\C^{\Del 4}}\M_{14}\Del_{\C^{\Del 4}}\M_{12}))_{\C}\\
&\iso&\Fun_{\C^{\Del 4}}(\F_{23}(\M), \M_{34}\Del_{\C^{\Del 4}}\M_{14}\Del_{\C^{\Del 4}}\M_{12})\\
&=&\Fun_{\C^{\Del 4}}(\M_{23}, \M_{34}\Del_{\C^{\Del 4}}\M_{14}\Del_{\C^{\Del 4}}\M_{12}).
\end{eqnarray*}

By \equref{tilde H}, if $H\in\Fun_{\C}(\M, \M\Del_{\C}\M\Del_{\C}\M)_{\C}$, its corresponding functor $\tilde H$ in
$\Fun_{\C^{\Del 4}}(\M_{23},\\ \M_{34}\Del_{\C^{\Del 4}}\M_{14}\Del_{\C^{\Del 4}}\M_{12})$ is given by $\tilde H(M_{23})=\crta B(H(M))$
which extends as a $\C^{\Del 4}$-module functor to all $\M_{23}$ and we get \equref{4H}.
\qed\end{proof}

\section{Comonoidal categories and coring categories over finite tensor categories}

Let $k$ be an algebraically closed field and let $\Cc_k$ denote the 2-category of
(small) finite abelian categories. We will denote by $vec$ 
the category of finite-dimensional $k$-vector spaces. 
Then $(\Cc_k, \Del, vec)$ is a symmetric monoidal 2-category (
\cite[Lemma 2.6]{Neu}, \cite[Proposition 2.9.10, Ejercicio 2.9.8]{Mnotas}). We also have that given a finite tensor category $\C$ the 2-category of
$\C$-bimodule categories $(\C\x\Bimod, \Del_{\C}, \C)$ is a monoidal 2-category, \cite[Theorem 1.1]{Gr}. We will consider their truncations to usual monoidal categories.

Comonoidal categories were introduced in \cite[Chapter 3]{Neu}. 
A finite abelian category $\C$ is comonoidal if
there are $k$-linear functors $\Delta: \C\to\C\Del\C$ and $\varepsilon: \C\to vec$ with natural isomorphisms:
$$a:(\Id_{\C}\Del\Delta)\Delta\to(\Delta\Del\Id_{\C})\Delta$$
$$l:(\varepsilon\Del\Id_{\C})\Delta\to\Id_{\C}$$
$$r:(\Id_{\C}\Del\varepsilon)\Delta\to\Id_{\C}$$
which satisfy 2 coherence diagrams:
$$\scalebox{0.84}{
\bfig 
\putmorphism(-200,500)(1,0)[(\C\Del\C\Del\Delta)(\C\Del\Delta)\Delta`
(\C\Del\Delta\Del\C)(\C\Del\Delta)\Delta`(\C\Del a)\Delta]{1500}1a
\putmorphism(1200,500)(1,0)[\phantom{(X \ot (Y \ot U)) \ot W\Delta\Del\C}`
   (\C\Del\Delta\Del\C)(\Delta\Del\C)\Delta`
   (\C\Del\Delta\Del\C)a]{1700}1a
\putmorphism(-160,500)(0,-1)[``(\C\Del\C\Del\Delta)a]{500}1r
\putmorphism(2950,500)(0,-1)[``(a\Del\C)\Delta]{500}1l
\putmorphism(-200,0)(1,0)[(\Delta\Del\Delta)\Delta` (\Delta\Del\C\Del\C)(\Delta\Del\C)\Delta`
   (\Delta\Del\C\Del\C)a]{3100}1b
\efig}
$$
and
$$\scalebox{0.84}{
\bfig
\putmorphism(-80,500)(1,0)[(\C\Del(\varepsilon\Del\C))(\C\Del\Delta)\Delta` ((\C\Del\varepsilon)\Del\C)(\Delta\Del\C)\Delta`
   (\C\Del\varepsilon\Del\C)a]{1680}1a
\putmorphism(350,500)(1,-1)[`\Delta`(\C\Del l)\Delta]{430}1l
\putmorphism(1200,500)(-1,-1)[``(r\Del\C)\Delta]{430}1r
\efig} \vspace{-0,2cm}
$$

In other words, a comonoidal category is a coalgebra object in the monoidal 2-category $(\Cc_k, \Del, vec)$.

\medskip

The coherence theorem for comonoidal categories holds, see \cite[Theorem 3.1]{CY}.
We generalize the above definition:

\begin{defn}
Let $\C$ be a finite tensor category and $\A$ an abelian finite category. We say that $\A$ is a {\em comonoidal $\C$-category}
if there are $\C$-bimodule functors $\Delta: \A\to\A\Del_{\C}\A$ and $\varepsilon: \A\to\C$ with $\C$-bimodule natural isomorphisms:
$$a:(\Id_{\A}\Del_{\C}\Delta)\Delta\to(\Delta\Del_{\C}\Id_{\A})\Delta$$
$$l:(\varepsilon\Del_{\C}\Id_{\A})\Delta\to\Id_{\A}$$
$$r:(\Id_{\A}\Del_{\C}\varepsilon)\Delta\to\Id_{\A}$$
which satisfy 2 coherence diagrams:
$$\scalebox{0.78}{
\bfig 
\putmorphism(-200,500)(1,0)[(\A\Del_{\C}\A\Del_{\C}\Delta)(\A\Del_{\C}\Delta)\Delta`
    (\A\Del_{\C}\Delta\Del_{\C}\A)(\A\Del_{\C}\Delta)\Delta`(\A\Del_{\C} a)\Delta]{1700}1a
\putmorphism(1460,500)(1,0)[\phantom{(X \ot (Y \ot U)) \ot W\Delta\Del_{\C}\A}`
   (\A\Del_{\C}\Delta\Del_{\C}\A)(\Delta\Del_{\C}\A)\Delta`
   (\A\Del_{\C}\Delta\Del_{\C}\A)a]{2000}1a
\putmorphism(-160,500)(0,-1)[``(\A\Del_{\C}\A\Del_{\C}\Delta)a]{500}1r
\putmorphism(3500,500)(0,-1)[``(a\Del_{\C}\A)\Delta]{500}1l
\putmorphism(-200,0)(1,0)[(\Delta\Del_{\C}\Delta)\Delta` (\Delta\Del_{\C}\A\Del_{\C}\A)(\Delta\Del_{\C}\A)\Delta`
   (\Delta\Del_{\C}\A\Del_{\C}\A)a]{3700}1b
\efig}
$$
and
$$\scalebox{0.8}{
\bfig
\putmorphism(-80,500)(1,0)[(\A\Del_{\C}(\varepsilon\Del_{\C}\A))(\A\Del_{\C}\Delta)\Delta` ((\A\Del_{\C}\varepsilon)\Del_{\C}\A)(\Delta\Del_{\C}\A)\Delta`
   (\A\Del_{\C}\varepsilon\Del_{\C}\A)a]{2000}1a
\putmorphism(280,350)(3,-1)[`\Delta`]{730}0l
\putmorphism(-100,500)(3,-1)[``]{1230}1l
\putmorphism(-100,450)(3,-1)[``(\A\Del_{\C} l)\Delta]{1230}0l
\putmorphism(2100,500)(-3,-1)[``]{1230}1r
\putmorphism(2100,450)(-3,-1)[``(r\Del_{\C}\A)\Delta]{1230}0r
\efig} \vspace{-0,2cm}
$$
\end{defn}

Otherwise stated, a comonoidal $\C$-category is a coalgebra object in the monoidal 2-category $(\C\x\Bimod, \Del_{\C}, \C)$.

\medskip

\begin{ex} \exlabel{can}
Let $\C$ be a finite tensor category, take an invertible $\C$-bimodule category $\N$ and consider $\A=\Can(\N; \C)_{\C}={}^{op}\N\Del\N$ with the
$\C$-bimodule functors $\Delta: {}^{op}\N\Del\N \to({}^{op}\N\Del\N)\Del_{\C}({}^{op}\N\Del\N)\simeq {}^{op}\N\Del\C\Del\N$
and $\varepsilon:{}^{op}\N\Del\N\to\C$ given on objects as follows. Let $\Delta(M\Del N)=M\Del\crta\coev(I)\Del N \cong M\Del I\Del N$ and let $\varepsilon$
be the composition:
$$ \bfig \putmorphism(0, 0)(1, 0)[{}^{op}\N\Del\N`{}^{op}\N\Del_{\C}\N`\pi]{750}1a
\putmorphism(1000, 0)(1, 0)[` \C`\crta\ev]{300}1a
\efig
$$
which is given by $\varepsilon(M\Del N)=\o\Hom_{\N}(N,M)$. On a morphism $f: M\Del N\to M'\Del N'$ in $\M\Del\N$ we define
$\Delta(f)=f_2: M\Del I\Del N\to M'\Del I\Del M'$. We lack of the definition of $\varepsilon$ on morphisms: $\varepsilon(f):
\o\Hom_{\N}(N,M)\to\o\Hom_{\N}(N',M')$). Let us see that
(up to the lack of the definition of the functor $\varepsilon$ on morphisms) $\Can(\N; \C)_{\C}$ is a comonoidal $\C$-category.
For the coassociativity of $\Delta$ the associativity isomorphism
$a:(\Id_{\A}\Del_{\C}\Delta)\Delta\to(\Delta\Del_{\C}\Id_{\A})\Delta$ for every $M\Del N\in {}^{op}\N\Del\N$ is a morphism in
$({}^{op}\N\Del\N)^{\Del_{\C} 3}$:
$$a(M\Del N): \displaystyle\bigoplus_{\substack{j\in J\\i\in J}} (M\Del  V_i)\Del_{\C}((W_i\Del V_j)\Del_{\C}(W_j\Del N)) \to
\displaystyle\bigoplus_{\substack{i\in J \\ j\in J}} ((M\Del  V_i)\Del_{\C}(W_i\Del V_j))\Del_{\C}(W_j\Del N)$$
induced by the canonical associativity equivalence, which clearly satisfies the pentagon axiom. The natural isomorphism
$r:(\Id_{\A}\Del_{\C}\varepsilon)\Delta\to\Id_{\A}$ is given by the identity functor, 
because of \equref{crta coev eq}. In other words, $r$ is identity because of one axiom for the dual object:
$(\N\Del_{\C} \crta\ev)(\crta\coev\Del_{\C} \N)\iso\Id_{\N}$. On the other hand, the natural isomorphism
$l:(\varepsilon\Del_{\C}\Id_{\A})\Delta\to\Id_{\A}$ is identity because of the other dual object axiom:
$(\crta\ev\Del_{\C} {}^{op}\N)({}^{op}\N\Del_{\C} \crta\coev)\iso\Id_{{}^{op}\N}$, as we see here:
$$\scalebox{0.84}{
\bfig \hspace{-0,5cm}
\putmorphism(-200,900)(1,0)[{}^{op}\N\Del\N` ({}^{op}\N\Del\N)\Del_{\C}({}^{op}\N\Del\N)`\Delta]{1400}1a
\putmorphism(1200,900)(1,0)[\phantom{(X \ot (Y \ot U)) \ot W\Delta\Del\C}`
   \C\Del_{\C} {}^{op}\N\Del\N` \varepsilon\Del_{\C} {}^{op}\N\Del\N]{1700}1a
\putmorphism(-160,900)(0,-1)[``\simeq]{900}1l
\putmorphism(1400,900)(0,-1)[`({}^{op}\N\Del_{\C}\N)\Del_{\C}({}^{op}\N\Del\N)`\pi\Del_{\C} {}^{op}\N\Del\N]{450}1r
\putmorphism(2950,900)(0,-1)[``\simeq]{900}1r
\putmorphism(200,150)(3,1)[ ` ` ]{400}1l
\putmorphism(585,290)(3,1)[ ` ` ({}^{op}\N\Del_{\C}\crta\coev)\Del\N ]{400}0l
\putmorphism(1700,430)(3,-1)[`` ]{1300}1r
\putmorphism(1340,570)(3,-1)[`` (\crta\ev\Del_{\C} {}^{op}\N)\Del\N ]{1300}0r
\putmorphism(-200,0)(1,0)[({}^{op}\N\Del_{\C}\C)\Del\N ` {}^{op}\N\Del\N ` \simeq\Del\N]{3100}1b
\put(350,610){\fbox{1}}
\put(1350,160){\fbox{3}}
\put(2300,610){\fbox{2}}
\efig}
$$
Here the inner diagrams $\langle 1\rangle$ and $\langle 2\rangle$ commute by the definitions of $\Delta$ and $\varepsilon$, respectively,
and the triangle $\langle 3\rangle$ commutes by the latter axiom for the dual object ${}^{op}\N$. The compatibility condition for $a, l$ and
$r$ is easily proved, as well as that they all are $\C$-bimodule natural isomorphisms.
\medskip

Similarly, there is a structure of an (almost) comonoidal $\C$-category on $\Can{}_{\C}(\N; \C)=\N\Del\N^{op}$, induced by the functors $\coev$ and $\ev$
(up to the lack of the definition of the counit functor on morphisms).
\end{ex}

In the particular case when $\N=\C$ in the above example, we obtain indeed a comonoidal $\C$-category (with properly defined counit functor).

\begin{lma} \lelabel{CC coring}
For any finite tensor category $\C$ the category $\A=\C\Del\C$ is a comonoidal $\C$-category with the functors
$\Delta: \C\Del\C \to (\C\Del\C) \Del_{\C} (\C\Del\C) \simeq \C\Del\C\Del\C$ and $\varepsilon:\C\Del\C\to\C$ given by
$$\Delta(X\Del Y)=(X\Del I) \Del_{\C} (I\Del Y)\quad\textnormal{and}\quad \varepsilon(X\Del Y)=X\ot Y$$
on objects and $\Delta(f)=f_2: X\Del I\Del Y\to X'\Del I\Del Y'$ and $\varepsilon(f)$ is the induced morphism between
$X\ot Y\to X'\ot Y'$, for $f: X\Del Y\to X'\Del Y'$.
\end{lma}

\begin{proof}
The tensor product functor of $\C$ is biexact, hence it induces a well-defined functor $\ot:\C\Del\C\to\C$ given by
$X\Del Y\mapsto X\ot Y$. Thus $\varepsilon(f)$ is given by the commuting diagram:
$$\scalebox{0.88}{\bfig
 \putmorphism(0,400)(1,0)[X\ot Y` X'\ot Y'` \varepsilon(f)]{1500}1a
 \putmorphism(0,0)(1,0)[X\Del Y` X'\Del Y'` f]{1500}1a
\putmorphism(0,400)(0,-1)[\phantom{B\ot B}``\ot]{380}{-1}l
\putmorphism(1500,400)(0,-1)[\phantom{B\ot B}``\ot]{380}{-1}r
\efig}
$$
The functors $\Delta$ and $\varepsilon$ are directly proved to be $\C$-bilinear. For the coassociativity we find:
$(\Delta\Del_{\C}\id)\Delta(X\Del Y)=\Delta(X\Del I)\Del_{\C}(I\Del Y)=(X\Del I) \Del_{\C} (I\Del I) \Del_{\C}(I\Del Y)$ and
$(\id\Del_{\C}\Delta)\Delta(X\Del Y)=(X\Del I)\Del_{\C}\Delta(I\Del Y)=(X\Del I) \Del_{\C} (I\Del I) \Del_{\C}(I\Del Y)$,
for $X,Y\in\C$.
Then the natural isomorphism $a$ is the identity. For the compatibility with the counit functor we find:
$(\varepsilon\Del_{\C}\id)\Delta(X\Del Y)=\varepsilon(X\Del I)\Del_{\C}(I\Del Y)=(X\ot I) \Del_{\C} (I\Del Y) \iso X\Del Y$ and
$(\id\Del_{\C}\varepsilon)\Delta(X\Del Y)=(X\Del I)\Del_{\C}\varepsilon(I\Del Y)=(X\Del I) \Del_{\C} (I\ot Y) \iso X\Del Y$.
So we may take the natural isomorphisms $l$ and $r$ to be identities, then all the coherence diagrams are trivially satisfied.
The rest of the proof is direct.
\qed\end{proof}

We will be interested in the following:

\begin{defn}
Let $\C$ be a finite tensor category and $\A$ an abelian finite category. We say that $\A$ is a {\em $\C$-coring category}
if $\A$ is a coalgebra in the monoidal category $(\C\x\Bimod, \Del_{\C}, \C)$.
\end{defn}

That is, if there are $\C$-bimodule functors $\Delta: \A\to\A\Del_{\C}\A$ and $\varepsilon: \A\to\C$ with $\C$-bimodule natural isomorphisms:
$$a:(\Id_{\A}\Del_{\C}\Delta)\Delta\to(\Delta\Del_{\C}\Id_{\A})\Delta$$
$$l:(\varepsilon\Del_{\C}\Id_{\A})\Delta\to\Id_{\A}$$
$$r:(\Id_{\A}\Del_{\C}\varepsilon)\Delta\to\Id_{\A}.$$

\begin{defn}
A functor of two $\C$-coring categories $\A_1$ and $\A_2$ is a $\C$-bimodule functor $F:\A_1\to\A_2$ with two $\C$-bimodule natural isomorphisms
$$\gamma: \Delta_2\comp F\to (F\Del_{\C}F)\Delta_1\quad\textnormal{and}\quad\delta: \varepsilon_2\comp F\to\varepsilon_1.$$
\end{defn}

\begin{ex} \exlabel{canonical cor}
The objects $\Can(\N; \C)_{\C}$ and $\Can_{\C}(\N; \C)$ from \exref{can} are examples of almost $\C$-coring categories (up to the lack of the definition
of the counit functor on morphisms). 
If we consider them without the counit ``functor'' $\varepsilon$, we will call them {\em canonical quasi $\C$-coring categories}. The structure discussed
in \leref{CC coring} we will call a proper {\em canonical $\C$-coring category}.
\end{ex}

\section{Amitsur cohomology over symmetric finite tensor categories} \selabel{Amitsur coh}

Let us recall first the definition of the Amitsur cohomology groups for a symmetric finite tensor category $\C$ that we introduced in \cite[Section 7]{Femic2}.

Let $P$ be an additive covariant functor from a full subcategory of the category of symmetric tensor categories that contains all Deligne tensor
powers $\C^{\Del n}$ of $\C$ to abelian groups. We define $\C^{\Del 0}=k$. Then we consider
$$\delta_n=\sum_{i=1}^{n+1} (-1)^{i-1}P(e^n_i):\ P(\C^{\Del n})\to P(\C^{\Del (n+1)})$$
where $e^n_i: \C^{\Del n}\to\C^{\Del (n+1)}$ for $i=1,\cdots, n+1$ are the augmentation functors from \equref{e's}:
$e_i^n(X^1\Del\cdots\Del X^n) = X^1\Del\cdots \Del I\Del X^i\Del\cdots\Del X^n$. One proves that $\delta_{n+1}\circ \delta_n=0$ is fulfilled, so we obtain a complex:
$$ \bfig \putmorphism(0, 0)(1, 0)[0`P(\C)`]{420}1a
\putmorphism(400, 0)(1, 0)[\phantom{P(S)}`P(\C^{\Del 2})`\delta_1]{560}1a
\putmorphism(960, 0)(1, 0)[\phantom{P(S^{\ot 2})}` P(\C^{\Del 3})`\delta_2]{600}1a
\putmorphism(1550, 0)(1, 0)[\phantom{P(S^{\ot 3})}`\cdots`\delta_3]{550}1a
\efig
$$
which we call {\em Amitsur complex $C(\C/vec, P)$}. We set:
$$Z^n(\C, P)=\Ker\delta_n,~~B^n(\C, P)=\im\delta_{n-1}~~~\textnormal{and}~~~H^n(\C, P)=Z^n(\C, P)/B^n(\C, P)$$
and we call the latter the $n$-th Amitsur cohomology group of $\C$ with values in $P$. Elements in $Z^n(\C,P)$ and $B^n(\C,P)$ are called $n$-cocycles and
$n$-coboundaries, respectively.

We now prove:

\begin{prop} \prlabel{Knus}
Let $F: \C\to\D$ be a functor between two symmetric finite tensor categories. Then $F$ induces group maps $F_*:\ H^n(\C,P)\to H^n(\D,P)$.
If $G: \C\to \D$ is another symmetric tensor functor, then $F_*=G_*$ (for $n\geq 1$).
\end{prop}

\begin{proof}
The functor $F$ induces the functor $F^n: \C^{\Del n}\to\D^{\Del n}$ given by $F^n(X^1\Del \dots \Del X^n)=F(X^1)\Del \dots \Del F(X^n)$.
It is a monoidal functor with the tensor functor structure $\xi_{X, Y}^n: F^n(X\odot Y) \to F^n(X)\odot F^n(Y)$,
for $X=X^1\Del\cdots\Del X^n, Y=X^1\Del\cdots\Del X^n\in\C^{\Del n}$, given by: \vspace{-0,5cm}
\begin{equation*}
\scalebox{0.88}{\bfig
 \putmorphism(-130,400)(1,0)[F^n(X\odot Y)` F^n(X)\odot F^n(Y) `\xi_{X, Y}^n]{2750}1a
 \putmorphism(-50,0)(1,0)[F(X^1Y^1)\Del\cdots\Del F(X^nY^n)` F(X^1)F(Y^1)\Del\cdots\Del F(X^n)F(Y^n)` \xi_{X^1, Y^1}\Del\cdots\Del\xi_{X^n, Y^n}]{2700}1a
\putmorphism(-60,400)(0,-1)[\phantom{B\ot B}``=]{380}1l
\putmorphism(2630,400)(0,-1)[\phantom{B\ot B}``=]{380}1r
\efig}
\end{equation*}
where $\xi_{-, -}: F(-\ot -)\to F(-)\ot F(-)$ is the monoidal structure of $F$.
Set $F'^{(n)}=P(F^n): P(\C^{\Del n})\to P(\D^{\Del n})$ and $F': C(\C,P)\to C(\D,P)$ for the map induced on the Amitsur complexes. Then
$F_*: H^n(\C,P)\to H^n(\D,P)$ is given by $F_*([X])=[F'^{(n)}(X)]$. The proof that $F_*=G_*$
is analogous to that of \cite[Prop. 5.1.7]{KO1}. 
We prove that the induced maps on the Amitsur complexes $F', G': C(\C,P)\to C(\D,P)$ are homotopic. Recall that a homotopy between $F'$ and $G'$ is
a collection of maps $\theta^{(n)}: C^{n+1}(\C,P)\to C^n(\D,P)$ such that $\delta'_{n-1}\comp\theta^{(n-1)}+\theta^{(n)}\comp\delta_n=G'^{(n)}-F'^{(n)}$,
where $\delta_n: P(\C^{\Del n})\to P(\C^{\Del (n+1)})$ and $\delta'_{n-1}: P(\D^{\Del (n-1)})\to P(\D^{\Del n})$.

For every $i\in\{1, \dots, n\}$ we define $\theta_i^n: \C^{\Del (n+1)}\to \D^{\Del n}$ to be the functor given by
$\theta_i^n(X^1\Del \cdots \Del X^{n+1})=F(X^1)\Del \cdots \Del F(X^i)G(X^{i+1})\Del \cdots \Del G(X^{n+1})$. Here the product $F(X^i)G(X^{i+1})$
obviously is the tensor product in $\D$. Let us see that $\theta_i^n$ is a symmetric tensor functor, with the tensor structure:
$c_{X, Y}^{n,i}: \theta_i^n(X\odot Y)\to\theta_i^n(X)\odot\theta_i^n(Y)$ given by:
\begin{equation} \eqlabel{teta tensor str}
\scalebox{0.88}{\bfig
 \putmorphism(-700,800)(1,0)[\theta_i^n(X\odot Y)` \theta_i^n(X)\odot \theta_i^n(Y) `c_{X, Y}^{n,i}]{2700}1a
 \putmorphism(-50,400)(0,-1)[F(X^1Y^1)\Del\cdots\Del F(X^iY^i)G(X^{i+1}Y^{i+1})\Del\cdots\Del G(X^{n+1}Y^{n+1})`
                    F(X^1)F(Y^1)\Del\cdots\Del F(X^i)F(Y^i)G(X^{i+1})G(Y^{i+1})\Del\cdots\Del G(X^{n+1})G(Y^{n+1})` ]{600}0r
 \putmorphism(-600,400)(-1,-1)[`` \xi\Del\cdots\Del(\xi\ot\zeta)\Del\cdots\Del\zeta]{600}1r
\putmorphism(0,-200)(0,-1)[\phantom{F(X^1Y^1)\Del\cdots\Del F(X^iY^i)G(X^{i+1}Y^{i+1})\Del\cdots\Del G(X^{n+1}Y^{n+1})}` ` ]{600}0r
\putmorphism(700,-300)(0,-1)[` F(X^1)F(Y^1)\Del\cdots\Del F(X^i)G(X^{i+1})F(Y^i)G(Y^{i+1})\Del\cdots\Del G(X^{n+1})G(Y^{n+1})` ]{600}0r										
\putmorphism(-1200,-200)(1,-1)[`` \Id\Del\cdots\Del(\Id\ot\cdots\ot\Phi_{F(Y^i),G(X^{i+1})}\ot\cdots\ot\Id)\Del\cdots\Del\Id]{680}1r
\putmorphism(-660,800)(0,-1)[\phantom{B\ot B}``=]{380}1l
\putmorphism(2000,800)(0,-1)[\phantom{B\ot B}``=]{1690}1r
\efig}
\end{equation}
where $\Phi$ is the braiding in $\D$. To prove that $c_{X, Y}^{n,i}$ defines a monoidal functor structure on $\theta_i^n$ one uses the properties of the tensor
structures $\xi$ of $F$ and $\zeta$ of $G$ and the properties of the braiding. The parts where only $F$ (or only $G$) appears will work since $F$ (and $G$)
are braided monoidal functors, that, is, $\xi$ and $\zeta$ are compatible with the braidings of $\C$ and $\D$.
We only show here the part of the check that is affected by the braiding (where $F(-)$ and $G(-)$ interchange their places),
which comes down to check that $\Sigma=\Omega$ in the computation bellow. In this part $\Phi$ acts on the components $i$ and $i+1$ of objects $X,Y,Z\in\C^{\Del n}$.
To simplify the notation, in the computation that follows we will identify: $X=X^i$ and $X'=X^{i+1}$ by abuse, where $X^i$ actually denotes the $i$-th component of
an object $X\in\C^{\Del n}$, and similarly for $Y,Z\in\C^{\Del n}$. We find:
$$\Sigma=
\scalebox{0.9}[0.9]{
\gbeg{11}{7}
\gvac{2} \got{2}{F((XY)Z)} \gvac{3} \got{1}{G((X'Y')Z')} \gnl
\gvac{3} \glmpt \gnot{\xi_{XY,Z}} \gcmpb \grmpb \glmptb \gnot{\hspace{0,12cm}\zeta_{X'Y',Z'}} \gcmpb \grmp \gnl
\gvac{2} \gcn{2}{1}{5}{1} \gvac{1} \gbr \gcn{2}{4}{1}{6} \gnl
\gvac{2} \glmptb \gnot{\hspace{-0,4cm}\xi_{X,Y}} \grmpb \glmpb \gnot{\hspace{-0,4cm}\zeta_{X,Y}} \grmptb \gcl{1} \gnl
\gcn{2}{2}{5}{1} \gvac{1} \gbr \gcn{1}{2}{1}{3} \gcn{2}{2}{1}{4} \gnl
\gvac{2} \gcn{1}{1}{3}{1} \gvac{1} \gcl{1} \gnl
\gob{1}{F(X)} \gob{3}{G(X')} \gob{1}{F(Y)} \gvac{1} \gob{1}{G(Y')} \gvac{1} \gob{1}{F(Z)} \gob{3}{G(Z')}
\gend}\stackrel{nat. \Phi}{=}
\scalebox{0.9}[0.9]{
\gbeg{11}{7}
\gvac{2} \got{2}{F((XY)Z)} \gvac{2} \got{2}{G((X'Y')Z')} \gnl
\gvac{2} \glmpt \gnot{\xi_{XY,Z}} \gcmp \grmpb \glmptb \gnot{\hspace{0,12cm}\zeta_{X'Y',Z'}} \gcmp \grmpb \gnl
\gvac{2} \glmptb \gnot{\hspace{-0,4cm}\xi_{X,Y}} \grmpb \gcl{1} \glmpb \gnot{\hspace{-0,4cm}\zeta_{X,Y}} \grmptb \gcn{2}{4}{1}{6} \gnl
\gvac{2} \gcl{1} \gcl{1} \gbr \gcn{2}{3}{1}{5} \gnl
\gcn{2}{2}{5}{1} \gvac{1} \gbr \gcn{1}{2}{1}{3} \gnl
\gvac{2}  \gcn{2}{1}{3}{1} \gcl{1} \gnl
\gob{1}{F(X)} \gob{3}{G(X')} \gob{1}{F(Y)} \gvac{1} \gob{1}{G(Y')} \gvac{1} \gob{1}{F(Z)} \gob{3}{G(Z')}
\gend}
$$

\bigskip

$$\stackrel{\xi, \zeta}{=}
\scalebox{0.9}[0.9]{
\gbeg{11}{7}
\gvac{2} \got{2}{F(X(YZ))} \gvac{2} \got{2}{G(X'(Y'Z'))} \gnl
\gvac{2} \glmpt \gnot{\xi_{X,YZ}} \gcmp \grmpb \glmptb \gnot{\hspace{0,12cm}\zeta_{X',Y'Z'}} \gcmp \grmpb \gnl
\gvac{2} \gcl{1} \glmptb \gnot{\hspace{-0,4cm}\xi_{X,Y}} \grmpb \gcl{1} \glmpb \gnot{\hspace{-0,4cm}\zeta_{X,Y}} \grmptb \gnl
\gvac{2} \gcl{1} \gcl{1} \gbr \gcl{1} \gcn{2}{3}{1}{6} \gnl
\gcn{2}{2}{5}{1} \gvac{1} \gbr \gbr \gnl
\gvac{2}  \gcn{2}{1}{3}{1} \gcl{1} \gcl{1} \gcn{1}{1}{1}{5} \gnl
\gob{1}{F(X)} \gob{3}{G(X')} \gob{1}{F(Y)} \gvac{1} \gob{1}{G(Y')} \gvac{1} \gob{1}{F(Z)} \gob{3}{G(Z')}
\gend}\stackrel{nat. \Phi}{=}
\scalebox{0.9}[0.9]{
\gbeg{11}{7}
\gvac{1} \got{2}{F(X(YZ))} \gvac{2} \got{2}{G(X'(Y'Z'))} \gnl
\gvac{2} \glmpt \gnot{\xi_{XY,Z}} \gcmpb \grmpb \glmptb \gnot{\hspace{0,12cm}\zeta_{X'Y',Z'}} \gcmp \grmpb \gnl
\gvac{1} \gcn{2}{2}{5}{3} \gvac{1} \gbr \gvac{1} \gcn{2}{1}{1}{3} \gnl
\gvac{3} \gcn{2}{2}{3}{1} \glmptb \gnot{\hspace{-0,4cm}\xi_{X,Y}} \grmpb \glmpb \gnot{\hspace{-0,4cm}\zeta_{X,Y}} \grmptb \gnl
\gcn{2}{2}{5}{1} \gvac{2} \gcn{2}{2}{3}{1} \gbr \gcn{1}{2}{1}{4} \gnl
\gvac{2} \gcn{1}{1}{3}{1} \gvac{3} \gcl{1} \gcn{1}{1}{1}{2} \gnl
\gob{1}{F(X)} \gob{3}{G(X')} \gob{1}{F(Y)} \gvac{1} \gob{1}{G(Y')} \gvac{1} \gob{1}{F(Z)} \gob{3}{G(Z')}
\gend}=\Omega.
$$
To prove that $\theta_i^n$ is a symmetric monoidal functor, we should check if
\begin{equation*}
\scalebox{0.88}{\bfig
 \putmorphism(0,400)(1,0)[\theta_i^n(X\odot Y)` \theta_i^n(X)\odot \theta_i^n(Y) `c_{X, Y}^{n,i}]{2200}1a
 \putmorphism(10,0)(1,0)[\theta_i^n(Y\odot X)` \theta_i^n(Y)\odot \theta_i^n(X)` c_{Y, X}^{n,i}]{2200}1a
\putmorphism(60,400)(0,-1)[\phantom{B\ot B}``\theta_i^n(\Phi)]{380}1l
\putmorphism(2200,400)(0,-1)[\phantom{B\ot B}``\Phi_{\theta_i^n(X),\theta_i^n(Y)}]{380}1r
\efig}
\end{equation*}
commutes, where $c_{X, Y}^{n,i}$ is the composition in the left hand-side of the diagram \equref{teta tensor str}. As above, the parts of the Deligne tensor powers in this
computation where only $F$ or $G$ are present work since these are braided monoidal functors, let us see the places in the computation where $F(-)$ and $G(-)$ interchange their places. We find:
$$\scalebox{0.9}[0.9]{
\gbeg{10}{6}
\gvac{1} \got{2}{F(X^i Y^i)} \gvac{2} \got{2}{G(X^{i+1}Y^{i+1})} \gnl
\gvac{1} \glmpt \gnot{F(\Phi_{\C})} \gcmp \grmp \glmptb \gnot{\hspace{0,12cm}G(\Phi_{\C})} \gcmp \grmp \gnl
\gvac{1} \glmpt \gnot{\xi} \gcmp \grmpb \glmptb \gnot{\hspace{0,12cm}\zeta} \gcmp \grmp \gnl
\gcn{1}{2}{4}{1} \gvac{2} \gbr \gvac{1} \gcn{1}{2}{0}{3} \gnl
\gvac{3} \gcn{1}{1}{1}{0} \gcn{1}{1}{1}{3} \gnl
\gob{1}{F(Y^i)} \gvac{1} \gob{2}{G(Y^{i+1})}\gvac{1} \gob{1}{F(X^i)} \gvac{2} \gob{1}{G(X^{i+1})}
\gend}\stackrel{\xi, \zeta \hspace{0,1cm} are \hspace{0,1cm} braided}{=}
\scalebox{0.9}[0.9]{
\gbeg{10}{6}
\gvac{1} \got{2}{F(X^i Y^i)} \gvac{2} \got{2}{G(X^{i+1}Y^{i+1})} \gnl
\gvac{1} \glmpt \gnot{\xi} \gcmp \grmpb \glmptb \gnot{\hspace{0,12cm}\zeta} \gcmp \grmp \gnl
\gvac{2} \gbr \gbr \gnl
\gcn{1}{2}{5}{1} \gvac{2} \gbr \gcn{1}{2}{1}{5} \gnl
\gvac{3} \gcn{1}{1}{1}{0} \gcn{1}{1}{1}{3} \gnl
\gob{1}{F(Y^i)} \gvac{1} \gob{2}{G(Y^{i+1})}\gvac{1} \gob{1}{F(X^i)} \gvac{2} \gob{1}{G(X^{i+1})}
\gend}\stackrel{\Phi \hspace{0,1cm}symm.}{=}
\scalebox{0.9}[0.9]{
\gbeg{10}{7}
\gvac{1} \got{2}{F(X^i Y^i)} \gvac{2} \got{2}{G(X^{i+1}Y^{i+1})} \gnl
\gvac{1} \glmpt \gnot{\xi} \gcmp \grmpb \glmptb \gnot{\hspace{0,12cm}\zeta} \gcmp \grmp \gnl
\gvac{2} \gcl{1} \gbr \gcl{1} \gnl
\gvac{2} \gcl{1} \gbr \gcl{1} \gnl
\gvac{2} \gbr \gbr \gnl
\gcn{2}{1}{5}{1} \gvac{1} \gbr \gcn{1}{1}{1}{5} \gnl
\gob{1}{F(Y^i)} \gvac{1} \gob{2}{G(Y^{i+1})}\gvac{1} \gob{1}{F(X^i)} \gvac{2} \gob{1}{G(X^{i+1}).}
\gend}
$$
Here we used that the category $\D$ (its braiding $\Phi$) is symmetric. Thus $\theta_i^n$ is indeed a symmetric monoidal functor.

Now the desired map $\theta^{(n)}$ is given by $\theta^{(n)}=\displaystyle\sum_{i=1}^n(-1)^{i-1}P(\theta_i^{(n)})$.

Moreover, let $\sigma_i^n: \C^{\Del n}\to \D^{\Del n}$ be the symmetric tensor functors given by
$\sigma_i(X^1\Del \cdots 
\Del X^n)=F(X^1)\Del \cdots \Del F(X^{i-1})\Del G(X^i)\Del \cdots \Del G(X^n)$ for $i=1, \dots, n+1$. Its monoidal structure is similar to that of
the functor $F^n$. Then we obtain the relations:
$$ \theta_j^n e_i^n=
\left\{\begin{array}{c}
 e_{i-1}^{n-1} \theta_j^{n-1}, \quad\textnormal{for}\hspace{0,12cm} j\leq i-2 \\
\sigma_i^n, \quad j=i-1  \hspace{0,12cm}\textnormal{or}\hspace{0,12cm} j=i\\
e_i^{n-1} \theta_{j-1}^{n-1},\quad j\geq i+1
\end{array}\right.
$$
Now let us prove the homotopy relation. We find:
$$\begin{array}{rl}
\theta^{(n)}\comp\delta_n \hskip-1em&= \displaystyle\sum_{\substack{1\leq i\leq n+1 \\ 1\leq j\leq n}}(-1)^{i+j} P(\theta_j^n e_i^n)\\
&= \displaystyle\sum_{\substack{3\leq i\leq n+1 \\ 1\leq j\leq i-2}}(-1)^{i+j} P(e_{i-1}^{n-1}\theta_j^{n-1}) +
\displaystyle\sum_{\substack{2\leq i\leq n+1 \\ j=i-1}}(-1)^{i+j} P(\sigma_i^n) \\
& + \displaystyle\sum_{\substack{1\leq i\leq n \\ j=i}}(-1)^{i+j} P(\sigma_i^n) +
\displaystyle\sum_{\substack{1\leq i\leq n-1 \\ i+1\leq j\leq n}}(-1)^{i+j} P(e_i^{n-1}\theta_{j-1}^{n-1}).
\end{array}$$
The two middle sums are equal to: $-P(\sigma_2)-\dots -P(\sigma_{n+1})+P(\sigma_1)+\dots +P(\sigma_n)=P(\sigma_1)-P(\sigma_{n+1})=G'^{(n)}-F'^{(n)}$.
Then:
\begin{multline*}
\theta^{(n)}\comp\delta_n + F'^{(n)}- G'^{(n)} = \\
\displaystyle-\sum_{\substack{3\leq i\leq n+1 \\ 1\leq j\leq (i-1)-1}}(-1)^{i+j-1} P(e_{i-1}^{n-1}\theta_j^{n-1}) -
\displaystyle\sum_{\substack{1\leq i\leq n-1 \\ i\leq j-1\leq n-1}}(-1)^{i+j-1} P(e_i^{n-1}\theta_{j-1}^{n-1})\\
=\displaystyle-\sum_{\substack{2\leq i\leq n \\ 1\leq j\leq i-1}}(-1)^{i+j} P(e_i^{n-1}\theta_j^{n-1}) -
\displaystyle\sum_{\substack{1\leq i\leq n-1 \\ i\leq j\leq n-1}}(-1)^{i+j} P(e_i^{n-1}\theta_j^{n-1})\\
=\displaystyle-\sum_{\substack{2\leq i\leq n-1 \\ 1\leq j\leq n-1}}(-1)^{i+j} P(e_i^{n-1}\theta_j^{n-1}) -
\displaystyle\sum_{1\leq j\leq n-1}(-1)^{1+j} P(e_1^{n-1}\theta_j^{n-1}) - \displaystyle\sum_{1\leq j\leq n-1}(-1)^{n+j} P(e_n^{n-1}\theta_j^{n-1})\\
=\displaystyle\sum_{\substack{1\leq i\leq n \\ 1\leq j\leq n-1}}(-1)^{i+j} P(e_i^{n-1}\theta_j^{n-1})=\delta'_{n-1}\comp\theta^{n-1}.
\end{multline*}
Thus the maps $F', G': C(\C,P)\to C(\D,P)$ are homotopic and this clearly implies that the induced group maps $F_*$ and $G_*$ are equal.
\qed\end{proof}

We will consider the cases: $P=\Pic$, where $\Pic(\C)$ is the Picard group of a symmetric finite tensor category $\C$, and $P=\Inv$,
where $\Inv(\C)$ is the group of invertible objects of $\C$.

\subsection{On 3-cocycles on invertible objects}

We start with a more general definition:

\begin{defn}
Let $X=X^1\Del\cdots\Del X^n\in\C^{\Del n}$. The object $|X|:=X^1\cdots X^n\in\C$ is called {\em the norm of $X$}. For any cocycle we will say that it is
{\em normalized} if its norm is isomorphic to the unit object $I$.
\end{defn}

\medskip

Assume {\em from now on that $\C$ is symmetric. }

\medskip

We will denote $X_i=e^n_i(X)$ for $i=1,2,\dots, n+1$ (mind the difference between upper and lower indeces).
In particular, $X\in\Inv(\C^{\Del 3})$ is a cocycle in $Z^3(\C,\Inv)$ if and only if
$$X_1X_2^{-1}X_3X_4^{-1}\iso I^{\Del 4}.$$
Given that $\C$ is symmetric, this is equivalent to: $X_1X_3\iso X_2X_4$, which means:
\begin{equation}\eqlabel{3-cocycle}
X^1\Del\crta{X^1}X^2\Del\crta{X^2}\Del\crta{X^3}X^3\iso X^1\crta{X^1}\Del\crta{X^2}\Del X^2\crta{X^3}\Del X^3.
\end{equation}

\medskip

\begin{lma} \lelabel{compatibility for coring}
For $[X]\in Z^3(\C,\Inv)$ the following hold:
\begin{enumerate}
\item $X^1\Del |X|^{-1}X^2X^3\iso I\Del I\iso |X|^{-1}X^1X^2\Del X^3.$
\item $X$ is cohomologous to a normalized cocycle.
\end{enumerate}
\end{lma}

\begin{proof}
Denote $X=X^1\Del X^2\Del X^3=U^1\Del U^2\Del U^3$ and $Y=X^{-1}=Y^1\Del Y^2\Del Y^3=V^1\Del V^2\Del V^3$. The 3-cocycle condition 
then reads:
\begin{eqnarray} \eqlabel{2-coc-cond-todo}
X^1U^1V^1\Del U^2Y^1V^2\Del X^2U^3Y^2\Del X^3Y^3V^3\iso I\Del I\Del I\Del I.
\end{eqnarray}

Tensoring out the second, third and fourth Deligne tensor factors above we get:
$$X^1U^1V^1\Del \vert Y\vert~ U^2V^2X^2U^3V^3X^3 \iso X^1\Del\vert X\vert^{-1}X^2X^3\iso I\Del I$$
which is the first equality that was to prove. The second one is obtained similarly, after
tensoring out the first three Deligne tensor factors in the 3-cocycle condition.

For the second part, first note that $\delta_2(|X|^{-1}\Del I)\iso I\Del |X|^{-1}\Del I$. The cocycle
$X\delta_2(|X|^{-1}\Del I)\iso X^1\Del |X|^{-1}X^2\Del X^3$ is obviously normalized and cohomologous to $X$.
\qed\end{proof}

\subsection{Extended cocycles yield coboundaries}  \sslabel{Extended cocycles}

In this subsection we will consider the Amitsur complex $C(\C\Del\C/\C, P)$ and we will show that cocycles from the complex $C(\C/vec, P)$
give rise to coboundaries in the former complex.

\medskip

Let $\sigma, \tau:\C\to\D$ be two tensor functors. Consider a $\D$-bimodule category $\N$ as a left $\C$-module category through $\sigma$
and a right $\C$-module category through $\tau$. Then $\N$ is a $\C$-bimodule category. If $\D$ is braided we may consider $\N$ as a one-sided $\D$-bimodule category.
If, moreover, $\sigma$ and $\tau$ are braided tensor functors, $\N$ is a one-sided $\C$-bimodule category in two ways - via $\sigma$ and via $\tau$.
As an example think of $\D=\C\Del\C$ with $\sigma(X)=X\Del I$ and $\tau(X)=I\Del X$ for all $X\in\C$.

\medskip

Suppose $\D$ is braided and $\N\in\dul\Pic(\D)$. The category $\N^{op}\Del_{\C}\N$ is a quasi $\D$-coring category similarly as in \exref{can}.
We will call it a {\em canonical quasi $\D/\C$-coring category} and will denote it by $\Can(\N; \D/\C)$. Moreover, by \leref{braided: product over braided}
the category $\D\Del_{\C}\D$ is braided and $\N^{op}\Del_{\C}\N$ is a one-sided $\D\Del_{\C}\D$-bimodule category.

\medskip

Now, let $\D=\C\Del\C$. The n-th Deligne tensor power of $\C\Del\C$ in the Amitsur complex $C(\C\Del\C/\C, P)$ is
$\underbrace{(\C\Del\C)\Del_{\C}\cdots \Del_{\C}(\C\Del\C)}_n=(\C\Del\C)^{\Del_{\C}n}$. The corresponding $n$-th cohomology group we will
denote by $H^n(\C\Del\C/{\C}, P)$ for a suitable functor $P$.

Observe that there is a natural equivalence:
\begin{equation} \eqlabel{A-C}
(\C\Del\C)^{\Del_{\C}n}\simeq \C^{\Del (n+1)}
\end{equation}
given by
$$(X^1\Del Y^1)\Del_{\C}\cdots\Del_{\C} (X^n\Del Y^n)\mapsto X^1\Del\cdots\Del X^n\Del Y^1\cdots Y^n.$$

\begin{lma} \lelabel{equiv fun}
For a braided finite tensor category $\C$ the above assignment defines a tensor equivalence. If $\C$ is symmetric, it is a braided monoidal equivalence of categories.
\end{lma}

\begin{proof}
It suffices to prove the claim for $n=2$, let $F: (\C\Del\C)\Del_{\C}(\C\Del\C)\to \C\Del\C\Del\C$ denote the corresponding equivalence. Consider the morphism $\omega$ defined
via the commuting diagram:
$$\scalebox{0.82}{\bfig
 \putmorphism(-20,400)(-2,-1)[F\left( ((C\Del D)\Del_{\C}(E\Del F)) \boxdot ((C'\Del D')\Del_{\C}(E'\Del F')) \right)`
	F\left((CC'\Del DD')\Del_{\C}(EE'\Del FF')\right)  ` =]{600}1l
 \putmorphism(-500,100)(-2,-1)[ `` ]{600}1l
 \putmorphism(-500,120)(-2,-1)[ `` =]{600}0l
 \putmorphism(-450,100)(-2,-1)[ `CC'\Del EE'\Del DD'FF' ` ]{600}0l 
 \putmorphism(-1080,-220)(2,-1)[ `CC'\Del EE'\Del DFD'F' ` ]{1000}1l
 \putmorphism(-1080,-250)(2,-1)[ ` ` \Id\ot\Phi_{D', F}\ot\Id]{1000}0l
\putmorphism(320,-620)(2,1)[\phantom{B\ot B}` ` ]{80}1r
\putmorphism(320,-640)(2,1)[\phantom{B\ot B}` ` =]{80}0r
\putmorphism(280,-660)(2,1)[\phantom{B\ot B}` (C\Del E\Del DF)\odot(C'\Del E'\Del D'F')` ]{450}0r
\putmorphism(1000,-310)(2,1)[\phantom{B\ot B}` ` =]{100}0r
\putmorphism(1000,-280)(2,1)[\phantom{B\ot B}` ` ]{100}1r
\putmorphism(1160,-130)(2,1)[\phantom{B\ot B}` F(((C\Del D)\Del_{\C}(E\Del F)) \odot F((C'\Del D')\Del_{\C}(E'\Del F')) ` ]{100}0r
 \putmorphism(360,400)(2,-1)[`` \omega]{960}1r
\efig}
$$
To prove that $\omega$ defines a monoidal structure on $F$ we take a third object $(C''\Del D'')\Del_{\C}(E''\Del F'')$ and we should check
that the identity $(\Id\ot\omega)\omega_{\bullet, \bullet\bullet} F(\alpha)=\alpha_{F, F, F}(\omega\ot\Id)\omega_{\bullet\bullet, \bullet}$ holds,
where $\alpha$ is the associativity constraint in $\D\Del_{\C}\D$.
Since $\omega$ for the first two objects is basically given by $\Phi_{D', F}$, the above identity (neglecting the associativity constraint)
will be fulfilled if we prove that: $(FD'\ot\Phi_{D'', F'})(\Phi_{D'D'', F}\ot F')=(\Phi_{D', F}\ot F'D'')(D'\ot\Phi_{D'', FF'})$. But this is true by
the two braiding axioms and naturality.

Let us now prove that $\omega$ is a braided tensor equivalence. Recall that the braiding $\Psi$ in $(\C\Del\C)\Del_{\C}(\C\Del\C)$ is given by
$\Psi=\tilde\Phi\Del_{\C}\tilde\Phi$ \equref{big braiding Psi}, where $\tilde\Phi$ is the braiding in $\C\Del\C$, which is given by
$\tilde\Phi=\Phi\Del\Phi$, being $\Phi$ the braiding in $\C$, that is: 
$$\scalebox{0.88}{\bfig
 \putmorphism(-30,500)(1,0)[(X\Del Y)\odot(X'\Del Y')` (X'\Del Y')\odot(X\Del Y) `\tilde\Phi_{X\Del Y, X'\Del Y'}]{1630}1a
 \putmorphism(-50,0)(1,0)[(X\ot X')\Del(Y\ot Y')` (X'\ot X)\Del(Y'\ot Y).` \Phi_{X,X'}\Del\Phi_{Y,Y'}]{1700}1a
\putmorphism(-60,500)(0,-1)[\phantom{B\ot B}``=]{480}1l
\putmorphism(1630,500)(0,-1)[\phantom{B\ot B}``=]{480}1r
\efig}
$$
Denote by $X=(C\Del D)\Del_{\C}(E\Del F)$ and $Y=(C'\Del D')\Del_{\C}(E'\Del F')$, then we should
check that
$$\scalebox{0.88}{\bfig
 \putmorphism(-30,500)(1,0)[F(XY)` F(X)F(Y) `\omega_{X,Y}]{1630}1a
 \putmorphism(-50,0)(1,0)[F(YX)` F(Y)F(X)` \omega_{Y,X}]{1700}1a
\putmorphism(-60,500)(0,-1)[\phantom{B\ot B}``F(\Psi)]{480}1l
\putmorphism(1630,500)(0,-1)[\phantom{B\ot B}``\Phi^{\Del 3}]{480}1r
\efig}
$$
commutes. Observe that this is the same as:
$$\scalebox{0.84}{
\bfig 
\putmorphism(0,600)(1,0)[CC'\Del EE'\Del DD'FF'	` CC'\Del EE'\Del (DF)(D'F') ` \Id\ot\Phi_{D', F}\ot\Id]{2200}1a
\putmorphism(60,600)(0,1)[`F\left((CC'\Del DD')\Del_{\C}(EE'\Del FF')\right)`=]{300}1l
\putmorphism(60,300)(0,1)[`F\left((C'C\Del D'D)\Del_{\C}(E'E\Del F'F)\right) ` F(\tilde\Phi\Del_{\C}\tilde\Phi)]{380}1l
\putmorphism(60,-80)(0,1)[` ` =]{300}1l
\putmorphism(0,-380)(1,0)[C'C\Del E'E\Del D'DF'F ` C'C\Del E'E\Del D'F'DF` \Id\ot\Phi_{D, F'}\ot\Id]{2200}1b
\putmorphism(2200,600)(0,-1)[`` \Phi_{C, C'}\Del\Phi_{E, E'}\Del\Phi_{DF, D'F'}]{980}1r
\efig}
$$
Observe that the composition of the left three vertical arrows in the above diagram equals: $\Phi_{C, C'}\Del\Phi_{E, E'}\Del\Phi_{D, D'}\Phi_{F, F'}$.
So, to check the commutativity of the above diagram, it is sufficient to see if
 $(D'\ot\Phi_{D, F'}\ot F)(\Phi_{D,D'}\ot\Phi_{F,F'})=\Phi_{DF, D'F'}(D\ot\Phi_{D', F}\ot F')$. Graphically this means that the identity
$$
\scalebox{0.9}[0.9]{
\gbeg{4}{4}
\got{1}{D} \got{1}{D'} \got{1}{F}  \got{1}{F'} \gnl
\gbr \gbr \gnl
\gcl{1} \gbr \gcl{1} \gnl
\gob{1}{D'} \gob{1}{F'} \gob{1}{D} \gob{1}{F}
\gend}=
\scalebox{0.9}[0.9]{
\gbeg{4}{6}
\got{1}{D} \got{1}{D'} \got{1}{F}  \got{1}{F'} \gnl
\gcl{1} \gbr \gcl{1} \gnl
\gcl{1} \gbr \gcl{1} \gnl
\gbr \gbr \gnl
\gcl{1} \gbr \gcl{1} \gnl
\gob{1}{D'} \gob{1}{F'} \gob{1}{D} \gob{1}{F}
\gend}
$$
should hold true. Well, this is only possible if $\Phi_{D', F}=(\Phi_{D', F})^{-1}$, which is fulfilled if $\C$ is symmetric.
\qed\end{proof}

We defined in \cite[Section 5.1]{Femic2} monoidal functors
$E^n_i, \delta_{n}: \dul\Pic(\C^{\Del n})\to\dul\Pic(\C^{\Del (n+1)})$ for $i=1,\cdots,n+1$ by
$
E^n_i(\M)=\M_i=\M\Del_{\C^{\Del n}} {}_{e^n_i}\C^{\Del (n+1)}
$ and 
$$\delta_{n}(\M)=\M_1\Del_{\C^{\Del (n+1)}}\M^{op}_2\Del_{\C^{\Del (n+1)}}\cdots
\Del_{\C^{\Del (n+1)}}\N_{n+1}$$
on objects. One computes that
$$\delta_{n+1}\delta_{n}(\M)=(\boxtimes_{\C^{\Del (n+2)}})_{j=2}^{n+2}(\boxtimes_{\C^{\Del (n+2)}})_{i=1}^{j-1} \hspace{0,2cm} (\M_{ij}\Del_{\C^{\Del (n+2)}} \M_{ij}^{op}),
$$
hence the functor
$$\lambda_{\M}= (\boxtimes_{\C^{\Del (n+2)}})_{j=2}^{n+2}(\boxtimes_{\C^{\Del (n+2)}})_{i=1}^{j-1} \ev_{\M_{ij}}:\ \delta_{n+1}\delta_n(\M)\to \C^{\Del (n+2)}$$
is an equivalence (since so is $\ev$, \coref{coev inverse ev}). For $\M\in \dul{\Pic}(\C^{\Del n})$ we proved in \cite[Lemma 5.6]{Femic2} and the identity (44) ib\'idem:
\begin{equation}\eqlabel{2.2.1.1b}
\M_{ij}\cong \M_{j(i+1)}
\end{equation}
when $i\geq j\in \{1,\dots,n+1\}$, where $\M_{ij}=E_j^{n+1}\comp E_i^n(\M)$, and:
\begin{equation}\eqlabel{delta delta ev}
\delta_{n+1}\delta_{n}(\ev_{\M})= \lambda_{\M} \Del_{\C^{\Del (n+2)}}\crta\lambda_{\M}.
\end{equation}

Let us now denote the functors corresponding to $E^n_i$ and $\delta_{n}$ in the setting of the Amitsur complex $C(\C\Del\C/\C, \dul\Pic)$ by:
$$E_i^{'n}, \delta_n': \dul\Pic((\C\Del\C)^{\Del_{\C}n})\to\dul\Pic((\C\Del\C)^{\Del_{\C}(n+1)}).$$

Suppose that $\C$ is symmetric. Due to \leref{equiv fun} the augmentation functors
$$\tilde\eta^n_i:\ (\C\Del\C)^{\Del_{\C}n}\to (\C\Del\C)^{\Del_{\C} (n+1)}$$
for $i=1,2, \cdots, n+1$ can be viewed as functors
\begin{equation} \eqlabel{eta_i functors}
\eta_i^{n+1}:\ \C^{\Del (n+1)}\to \C^{\Del (n+2)}.
\end{equation}

\bigskip

\begin{lma} \lelabel{e_i s distintos}
For a symmetric category $\C$ the above functors $\eta_i^n: \C^{\Del n}\to \C^{\Del (n+1)}$ \equref{eta_i functors} coincide with the
augmentation functors $e^n_i: \C^{\Del n}\to \C^{\Del (n+1)}$ from \equref{e's} for $i=1,2, \cdots, n$. Consequently:
\begin{enumerate}
\item $e^n_i(X)=\tilde\eta^{n-1}_i(X)$, for all $X\in\Inv(\C^{\Del n})\iso\Inv((\C\Del\C)^{\Del_{\C}(n-1)})$ for $i=1,2, \cdots, n$;
\item $E^n_i(\M)=\M\Del_{\C^{\Del n}} {}_{e^n_i}\C^{\Del (n+1)}\simeq
\M\Del_{(\C\Del\C)^{\Del_{\C} (n-1)}} {}_{\tilde\eta^{n-1}_i}(\C\Del\C)^{\Del_{\C} n}=E^{' n-1}_i(\M)$ for all $\M\in\dul\Pic(\C^{\Del n})$
for $i=1,2, \cdots, n$.
\end{enumerate}
\end{lma}

\begin{proof}
Since $X=X^1\Del\cdots \Del X^n$ corresponds to $(X^1\Del I)\Del_{\C}\cdots\Del_{\C} (X^{n-2}\Del I)\Del_{\C} (X^{n-1}\Del X^n)$
in $(\C\Del\C)^{\Del_{\C}(n-1)}$ by the equivalence \equref{A-C}, for every $i=1,2, \cdots, n-1$ we have:
$\tilde\eta^{n-1}_i(X)=(X^1\Del I)\Del_{\C}\cdots\Del_{\C} (I\Del I)\Del_{\C}(X^i\Del I)\Del_{\C}\cdots\Del_{\C}(X^{n-2}\Del I)\Del_{\C} (X^{n-1}\Del X^n)$
which corresponds to $X^1\Del \cdots \Del I\Del X^i\Del \cdots\Del X^{n-2}\Del X^{n-1}\Del X^n=e^n_i(X)$, whereas for $i=n$ we find:
$\tilde\eta^{n-1}_n(X)=(X^1\Del I)\Del_{\C}\cdots\Del_{\C}(X^{n-1}\Del X^n)\Del_{\C} (I\Del I)$, which corresponds to:
$X^1\Del \cdots \Del X^{n-1}\Del I\Del X^n=e^n_n(X)$.
\qed\end{proof}

We recall that in \cite[Section 5.1]{Femic2} we defined the category $\dul{Z}^n(\C,\dul{\Pic})$ whose objects are $(\M,\alpha)$,
where $\M\in \dul{\Pic}(\C^{\Del n})$, and $\alpha:\ \delta_n(\M)\to \C^{\Del(n+1)}$ is an equivalence of $\C^{\Del(n+1)}$-module categories so
that $\delta_{n+1}(\alpha)\simeq\lambda_{\M}$. A morphism $(\M,\alpha)\to (\N,\beta)$ is an equivalence of $\C^{\Del n}$-module categories
$F:\ \M\to \N$ fulfilling $\beta\circ \delta_n(F)\simeq\alpha$. Equipped with the tensor product $(\M,\alpha)\ot (\N,\beta)=
(\M\Del_{\C^{\Del n}}\N, \alpha\Del_{\C^{\Del (n+1)}}\beta)$ and the unit object $(\C^{\Del n}, \C^{\Del (n+1)})$ the category $\dul{Z}^n(\C,\dul{\Pic})$
is symmetric monoidal. Every object in this category is invertible and the corresponding Grothendieck group we denote by ${Z}^n(\C,\dul{\Pic}).$
Furthermore, we denoted by $d_{n-1}:\ \dul{\Pic}(\C^{\Del (n-1)})\to \dul{Z}^n(\C,\dul{\Pic})$ the monoidal functor
given by $d_{n-1}(\N)=(\delta_{n-1}(\N),\lambda_{\N})$. The subgroup of $Z^n(\C,\dul{\Pic})$ consisting of elements represented by $d_{n-1}(\N)$ we
denoted by $B^n(\C,\dul{\Pic})$ and we defined:
\begin{equation} \eqlabel{coh gr dul Pic}
H^n(\C,\dul{\Pic})=Z^n(\C,\dul{\Pic})/B^n(\C,\dul{\Pic}).
\end{equation}

Let now $d_{n-1}':  \dul{\Pic}((\C\Del\C)^{\Del_{\C} (n-1)})\to \dul{Z}^n(\C\Del \C/\C,\dul{\Pic})$
denote the corresponding functor in the Amitsur complex $C(\C\Del\C/\C, \dul\Pic)$. It is given by $d_{n-1}'(\M)=(\delta'_{n-1}(\M), \lambda'_{\M})$.

\begin{lma} \lelabel{extended coc} 
If $[X]\in Z^n(\C,\Inv)$, then $[X\Del I]\in B^n(\C\Del \C/\C,\Inv)$.

If $(\M, \alpha)\in Z^n(\C,\dul\Pic)$, then $(\M\Del\C, \alpha\Del\C)\iso d_{n-1}'(\M)=(\delta'_{n-1}(\M), \lambda'_{\M})\in B^n(\C\Del \C/\C,\dul\Pic)$.
\end{lma}

\begin{proof}
By the n-cocycle condition and \leref{e_i s distintos} we have: $X\Del I=X_{n+1}=X_1X_2^{-1}\cdots X_n^{\pm 1}\\=\delta'_{n-1}(X)\in B^n(\C\Del \C/\C,\Inv)$.

For the second part, observe that the equivalence $\alpha:\ \M_1\Del_{\C^{\Del (n+1)}}\M^{op}_2\Del_{\C^{\Del (n+1)}}\cdots\Del_{\C^{\Del {n+1}}}\M_{n+1}^{\pm 1}\to \C^{\Del (n+1)}$
induces an equivalence
$$\beta:\ \M_{n+1}^{\pm 1}=\M^{\pm 1}\Del \C\to \M^{op}_1\Del_{\C^{\Del {n+1}}}\M_2\Del_{\C^{\Del {n+1}}}\cdots\Del_{\C^{\Del {n+1}}}\M_n^{\pm 1}=\delta_{n-1}'(\M^{op}).$$
It is given by the universal property of the evaluation functor through
\begin{equation}\eqlabel{beta}
\ev_{\delta_{n-1}'(\M)}(\Id_{\delta_{n-1}'(\M)}\Del_{\C^{\Del (n+1)}}\beta)=\alpha.
\end{equation}
We have $\delta'_n(\ev_{\delta_{n-1}'(\M)})=\delta'_n\delta_{n-1}'(\ev_{\M})\stackrel{\equref{delta delta ev}}{=}
\lambda'_{\M} \Del_{\C^{\Del (n+2)}}\crta\lambda'_{\M}$.
Applying the functor $\delta'_n$ to the equation \equref{beta} and by its monoidality we obtain:
\begin{multline*}
(\lambda'_{\M} \Del_{\C^{\Del (n+2)}}\crta\lambda'_{\M})(\delta'_n(\Id_{\delta_{n-1}'(\M)})\Del_{\C^{\Del (n+2)}}\delta'_n(\beta))=\\
\lambda'_{\M}\Del_{\C^{\Del (n+2)}}(\lambda'_{\M}\comp\delta'_n(\beta))=\alpha_1\Del_{\C^{\Del (n+2)}}\crta\alpha_2^{-1}\Del_{\C^{\Del (n+2)}}\dots
\Del_{\C^{\Del (n+2)}}\alpha_{n+1}^{\pm 1}.
\end{multline*}
Here $\alpha_{n+1}^{\pm 1}$ is $\alpha_{n+1}$ if $n$ is odd, and it is $\crta\alpha_{n+1}^{-1}$ if $n$ is even.
Now, observe that $\M\in\dul\Pic(\C^{\Del n})=\dul\Pic((\C\Del\C)^{\Del_{\C}(n-1)})$, so
$\lambda'_{\M}= (\boxtimes_{\C^{\Del (n+2)}})_{j=2}^{n+1}(\boxtimes_{\C^{\Del (n+2)}})_{i=1}^{j-1} \ev_{\M_{ij}}$.
Here we use the identification \equref{A-C}: $(\C\Del\C)^{\Del_{\C}(n+1)}\cong \C^{\Del (n+2)}$ in the Deligne tensor product of bimodule categories.
On the other hand, we have that $\delta_{n+1}(\alpha)=\lambda_{\M}$, which means:
\begin{multline*}
\alpha_1\Del_{\C^{\Del (n+2)}}\crta\alpha_2^{-1}\Del_{\C^{\Del (n+2)}}\dots\Del_{\C^{\Del (n+2)}}\alpha_{n+2}^{\mp 1}=\\
(\boxtimes_{\C^{\Del (n+2)}})_{j=2}^{n+2}(\boxtimes_{\C^{\Del (n+2)}})_{i=1}^{j-1} \ev_{\M_{ij}}=
\lambda'_{\M}\boxtimes_{\C^{\Del (n+2)}} (\boxtimes_{\C^{\Del (n+2)}})_{i=1}^{n+1} \ev_{\M_{i,n+2}}
\end{multline*}
Combining the last two equations and cancelling out the functor $\lambda'_{\M}$ by \leref{basic lema}, 2), we get:
$$(\lambda'_{\M}\comp\delta'_n(\beta)) \Del_{\C^{\Del (n+2)}}\alpha_{n+2}^{\mp 1}= (\boxtimes_{\C^{\Del (n+2)}})_{i=1}^{n+1} \ev_{\M_{i,n+2}}=
\delta_n(\ev_{\M})_{n+2}=(\ev_{\delta_n(\M)})_{n+2}.$$
Due to \leref{alfa dagger}, the right hand-side is equal to $(\alpha \Del_{\C^{\Del (n+1)}}\crta\alpha^{- 1})_{n+2}=
\alpha_{n+2}\Del_{\C^{\Del (n+2)}}\crta\alpha^{-1}_{n+2}$. Now by \leref{basic lema}, 2), applied to the functor $\alpha_{n+2}^{\mp 1}$
yields $\lambda'_{\M}\comp\delta'_n(\beta)=\alpha_{n+2}^{\pm 1}=\alpha^{\pm 1}\Del\C$ depending on whether $n$ is odd or even.
We have proved that $\beta$ is an isomorphism between $(\M\Del\C, \alpha\Del\C)$ and $(\delta'_{n-1}(\M), \lambda'_{\M})$
in $\dul{Z}^n(\C\Del \C/\C,\dul{\Pic})$.
\qed\end{proof}

\medskip

\begin{cor}
Let $\M\in\dul\Pic(\C^{\Del n})$. Then $\delta'_n(\M\Del\C)=\delta_n(\M)\Del\C$.
\end{cor}

\begin{proof}
We have that $\M\Del\C\in\dul\Pic(\C^{\Del (n+1)})\iso\dul\Pic((\C\Del\C)^{\Del_{\C} n})$ and by \leref{e_i s distintos} $\M\Del\C=\M_{n+1}$ in both categories.
Thus, applying \equref{2.2.1.1b}, we find:
\begin{eqnarray*}
\delta'_n(\M\Del\C)&=& \M_{n+1,1}\Del_{\C^{\Del (n+2)}}\M^{op}_{n+1,2}\Del_{\C^{\Del (n+2)}}\cdots\Del_{\C^{\Del (n+2)}}\M^{\pm 1}_{n+1, n+1}\\
&=&\M_{1,n+2}\Del_{\C^{\Del (n+2)}}\M^{op}_{2,n+2}\Del_{\C^{\Del (n+2)}}\cdots\Del_{\C^{\Del (n+2)}}\M^{\pm 1}_{n+1, n+2}\\
&=&(\M_1\Del\C)\Del_{(\C^{\Del (n+1)}\Del\C)}(\M^{op}_2\Del\C)\Del_{(\C^{\Del (n+1)}\Del\C)}\cdots\Del_{(\C^{\Del (n+1)}\Del\C)}(\M^{\pm 1}_{n+1}\Del\C)\\
&\iso &(\M_1\Del_{\C^{\Del (n+1)}}\M^{op}_2\Del_{\C^{\Del (n+1)}}\cdots\Del_{\C^{\Del (n+1)}}(\M^{\pm 1}_{n+1})\Del\C\\
&=&\delta_n(\M)\Del\C
\end{eqnarray*}
where the isomorphism in the penultimate line is due to \leref{Weq}.
\qed\end{proof}

\subsection{Few words on middle cohomology groups $H^n(\C,\dul{\Pic})$}

The following result will be crucial in \ssref{Full group} where we discuss the full group of Azumaya quasi coring categories.

\begin{prop}\prlabel{coh maps dul}
Let $F: \C\to\D$ be a functor between two symmetric finite tensor categories.
Then $F$ induces group morphisms $F_*:\ H^n(\C,\dul{\Pic})\to H^n(\D,\dul{\Pic})$.
If $G: \C\to\D$ is another symmetric tensor functor, then $F_*=G_*$.
\end{prop}

\begin{proof}
The induced functor $F^*:\ \dul{Z}^n(\C,\dul{\Pic})\to\dul{Z}^n(\D,\dul{\Pic})$ is given by
\begin{equation} \eqlabel{fun ind cat}
F^*(\M,\alpha)=(\M\Del_{\C^{\Del n}}{}_{F^{\Del n}}(\D^{\Del n}),\alpha\Del_{\C^{\Del (n+1)}} {}_{F^{\Del (n+1)}}(\D^{\Del (n+1)})).
\end{equation}
It induces maps $F_*:\ H^n(\C,\dul{\Pic})\to H^n(\D,\dul{\Pic})$ - recall \equref{coh gr dul Pic} - as we clearly have:
$$(\delta_{n-1}(\N)\Del_{\C^{\Del n}}\D^{\Del n},\lambda_{\N}\Del_{\C^{\Del (n+1)}}\D^{\Del (n+1)})=
(\delta_{n-1}(\N\Del_{\C^{\Del (n-1)}}\D^{\Del (n-1)}),\lambda_{\N\Del_{\C^{\Del (n-1)}}\D^{\Del (n-1)}}).$$
(Here the action of the functor $F$ is suppressed in the notation.)
So far we have that two tensor functors $F$ and $G$ induce maps $F_*, G_*: \ H^n(\C,P)\to H^n(\D,P)$ for $P$ being one of the functors: $\Inv, \Pic$ and $\dul{\Pic}$,
and for every natural $n$, recall \prref{Knus}. Hence we have the corresponding maps between the exact sequence \equref{VZ seq} and its analog with $\C$ replaced
by $\D$. We know from \prref{Knus} that these maps coincide on $H^n(\C,\Inv)$ and $H^n(\D,\Pic)$. Then by the five lemma we obtain that they also coincide on $H^n(\D,\dul{\Pic})$.
\qed\end{proof}

The result of \prref{coh maps dul} enables us to build a functor
\begin{equation} \eqlabel{H2 functor}
H^n(\bullet/vec,\dul{\Pic}): FTC_{symm}\to\dul{Ab}
\end{equation}
from the category of symmetric finite tensor categories to that of abelian groups. Then we can consider the colimit of $H^n(\C/vec,\dul{\Pic})$
over symmetric finite tensor categories $\C$. Namely, let $\Omega_{sf}$ denote the family of equivalence
classes of symmetric finite tensor categories. Consider the following relation of order on $\Omega_{sf}: [\C]\leq[\D]$ if and only if
there is a symmetric tensor functor $\C\to\D$. Then $\Omega_{sf}$ is a directed family: given $\C$ and $\D$, the category $\C\Del\D\in FTC_{symm}$
and we have symmetric tensor functors $\C\to\C\Del\D$ and $\D\to\C\Del\D$, given by $X\mapsto X\Del I$ and $Y\mapsto I\Del Y$, respectively.
For each symmetric tensor functor $\C\to\D$, by \prref{coh maps dul} we have a unique group map $\Omega_{\C, \D}: H^n(\C,\dul{\Pic})\to H^n(\D,\dul{\Pic})$
satisfying $\Omega_{\C, \C}=\Id$ and $\Omega_{\D, \E}\comp\Omega_{\C, \D}=\Omega_{\C, \E}$ for every triple $[\C]\leq[\D]\leq[\E]$. Then we may define
$$H^n(vec,\dul{\Pic})=\colim_{[\C]\in\Omega_{sf}} H^n(\C/vec,\dul{\Pic})=\colim~ H^n(\bullet/vec,\dul{\Pic}).$$

\section{Interpretation of $H^2(\C,\dul\Pic)$} \selabel{Interpret}

{\em Let $\C$ throughout be a symmetric finite tensor category. }
To develop the results on the interpretation of the middle term in the second level of the sequence \equref{VZ seq} necessary pieces of information
are the following. On one hand, a fact that we already needed when defining Amitsur cohomology in \seref{Amitsur coh} -
that the dual object for an invertible one-sided bimodule category is its opposite category and that the evaluation functor involved is an
equivalence functor, \cite[Section 4]{Femic2}. At this point we will also use the result of \coref{coev inverse ev} that the coevaluation functor is the quasi-inverse
of the evaluation functor. On the other hand, crucial are: \prref{sim bimod} - which enables us to freely interchange the order of the factors
in the Deligne tensor product over $\C$, since $(\C^{br}\x\dul\Mod, \Del_{\C}, \C)$ is a symmetric monoidal category; \leref{basic lema} -
which justifies the cancellation of equivalence functors between two invertible $\C$-bimodule categories in the product in $\Pic(\C)$, and
finally \leref{alfa-delta}, as we see next.

\begin{lma} \lelabel{alfa-delta}
Let 
$\M\in\dul\Pic(\C\Del\C)$.
Let $\Delta:\M\to\M\Del_{\C}\M$ be a $\C$-bimodule functor and assume
that its corresponding $\C^{\Del 3}$-module functor $\tilde\Delta: \M_2\to \M_3\Del_{\C^{\Del 3}}\M_1$ (from \leref{Delta tilde}) is an equivalence. Then we may consider
the following equivalence of $\C^{\Del 3}$-module categories:
\begin{equation} \eqlabel{alfa-delta}
\alpha^{-1}:=(\tilde\Delta\Del_{\C^{\Del 3}}\M_2^{op})\crta\coev_{\M_2}:\C^{\Del 3}\to\M_3\Del_{\C^{\Del 3}}\M_1\Del_{\C^{\Del 3}}\M_2^{op}.
\end{equation}

Then $\Delta$ is coassociative (in the sense that $(\Delta\Del_{\C}\Id)\Delta\simeq(\Id\Del_{\C}\Delta)\Delta$) if and only if
$(\M, \alpha)\in\dul Z^2(\C, \dul\Pic)$.
\end{lma}

\begin{proof}
In this setting, $(\M, \alpha)\in\dul Z^2(\C, \dul\Pic)$ if and only if $\delta_3(\alpha)=\lambda_{\M}$. Observe that
$$\begin{array}{ccr}
\lambda_{\M} =(\Del_{\C^{\Del 4}})_{j=2}^{4}(\Del_{\C^{\Del 4}})_{i=1}^{j-1} \ev_{\M_{ij}} &=&  \ev_{12}\Del_{\C^{\Del 4}}\ev_{13}\Del_{\C^{\Del 4}}\ev_{14}\\
& &  \Del_{\C^{\Del 4}} \ev_{23}\Del_{\C^{\Del 4}}\ev_{24}\\
& &  \Del_{\C^{\Del 4}} \ev_{34}.
\end{array}$$
and that by \leref{alfa dagger} we have $\delta_3(\alpha)=\alpha_1\Del_{\C^{\Del 4}}\alpha_2^{\dagger}\Del_{\C^{\Del 4}}\alpha_3\Del_{\C^{\Del 4}}\alpha_4^{\dagger}$.
We tensor the equation $\delta_3(\alpha)=\lambda_{\M}$ (by $\Del_{\C^{\Del 4}}$) with the equivalence functor $\alpha_1^{\dagger}\Del_{\C^{\Del 4}}\alpha_3^{\dagger}$.
Then due to \leref{alfa dagger} and \leref{basic lema}, 2), it is equivalent to
\begin{multline*}
\ev_{\delta_2(\M)^{op}_1} \Del_{\C^{\Del 4}} \ev_{\delta_2(\M)^{op}_3} \Del_{\C^{\Del 4}}\alpha_2^{\dagger} \Del_{\C^{\Del 4}}\alpha_4^{\dagger}\\
=\ev_{12}\Del_{\C^{\Del 4}}\ev_{13}\Del_{\C^{\Del 4}}\ev_{14} \Del_{\C^{\Del 4}} \ev_{23}\Del_{\C^{\Del 4}}\ev_{24} \Del_{\C^{\Del 4}} \ev_{34}  \Del_{\C^{\Del 4}} \alpha_1^{\dagger}\Del_{\C^{\Del 4}}\alpha_3^{\dagger}
\end{multline*}
Observe that
\begin{eqnarray*}
&&\hspace*{-4,5cm}
\delta_2(\M)_1 \Del_{\C^{\Del 3}} \delta_2(\M)_3 = \M_{11}\Del_{\C^{\Del 3}}\M^{op}_{21}\Del_{\C^{\Del 3}}\M_{31}
 \Del_{\C^{\Del 3}}\M_{13}\Del_{\C^{\Del 3}}\M^{op}_{23}\Del_{\C^{\Del 3}}\M_{33}\\
\hspace*{3cm}&=& \hspace*{-0,2cm}
\M_{12}\Del_{\C^{\Del 3}}\M^{op}_{13}\Del_{\C^{\Del 3}}\M_{14}\Del_{\C^{\Del 3}}
\M_{13}\Del_{\C^{\Del 3}}\M^{op}_{23}\Del_{\C^{\Del 3}}\M_{34}
\end{eqnarray*}
then canceling out the same evaluation functors in the above equation we get equivalently:
\begin{equation}\eqlabel{dagger equation}
\crta\ev_{13}\Del_{\C^{\Del 4}}\alpha_2^{\dagger} \Del_{\C^{\Del 4}}\alpha_4^{\dagger}=  \ev_{24} \Del_{\C^{\Del 4}} \alpha_1^{\dagger}\Del_{\C^{\Del 4}}\alpha_3^{\dagger}.
\end{equation}
Recall that the equivalence functor $\alpha_i^{\dagger}: \delta_2(\M)^{op}_i\to\C^{\Del 4}$ is given by:
$$\alpha_i^{\dagger}=\crta\ev_{\delta_2(\M)_i}(\Id_{\delta_2(\M)_i^{op}}\Del_{\C^{\Del 4}}\alpha^{-1}_i)
=\crta\ev_{\delta_2(\M)_i}(\Id_{\delta_2(\M)_i^{op}}\Del_{\C^{\Del 4}}(\tilde\Delta_i\Del_{\C^{\Del 4}}\M_{2i}^{op})\crta\coev_{\M_{2i}}).$$
Let us write $\alpha_i^{\dagger}$ in braided diagrams. Observe that $\tilde\Delta_i: \M_{2i}\to \M_{3i}\Del_{\C^{\Del 4}}\M_{1i}$.
We will write $ij$ for $\M_{ij}$ and $\crta{ij}$ for $\M_{ij}^{op}$, we have:
$$\alpha_i^{\dagger}=
\gbeg{6}{6}
\got{1}{\crta{1i}} \got{1}{2i} \got{1}{\crta{3i}} \gnl
\gcl{2} \gcl{2} \gcl{2} \gwdb{3} \gnl
\gvac{3} \glmptb\gnot{\hspace{-0,4cm}\tilde\Delta_i}\grmpb \gcl{3} \gnl
\gbr \gev \gcl{1} \gnl
\gcl{1} \gwev{4} \gnl
\gwev{6} \gnl
\gob{1}{}
\gend
$$
Equation \equref{dagger equation} written in braided diagrams (in the symmetric monoidal category $\dul{\Pic}(\C^{\Del 4})$, so the order of the in- and output objects is irrelevant!
and the same holds for the sign of the braiding) looks like this:
$$
\gbeg{14}{6}
\got{1}{\crta{13}} \got{1}{13} \got{1}{\crta{12}} \got{1}{22} \got{2}{\crta{32}=\crta{24}} \got{2}{}
 \got{1}{\crta{14}} \got{1}{24} \got{1}{\crta{34}} \gnl
\gev \gcl{2} \gcl{2} \gcl{2} \gwdb{3} \gcl{2} \gcl{2} \gcl{2} \gwdb{3} \gnl
\gvac{2} \gvac{3} \glmptb\gnot{\hspace{-0,4cm}\tilde\Delta_2}\grmpb \gcl{3} \gvac{3} \glmptb\gnot{\hspace{-0,4cm}\tilde\Delta_4}\grmpb \gcl{3} \gnl
\gvac{2} \gbr \gev \gcl{1} \gvac{1} \gbr \gev \gcl{1} \gnl
\gvac{2} \gcl{1} \gwev{4} \gvac{1} \gcl{1} \gwev{4} \gnl
\gvac{2} \gwev{6} \gwev{6} \gnl
\gob{1}{}
\gend=
\gbeg{14}{6}
\got{1}{\crta{24}} \got{1}{24} \got{1}{\stackrel{\crta{11}}{\crta{12}}} \got{1}{\stackrel{21}{13}} \got{1}{\stackrel{\crta{31}}{\crta{14}}}
 \got{3}{} \got{1}{\crta{13}} \got{1}{23} \got{1}{\stackrel{\crta{33}}{\crta{34}}} \gnl
\gev \gcl{2} \gcl{2} \gcl{2} \gwdb{3} \gcl{2} \gcl{2} \gcl{2} \gwdb{3} \gnl
\gvac{2} \gvac{3} \glmptb\gnot{\hspace{-0,4cm}\tilde\Delta_1}\grmpb \gcl{3} \gvac{3} \glmptb\gnot{\hspace{-0,4cm}\tilde\Delta_3}\grmpb \gcl{3} \gnl
\gvac{2} \gbr \gev \gcl{1} \gvac{1} \gbr \gev \gcl{1} \gnl
\gvac{2} \gcl{1} \gwev{4} \gvac{1} \gcl{1} \gwev{4} \gnl
\gvac{2} \gwev{6} \gwev{6} \gnl
\gob{1}{}
\gend
$$
We now cancel out $\ev_{12},\ev_{14},\ev_{34}$ and the (identity functors on) the tensor factors $\crta{12}, \crta{14}$ and $\crta{34}$ to get equivalently:
$$
\gbeg{11}{7}
\got{1}{\crta{13}} \got{1}{13} \got{1}{23} \got{1}{\crta{24}} \got{3}{} \got{1}{24} \gnl
\gev \gcl{3} \gcl{2} \gwdb{3} \gcl{2} \gwdb{3} \gnl
\gvac{1} \gvac{3} \glmptb\gnot{\hspace{-0,4cm}\tilde\Delta_2}\grmpb \gcl{1} \gvac{1} \glmptb\gnot{\hspace{-0,4cm}\tilde\Delta_4}\grmpb \gcl{3} \gnl
\gvac{3} \gev \gbr \gbr \gcl{1} \gcl{1} \gnl
\gvac{2} \gwev{4} \gcl{2} \gcl{2} \gbr \gnl
\gvac{8} \gcl{1} \gev \gnl
\gvac{6} \gob{1}{12} \gob{1}{34} \gob{1}{14}
\gend=
\gbeg{14}{7}
\got{1}{\crta{24}} \got{1}{24} \got{1}{\stackrel{21}{13}} \got{3}{} \got{1}{\crta{13}} \got{1}{23} \gnl
\gev \gcl{2} \gwdb{3} \gcl{3} \gcl{2} \gwdb{3} \gnl
\gvac{3} \glmptb\gnot{\hspace{-0,4cm}\tilde\Delta_1}\grmpb \gcl{3} \gvac{2} \glmptb\gnot{\hspace{-0,4cm}\tilde\Delta_3}\grmpb \gcl{3} \gnl
\gvac{2} \gbr \gcl{1} \gvac{2} \gbr \gcl{1} \gnl
\gvac{2} \gcl{2} \gbr \gvac{1} \gbr \gibr \gnl
\gvac{3} \gcl{1} \gev \gcl{1} \gev \gev \gnl
\gvac{2} \gob{1}{\stackrel{31}{14}} \gob{1}{\stackrel{11}{12}} \gvac{2} \gob{1}{34}
\gend
$$
Next, apply the dual basis axiom for the factors 23 and 24 on the left hand-side, and on the factors 13 and 23 on the right hand-side (recall that we identify $ev$
and $\crta{ev}=ev\comp\tau$), to obtain:
$$
\gbeg{8}{5}
\got{1}{\crta{13}} \got{1}{13} \got{1}{\crta{24}} \got{1}{23} \got{3}{24} \gnl
\gev \gcl{2} \gcl{1} \gvac{1} \gcl{1} \gnl
\gvac{3} \glmptb\gnot{\hspace{-0,4cm}\tilde\Delta_2}\grmpb \glmptb\gnot{\hspace{-0,4cm}\tilde\Delta_4}\grmpb \gnl
\gvac{2} \gev \gcl{1} \gcl{1} \gcl{1} \gnl
\gvac{4} \gob{1}{12} \gob{1}{34} \gob{1}{14}
\gend=
\gbeg{8}{5}
\got{1}{\crta{24}} \got{1}{24} \got{1}{13} \got{3}{23} \got{1}{\crta{13}} \gnl
\gev \gcl{1} \gvac{1} \gcl{1} \gvac{1} \gcl{2} \gnl
\gvac{2} \glmptb\gnot{\hspace{-0,4cm}\tilde\Delta_1}\grmpb \glmptb\gnot{\hspace{-0,4cm}\tilde\Delta_3}\grmpb \gnl
\gvac{2} \gcl{1} \gcl{1} \gcl{1} \gev \gnl
\gvac{2} \gob{1}{14} \gob{1}{12} \gob{1}{34}
\gend
$$
Finally, compose from ``above'' with equivalent functors $\coev_{13}$ and  $\coev_{24}$. By \coref{coev inverse ev} we have $\crta\ev\comp\coev=\Id$, so we
get equivalently:
$$
\gbeg{6}{6}
\got{1}{23}  \gnl
\gcl{1} \gwdb{3} \gnl
\gibr \gvac{1} \gcl{1} \gnl
\gcl{1} \glmptb\gnot{\hspace{-0,4cm}\tilde\Delta_2}\grmpb \glmptb\gnot{\hspace{-0,4cm}\tilde\Delta_4}\grmpb \gnl
\gev \gcl{1} \gcl{1} \gcl{1} \gnl
\gvac{2} \gob{1}{12} \gob{1}{34} \gob{1}{14}
\gend=
\gbeg{8}{6}
\got{5}{} \got{1}{23} \gnl
\gvac{1} \gwdb{4} \gcl{1} \gnl
\gvac{1} \gcl{1} \gvac{2} \gbr \gnl
\gvac{1} \glmptb\gnot{\hspace{-0,4cm}\tilde\Delta_1}\grmpb \glmpb\gnot{\hspace{-0,4cm}\tilde\Delta_3}\grmptb \gcl{1} \gnl
\gvac{1} \gcl{1} \gcl{1} \gcl{1} \gev \gnl
\gvac{1} \gob{1}{14} \gob{1}{12} \gob{1}{34}
\gend
$$
Applying the dual basis axiom for 24 on the left hand-side and 13 on the right hand-side, we obtain equivalently: 
$$
\gbeg{3}{6}
\got{1}{} \got{1}{23} \gnl
\gvac{1} \gcl{1} \gnl
\gvac{1} \glmptb\gnot{\hspace{-0,4cm}\tilde\Delta_2}\grmpb \gnl
\glmpb\gnot{\hspace{-0,4cm}\tilde\Delta_4}\grmptb \gcl{2} \gnl
\gcl{1} \gcl{1} \gnl
\gob{1}{34} \gob{1}{14} \gob{1}{12}
\gend=
\gbeg{3}{6}
\got{1}{23} \gnl
\gcl{1} \gnl
\glmptb\gnot{\hspace{-0,4cm}\tilde\Delta_3}\grmpb \gnl
\gcl{2} \glmptb\gnot{\hspace{-0,4cm}\tilde\Delta_1}\grmpb \gnl
\gvac{1} \gcl{1} \gcl{1} \gnl
\gob{1}{34} \gob{1}{14} \gob{1}{12}
\gend
$$
When we apply this equation to $M_2\in\M_2$ we obtain:
$$(\tilde\Delta_4\Del_{\C^{\Del 4}} \M_{12})\tilde\Delta_2(M_{22})\iso(\M_{34}\Del_{\C^{\Del 4}}\tilde\Delta_1)\tilde\Delta_3(M_{22})$$
which by \equref{3H} is equal to:
$$\tilde\Delta_4(M_{(1)32})\Del_{\C^{\Del 4}} M_{(2)12}\iso M_{(1)33}\Del_{\C^{\Del 4}}\tilde\Delta_1(M_{(2)13})$$
which we can read as:
$$\tilde\Delta_4(M_{(1)24})\Del_{\C^{\Del 4}} M_{(2)12}\iso M_{(1)34}\Del_{\C^{\Del 4}}\tilde\Delta_1(M_{(2)21})$$
and then again by \equref{3H} interpret as:
$$M_{{(1)}_{(1)34}}\Del_{\C^{\Del 4}} M_{(1)_{(2)14}}\Del_{\C^{\Del 4}} M_{(2)12}\iso M_{(1)34}\Del_{\C^{\Del 4}}M_{(2)_{(1)31}}\Del_{\C^{\Del 4}}M_{(2)_{(2)11}}.$$
The above is an identity in
$\Fun_{\C^{\Del 4}}(\M_{23}, \M_{34}\Del_{\C^{\Del 4}}\M_{14}\Del_{\C^{\Del 4}}\M_{12})$
that by \equref{4H} corresponds to the identity:
$$M_{{(1)}_{(1)}}\Del_{\C} M_{(1)_{(2)}}\Del_{\C} M_{(2)}\iso M_{(1)}\Del_{\C}M_{(2)_{(1)}}\Del_{\C}M_{(2)_{(2)}}.$$
This is precisely the coassociativity condition for $\Delta:\M\to\M\Del_{\C}\M$.

\qed\end{proof}

\begin{rem} \rmlabel{Delta tilde za CxC}
Observe that if $\M=\C\Del\C$, then $\tilde\Delta((X\Del Y)_2)=\tilde\Delta(X\Del I\Del Y)=(X\Del I\Del I)\Del_{\C^{\Del 3}}(I\Del I\Del Y)\simeq
X\Del I\Del Y$, 
hence $\tilde\Delta\simeq\Id$.
\end{rem}

\begin{rem}\rmlabel{comput can cor}
Let $\N\in \dul{\Pic}(\C)$. Then $\delta_1(\N)=(\C\Del\N)\Del_{\C^{\Del 2}} (\N^{op}\Del\C)\simeq
\N^{op}\Del \N=\Can(\N;\C)$ with the coassociative comultiplication functor as in \exref{can}, \exref{canonical cor}.
We then have: $\delta_1(\N)_2\simeq\N^{op}\Del\C\Del\N$. Observe that there is an equivalence:
$$(\N^{op}\Del \N)_3\Del_{\C^{\Del 3}} (\N^{op}\Del \N)_1=(\N^{op}\Del \N\Del \C)\Del_{\C^{\Del 3}} (\C\Del \N^{op}\Del \N)
\simeq \N^{op}\Del \C \Del \N=(\N^{op}\Del \N)_2$$
where we applied \leref{Weq} and it involves $\crta\ev$ for $\N$ in the middle component of the threefold Deligne tensor product.
This makes $\tilde\Delta:(\N^{op}\Del \N)_2\to (\N^{op}\Del \N)_3\Del_{\C^{\Del 3}} (\N^{op}\Del \N)_1\simeq(\N^{op}\Del \N)_2$ an equivalence
(recall from \leref{Delta tilde} how $\tilde\Delta$ is defined). In particular we may write: $\tilde\Delta^{-1} \simeq\N^{op}\Del\ev_{\N}\Del\N
\simeq\ev_{\N_{13}}\Del_{\C^{\Del 3}}(\N^{op}\Del\C\Del\N)$.

Furthermore, it is:
\begin{eqnarray*}
&&\hspace*{-2cm}
\delta_2\delta_1(\N)=\delta_1(\N)_1 \Del_{\C^{\Del 3}} \delta_1(\N)_3 \Del_{\C^{\Del 3}} \delta_1(\N^{op})_2\\
&=&
\N_{11}\Del_{\C^{\Del 3}}\N^{op}_{21}\Del_{\C^{\Del 3}}
\N_{13}\Del_{\C^{\Del 3}}\N^{op}_{23}\Del_{\C^{\Del 3}}
\N^{op}_{12}\Del_{\C^{\Del 3}}\N_{22}\\
&=&
\N_{12}\Del_{\C^{\Del 3}}\N^{op}_{13}\Del_{\C^{\Del 3}}
\N_{13}\Del_{\C^{\Del 3}}\N^{op}_{23}\Del_{\C^{\Del 3}}
\N^{op}_{12}\Del_{\C^{\Del 3}}\N_{23}
\end{eqnarray*}
and $\lambda_\N=
\ev_{\N_{12}}\Del_{\C^{\Del 3}}\ev_{\N_{13}}\Del_{\C^{\Del 3}}\ev_{\N_{23}}.$
Putting $\tilde\Delta$ in the formula \equref{alfa-delta}, we get:
$$\begin{array}{lll}
\alfa &=& \crta\coev^{-1}_{\delta_1(\N)_2}\comp(\tilde\Delta^{-1}\Del_{\C^{\Del 3}}\delta_1(\N)_2^{op})\\
&=& \ev_{\delta_1(\N)_2}\comp(\ev_{\N_{13}}\Del_{\C^{\Del 3}}(\N^{op}\Del\C\Del\N)\Del_{\C^{\Del 3}}(\N^{op}\Del\C\Del\N))\\
&=& \lambda_{\N},
\end{array}$$ \vspace{-1cm}
\begin{equation} \eqlabel{alfa can}
\end{equation}
where we applied \coref{coev inverse ev}.
\end{rem}

\par\medskip

\begin{lma} \lelabel{cobound coring}
Let 
$\M\in\dul\Pic(\C\Del\C)$,
$\Delta:\M\to\M\Del_{\C}\M, \tilde\Delta: \M_2\to \M_3\Del_{\C^{\Del 3}}\M_1$ and
$\alpha: \M_3\Del_{\C^{\Del 3}}\M_1\Del_{\C^{\Del 3}}\M_2^{op} \to \C^{\Del 3}$ be as in \leref{alfa-delta}, and let $\N\in\dul\Pic(\C)$. Then there is an equivalence
functor $\M\simeq\N^{op}\Del \N=\Can(\N, \C)$ of $\C$-bimodule categories with coassociative comultiplication functors if and only if
$(\M, \alpha)\iso(\delta_1(\N), \lambda_{\N})$ in $\dul Z^2(\C, \dul\Pic)$.
\end{lma}

\begin{proof}
Set $\Delta', \tilde{\Delta'}$ and $\alpha'$ for the corresponding functors for $\delta_1(\N)$. By \equref{alfa can} we know that
$\alpha'=\lambda_{\N}$.
There is an equivalence $F:\M\to\delta_1(\N)$ of $\C$-bimodule categories with coassociative
comultiplication functors if and only if $F$ is $\C$-bilinear and the following diagram commutes:
\begin{eqnarray}
\scalebox{0.88}{\bfig
 \putmorphism(-50,500)(1,0)[\M`\M\Del_{\C} \M`\Delta]{770}1a
 \putmorphism(-50,0)(1,0)[\delta_1(\N)`\delta_1(\N)\Del_{\C}\delta_1(\N).`\Delta']{800}1a
\putmorphism(-40,500)(0,-1)[\phantom{B\ot B}``F]{480}1l
\putmorphism(700,500)(0,-1)[\phantom{B\ot B}``F\Del_{\C} F]{480}1r
\efig}
\end{eqnarray}
By the adjunction isomorphism from \leref{Delta tilde} this is equivalent to commutativity of
the diagram
\begin{eqnarray}
\scalebox{0.88}{\bfig
 \putmorphism(-50,500)(1,0)[\M_2`\M_3\Del_{\C^{\Del 3}}\M_1`\tilde{\Delta}]{870}1a
 \putmorphism(-50,0)(1,0)[\delta_1(\N)_2`\delta_1(\N)_3\Del_{\C^{\Del 3}}\delta_1(\N)_1`\tilde{\Delta}']{900}1a
\putmorphism(-40,500)(0,-1)[\phantom{B\ot B}``F_2]{480}1l
\putmorphism(760,500)(0,-1)[\phantom{B\ot B}``F_3\Del_{\C^{\Del 3}}F_1]{480}1r
\efig}
\end{eqnarray}
where $F$ is $\C$-bilinear.
This, in turn, is equivalent to commutativity of the right square in the next diagram
\begin{eqnarray}
\hspace{-0,6cm}
\scalebox{0.88}{\bfig
\putmorphism(-30,500)(1,0)[\C^{\Del 3}`\M_2\Del_{\C^{\Del 3}}\M_2^{op}`\crta\coev_{\M_2}]{940}1a
\putmorphism(-50,0)(1,0)[\C^{\Del 3}`{\delta_1(\N)}_2\Del_{\C^{\Del 3}}\delta_1(\N)^{op}_2`\crta\coev_{\delta_1(\N)_2}]{970}1a
\putmorphism(-40,500)(0,-1)[\phantom{B\ot B}``=]{480}1l
\putmorphism(850,500)(0,-1)[\phantom{B\ot B}``F_2\Del_{\C^{\Del 3}} (F_2^{op})^{-1}]{480}1r
\putmorphism(1220,500)(1,0)[`\M_3 \Del_{\C^{\Del 3}}\M_1\Del_{\C^{\Del 3}} \M^{op}_2`\tilde{\Delta}\Del_{\C^{\Del 3}}\M^{op}_2]{1720}1a
\putmorphism(1370,0)(1,0)[`\delta_1(\N)_3\Del_{\C^{\Del 3}}\delta_1(\N)_1 \Del_{\C^{\Del 3}} {\delta_1(\N)}^{op}_2.`\tilde{\Delta}' \Del_{\C^{\Del 3}} {\delta_1(\N)}^{op}_2]{1680}1a
\putmorphism(2680,500)(0,-1)[\phantom{B\ot B}`` F_3\Del_{\C^{\Del 3}} F_1 \Del_{\C^{\Del 3}}(F^{op}_2)^{-1}]{480}1r
\efig}
\end{eqnarray}
Observe that the left square is also commutative. Now, commutativity of the full diagram
is equivalent to $\alpha'\comp \delta_2(F)=\alpha$, meaning that $(\M, \alpha)\iso(\delta_1(\N), \lambda_{\N})$.
\qed\end{proof}

\begin{lma} \lelabel{Azumaya corings}
Let 
$\M\in\dul\Pic(\C\Del\C)$,
$\Delta:\M\to\M\Del_{\C}\M, \tilde\Delta: \M_2\to \M_3\Del_{\C^{\Del 3}}\M_1$ and
$\alpha: \M_3\Del_{\C^{\Del 3}}\M_1\Del_{\C^{\Del 3}}\M_2^{op} \to \C^{\Del 3}$ be as in \leref{alfa-delta}. If $\Delta$ is coassociative (or,
equivalently, $(\M, \alpha)\in\dul Z^2(\C, \dul\Pic)$), then there is an equivalence functor $\M\Del\C\simeq\Can(\M; \C\Del\C/\C)$ of bimodule
categories with coassociative comultiplication functors.
\end{lma}

\begin{proof}
If $(\M, \alpha)\in Z^2(\C,\dul\Pic)$, by \leref{extended coc} we have: $(\M\Del\C, \alpha\Del\C)\iso (\delta'_{1}(\M), \lambda'_{\M})$
in $B^2(\C\Del \C/\C,\dul\Pic)$. Then there is an equivalence $\M\Del\C\simeq \delta'_1(\M)= ((\C\Del\C)\Del_{\C}\M)\Del_{(\C\Del\C)_{\C}(\C\Del\C)}
(\M^{op}\Del_{\C}(\C\Del\C))\simeq \M^{op}\Del_{\C}\M=\Can(\M; \C\Del\C/\C)$ of $\C\Del\C$-bimodule categories. By \leref{cobound coring}
this equivalence is compatible with the comultiplication functors.
\qed\end{proof}

\bigskip

\begin{defn}
Let $\C$ be a symmetric finite tensor category. An invertible $\C$-bimodule category $\M$ with a coassociative $\C$-bimodule functor $\Delta:\M\to\M\Del_{\C}\M$
is an invertible quasi $\C$-coring category. If moreover the corresponding $\C^{\Del 3}$-module functor $\tilde\Delta: \M_2\to \M_3\Del_{\C^{\Del 3}}\M_1$
(from \leref{Delta tilde}) is an equivalence, we will say that $\M$ is an {\em Azumaya quasi $\C$-coring category}.
\end{defn}

\begin{rem}
The name {\em Azumaya} in the above definition is inherited from Azumaya corings introduced in \cite[Theorem 3.4]{CF}. It has to do with the fact that
in the case of Azumaya corings 
the coassociativity of the map $\Delta$ is equivalent to an isomorphism analogous to the equivalence of categories in \leref{Azumaya corings}. When working
with Azumaya coring categories, the coassociativity only implies the mentioned equivalence. The reason for loosing the ``if'' part in the claim is the missing
of the ``faithful flatness'' condition, as we announced in the introduction.
\end{rem}

\begin{ex}
From the definition  and due to \rmref{comput can cor} it is clear that the canonical quasi coring categories $\Can(\N; \C)=\N^{op}\Del \N$ with $\N\in\dul{\Pic}(\C)$
are Azumaya quasi $\C$-coring categories. Similarly, by \rmref{Delta tilde za CxC} and \exref{canonical cor} we have that $\C\Del\C$ is an  Azumaya $\C$-coring category.
\end{ex}

We will denote by $\dul{AzQCor}(\C)$ the category of Azumaya quasi $\C$-coring categories and their equivalences. We have:

\begin{prop}
For a symmetric finite tensor category $\C$ the category \\ $(\dul{AzQCor}(\C),\Del_{\C^{\Del 2}},\Can(\C;\C))$ is symmetric monoidal.
\end{prop}

\begin{proof}
Let $(\M,\Delta)$ and $(\M',\Delta')$ be two Azumaya quasi $\C$-coring categories. We define $\tilde{D}$ to be the composition:
$$ \bfig
\putmorphism(0, 500)(1, 0)[(\M\Del_{\C^{\Del 2}}\M')_2 = \M_2\Del_{\C^{\Del 3}}\M'_2`
   \M_3\Del_{\C^{\Del 3}}\M_1 \Del_{\C^{\Del 3}} \M'_3\Del_{\C^{\Del 3}}\M'_1`
  \tilde{\Delta}\Del_{\C^{\Del 3}} \tilde{\Delta}']{2000}1a
\putmorphism(-200, 250)(1, 0)[\phantom{(\M\Del_{\C^{\Del 2}}\M')_2 = \M_2\Del_{\C^{\Del 3}}\M'_2}`
   \M_3\Del_{\C^{\Del 3}}\M'_3 \Del_{\C^{\Del 3}} \M_1\Del_{\C^{\Del 3}}\M'_1`
   \M_3\Del_{\C^{\Del 3}} \tau\Del_{\C^{\Del 3}} \M'_1]{2400}1a
\putmorphism(70, 0)(1, 0)[\phantom{(\M\Del_{\C^{\Del 2}}\M')_2 = \M_2\Del_{\C^{\Del 3}}\M'_2}`
  (\M\Del_{\C^{\Del 2}}\M')_3 \Del_{\C^{\Del 3}}(\M\Del_{\C^{\Del 2}}\M')_1.`=]{1930}1a
\efig
$$
The functor corresponding to $\tilde{D}$ by the isomorphism in \leref{Delta tilde} is the comultiplication functor on $\M\Del_{\C^{\Del 2}}\M'$:
$$D:\ \M\Del_{\C^{\Del 2}}\M' \to (\M\Del_{\C^{\Del 2}}\M')\Del_{\C} (\M\Del_{\C^{\Del 2}}\M').$$
From the definition of $\tilde{D}$ we have:
$$\tilde{D}(M\Del_{\C^{\Del 2}}M')_2= (M_{(1)}\Del_{\C^{\Del 2}} M'_{(1)})_{3}
\Del_{\C^{^{\Del 2}}} (M_{(2)}\Del_{\C^{\Del 2}} M'_{(2)})_{1},$$
and this yields:
\begin{equation} \eqlabel{Delta grande}
D(M\Del_{\C^{\Del 2}}M')=(M_{(1)}\Del_{\C^{\Del 2}} M'_{(1)}) \Del_{\C} (M_{(2)}\Del_{\C^{\Del 2}} M'_{(2)}).
\end{equation}
Now it is easy to see that $D$ is coassociative.
\qed\end{proof}

Furthermore, we have:

\begin{thm} 
Let $\C$ be a symmetric finite tensor category and let $(\M,\Delta)$ and $(\M',\Delta')$ be Azumaya quasi $\C$-coring categories.
Consider the corresponding $(\M,\alpha),(\M',\alpha')\in\dul{Z}^2(\C,\dul{\Pic})$. Let $F:\ \M\to \M'$ be an equivalence
in $\dul{\Pic}(\C\Del\C)$. Then $F$ is an equivalence of quasi coring categories if and only
if $F$ defines an isomorphism in $\dul{Z}^2(\C,\dul{\Pic})$.
\end{thm}

\begin{proof}
The proof is analogous to that of \leref{cobound coring} with $\M'= \delta_1(\N)$ and $\alpha'=\lambda_{\N}$.
\qed\end{proof}

Consecuently, there is a monoidal isomorphism of categories:

\begin{equation} \eqlabel{cats Az and coc}
\HH:\ \dul{AzQCor}(\C)\to \dul{Z}^2(\C,\dul{\Pic}).
\end{equation}

\bigskip

Let $K_0 [\dul{AzQCor}(\C)]$ be the Grothendieck group of Azumaya quasi $\C$-coring categories and let $\Can(\C)$ denote its
subgroup consisting of equivalence classes of quasi $\C$-coring categories represented by a 
category $\Can(\N;\C)$ for some $\N\in \dul{\Pic}(\C)$. Set
$$AzQCor(\C)=K_0 [\dul{AzQCor}(\C)]/\Can(\C)$$
for the {\em group of Azumaya quasi $\C$-coring categories}.

\begin{cor} \colabel{Brauer}
There is an isomorphism of abelian groups
$$\chi: AzQCor(\C)\to H^2(\C,\dul{\Pic}).$$
Consequently, there is an exact sequence
\begin{eqnarray}\eqlabel{2.2.2.1b}
1&\longrightarrow&H^2(\C,\Inv)\stackrel{\alpha_2}{\longrightarrow}H^1(\C,\dul{\Pic})\stackrel{\beta_2}{\longrightarrow}H^1(\C,{\Pic})\\
&\stackrel{\gamma_2}{\longrightarrow}&H^3(\C,\Inv)\stackrel{\omega}{\longrightarrow}AzQCor(\C)\stackrel{\beta_3}{\longrightarrow}H^2(\C,{\Pic})\nonumber\\
&\stackrel{\gamma_3}{\longrightarrow}&H^4(\C,\Inv)\stackrel{\alpha_4}{\longrightarrow}H^3(\C,\dul{\Pic})\stackrel{\beta_4}{\longrightarrow}H^3(\C,{\Pic})\nonumber\\
&\stackrel{\gamma_4}{\longrightarrow}&\cdots \nonumber
\end{eqnarray}
\end{cor}

\subsection{The connecting map on the second level}

Here we will describe the map $\omega$ appearing in the second level of the above sequence.

\begin{lma}
Let $\C$ be a symmetric finite tensor category and let $[X]\in H^3(\C,\Inv)$ with $X=X^1\Del X^2\Del X^3$.
There is a structure of a $\C$-coring category on $(\C\Del\C)_X$, where $(\C\Del\C)_X=\C\Del\C$ as tensor categories with the $\C$-bilinear
comultiplication and the counit functors
$\Delta_X: (\C\Del\C)_X\to (\C\Del\C)_X\Del_{\C}(\C\Del\C)_X\simeq\C\Del\C\Del\C$ and $\varepsilon_X: (\C\Del\C)_X\to\C$
defined via $\Delta_X(I\Del I)=X^1\Del X^2\Del X^3$ and $\varepsilon_X(I\Del I)=\vert X\vert^{-1}$ on objects.
For $f: Y\Del Z\to Y'\Del Z'$ in $\C\Del\C$ we define $\Delta_X(f)$ by the commuting diagram
\begin{equation} \eqlabel{Delta_X on f}
\scalebox{0.84}{
\bfig 
\putmorphism(0,500)(1,0)[YX^1\Del X^2\Del X^3Z ` Y'X^1\Del X^2\Del X^3Z' ` \Delta_X(f)]{3100}1a
\putmorphism(0,0)(1,0)[YX^1\Del X^2\Del ZX^3` Y\Del I\Del Z ` m(X^{-1})]{1050}1b
\putmorphism(950,0)(1,0)[\phantom{(X \ot (Y \ot U))}` Y'\Del I\Del Z'` f_2]{1000}1b
\putmorphism(2200,0)(1,0)[` Y'X^1\Del X^2\Del Z'X^3 ` m(X)]{910}1b
\putmorphism(60,500)(0,1)[``\Id\Del\Id\Del\Phi_{X^3, Z}]{500}1l
\putmorphism(3100,500)(0,1)[``\Id\Del\Id\Del\Phi_{Z', X^3}]{500}{-1}r
\efig}
\end{equation}
and $\varepsilon_X(f)= \vert X\vert^{-1}\ot\varepsilon(f)$, where $\varepsilon(f)$ is from \leref{CC coring}.

\end{lma}

\begin{proof}
The functors $\Delta_X$ and $\varepsilon_X$ extend to $\C$-bilinear functors as follows:
$\Delta_X(Y\Del Z)=(YX^1\Del X^2)\Del_{\C}(I\Del X^3Z)\simeq YX^1\Del X^2\Del X^3Z$ and $\varepsilon_X(Y\Del Z)=Y\vert X\vert^{-1}Z$.
The coassociativity of $\Delta_X$ means that there is a $\C$-bimodule natural isomorphism between the functors: $F=(\Id_{\C\Del\C}\Del_{\C}\Delta_X)\Delta_X$ and
$G=(\Delta_X\Del_{\C}\Id_{\C\Del\C})\Delta_X$. Applying them to an object $Y\Del Z\in\C\Del\C$ we obtain:
$(\Id_{\C\Del\C}\Del_{\C}\Delta_X)((YX^1\Del X^2)\Del_{\C}(I\Del X^3Z))=(YX^1\Del X^2)\Del_{\C}(I\crta{X^1}\Del\crta{X^2})\Del_{\C}(I\Del \crta{X^3}X^3Z))
\simeq YX^1\Del X^2\crta{X^1}\Del\crta{X^2}\Del\crta{X^3}X^3Z$, where $X=X^1\Del X^2\Del X^3=\crta{X^1}\Del\crta{X^2}\Del\crta{X^3}$.
On the other hand, $(\Delta_X\Del_{\C}\Id_{\C\Del\C})((YX^1\Del X^2)\Del_{\C}(I\Del X^3Z))=(YX^1\crta{X^1}\Del\crta{X^2})\Del_{\C}(I\Del\crta{X^3} X^2)
\Del_{\C}(I\Del X^3Z)\simeq YX^1\crta{X^1}\Del\crta{X^2}\Del\crta{X^3} X^2\Del X^3Z$. Let $\alpha: F\to G$ be the following transformation:
\begin{equation} \eqlabel{nat tr X-coring}
\scalebox{0.8}{
\bfig \hspace{-0,4cm}
\putmorphism(0,500)(1,0)[YX^1\Del X^2\crta X^1\Del\crta X^2\Del \crta X^3X^3Z=F(Y\Del Z) ` G(Y\Del Z)=YX^1\crta X^1\Del\crta X^2\Del\crta X^3X^2\Del X^3Z ` \alpha(Y\Del Z)]{2700}1a
\putmorphism(-300,0)(1,0)[YX^1\Del \crta X^1X^2\Del\crta X^2\Del \crta X^3X^3Z` YX^1\crta{X^1}\Del\crta{X^2}\Del X^2\crta{X^3}\Del X^3Z ` =]{3280}1b
\putmorphism(-280,500)(0,1)[``\Id\Del\Phi_{X^2,\crta X^1}\Del\Id\Del\Id]{500}1r
\putmorphism(3000,500)(0,1)[``\Id\Del\Id\Del\Phi_{X^2,\crta{X^3}}\Del\Id]{500}{-1}l
\efig}
\end{equation}
where the bottom line is equality by the 3-cocycle condition \equref{3-cocycle}. It is clear that $\alpha(Y\Del Z)$ is an isomorphism for all $Y\Del Z\in\C\Del\C$.
Observe that by definition \equref{Delta_X on f} the morphism $F(f)$
for $f: Y\Del Z\to Y'\Del Z'$ in $\C\Del\C$ is given by the commuting diagram $\langle 1\rangle$ (and the outer diagram) in:
\begin{equation} \eqlabel{F(f)}
\scalebox{0.8}{
\bfig \hspace{-0,4cm}
\putmorphism(0,500)(1,0)[YX^1\Del X^2\crta X^1\Del\crta X^2\Del \crta X^3X^3Z=F(Y\Del Z) ` F(Y'\Del Z')=Y'X^1\Del X^2\crta X^1\Del\crta X^2\Del \crta X^3X^3Z' ` F(f)]{2700}1a
\putmorphism(-280,500)(0,1)[`YX^1\Del X^2\crta X^1\Del\crta X^2\Del Z\crta X^3X^3`\Id\Del\Id\Del\Id\Del\Phi_{\crta X^3X^3,Z}]{500}1r
\putmorphism(-280,0)(0,1)[`(Y\Del I\Del I\Del Z)\Del_{\C^4}W` =]{300}1r
\putmorphism(280,-300)(1,0)[`(Y'\Del I\Del I\Del Z')\Del_{\C^4}W` f_{23}\Del_{\C^4}W]{2500}1a
\putmorphism(-280,-300)(1,-1)[`Y\Del I\Del I\Del Z` m(W^{-1})]{500}1r
\putmorphism(600,-810)(1,0)[`Y'\Del I\Del I\Del Z'` f_{23}]{1800}1b
\putmorphism(2600,-680)(1,1)[`` m(W)]{240}1l
\putmorphism(3000,0)(0,1)[Y'X^1\Del X^2\crta X^1\Del\crta X^2\Del Z'\crta X^3X^3`` =]{300}{-1}r
\putmorphism(3000,500)(0,1)[``\Id\Del\Id\Del\Id\Del\Phi_{Z', \crta{X^3}X^3}]{500}{-1}l
\put(1300,100){\fbox{1}}
\put(1300,-600){\fbox{2}}
\efig}
\end{equation}
where $W=X^1\Del X^2\crta X^1\Del\crta X^2\Del \crta X^3X^3$ (observe that the trapeze $\langle 2\rangle$ above commutes automatically). The morphism $G(f)$ is similarly defined, with
$W$ replaced by $U=X^1\crta X^1\Del\crta X^2\Del\crta X^3X^2\Del X^3$. Set $a_X=\alpha(I\Del I)$. To see that $\alpha$ is a natural transformation observe the following diagram:
\begin{equation*} \eqlabel{}
\scalebox{0.8}{
\bfig 
\putmorphism(0,500)(1,0)[Y\crta\ot W\crta\ot Z=F(Y\Del Z) ` G(Y\Del Z)=Y\crta\ot U\crta\ot Z ` \alpha(Y\Del Z)]{2700}1a
\putmorphism(-280,500)(0,1)[`` \Id\Del\Id\Del\Id\Del\Phi_{\crta X^3X^3,Z}]{500}1r  %
\putmorphism(0,0)(1,0)[(Y\Del I\Del I\Del Z)\Del_{\C^4}W`(Y\Del I\Del I\Del Z)\Del_{\C^4}U` \Id\Del_{\C^4}a_X]{2800}1a
\putmorphism(3000,500)(0,1)[``\Id\Del\Id\Del\Id\Del\Phi_{X^3,Z}]{500}1l
\putmorphism(3000,0)(0,1)[``f_{23}\Del_{\C^4}U]{500}1l
\putmorphism(3000,-500)(0,1)[``\Id\Del\Id\Del\Id\Del\Phi_{Z',X^3}]{500}1l
\putmorphism(-280,0)(0,1)[`` f_{23}\Del_{\C^4}W]{500}1r  %
\putmorphism(0,-500)(1,0)[(Y'\Del I\Del I\Del Z')\Del_{\C^4}W`(Y'\Del I\Del I\Del Z')\Del_{\C^4}U` \Id\Del_{\C^4}a_X]{2800}1a
\putmorphism(-280,-500)(0,1)[`` \Id\Del\Id\Del\Id\Del\Phi_{Z',\crta X^3X^3}]{500}1r  %
\putmorphism(0,-1000)(1,0)[Y'\crta\ot W\crta\ot Z'=F(Y'\Del Z') ` G(Y'\Del Z')=Y'\crta\ot U\crta\ot Z'` \alpha(Y'\Del Z')]{2700}1a
\efig}
\end{equation*}
All the three inner rectangulars above clearly commute, and the commutativity of the outer diagram assures the naturality of $\alpha$. It is direct to see that
$\alpha$ is a $\C$-bilinear transformation.

We next prove the compatibility condition for the comultiplication and the counit functors:
$(\varepsilon_X\Del_{\C}\Id)\Delta_X(Y\Del Z)=(\varepsilon_X\Del_{\C}\Id)((YX^1\Del X^2)\Del_{\C}(I\Del X^3Z))=YX^1\vert X\vert^{-1} X^2\Del_{\C}(I\Del X^3Z)
\simeq YX^1\vert X\vert^{-1} X^2\Del X^3Z$ and $(\Id\Del_{\C}\varepsilon_X)\Delta_X(Y\Del Z)=(\Id\Del_{\C}\varepsilon_X)((YX^1\Del X^2)\Del_{\C}(I\Del X^3Z))=
(YX^1\Del X^2)\Del_{\C} \vert X\vert^{-1} X^3Z\simeq YX^1\Del X^2 \vert X\vert^{-1} X^3Z$. Consider the transformations $l$ and $r$ defined below:
$$\scalebox{0.8}{
\bfig 
\putmorphism(100,500)(2,-1)[ ` Y\Del Z ` ]{1000}1r
\putmorphism(100,530)(2,-1)[ `  ` l]{1000}0r
\putmorphism(450,0)(1,0)[` ` \iso]{480}1a
\putmorphism(0,500)(0,1)[YX^1\vert X\vert^{-1} X^2\Del X^3Z` Y\vert X\vert^{-1} X^1 X^2\Del X^3Z ` \Phi_{X^1,\vert X\vert^{-1}} X^1 X^2\Del\Id]{500}1l
\efig}\quad\textnormal{and}\quad
\scalebox{0.8}{
\bfig \hspace{-0,4cm}
\putmorphism(100,500)(2,-1)[ ` Y\Del Z ` ]{1000}1r
\putmorphism(100,530)(2,-1)[ `  ` r]{1000}0r
\putmorphism(450,0)(1,0)[` ` \iso]{480}1a
\putmorphism(0,500)(0,1)[YX^1\Del X^2 \vert X\vert^{-1} X^3Z ` YX^1\Del\vert X\vert^{-1} X^2X^3Z ` \Id\Del\Phi_{X^2,\vert X\vert^{-1}}X^3Z]{500}1l
\efig}$$
where the bottom isomorphisms are due to \leref{compatibility for coring}. They are clearly $\C$-bimodule natural isomorphisms.
\qed\end{proof}

The coring structure presented in the previous lemma is the main protagonist in the next result:

\begin{prop} \prlabel{connection}
Let $\C$ be a symmetric finite tensor category. 
There is a group morphism
$$\omega: H^3(\C,\Inv)\to AzQCor(\C)$$ 
given by $\omega([X])=[(\C\Del\C)_X]$ for $[X]\in H^3(\C,\Inv)$.
\end{prop}

\begin{proof}
Set $[X=X^1\Del X^2\Del X^3]\in H^3(\C,\Inv)$.
Let us prove that $\omega$ does not depend on the representatives of classes. If $X\sim Y$ in $ H^3(\C,\Inv)$, then there is $Z=Z^1\Del Z^2=\crta{Z^1}\Del\crta{Z^2}\in\Inv(\C\Del\C)$
such that $X\iso\tilde\delta Z\cdot Y$, or, equivalently,
\begin{equation}\eqlabel{homogeneous}
Z^1X^1\Del X^2\Del Z^2X^3\iso\crta{Z^1}Y^1\Del Z^1\crta{Z^2}Y^2\Del Z^2Y^3.
\end{equation}
In order to prove that $\omega([X])=\omega([Y])$, observe the following $\C$-bimodule functor: $F: (\C\Del\C)_Y\to(\C\Del\C)_X$, given by $F(I\Del I)=Z^1\Del Z^2$.
We find that $(F\Del_{\C}F)\Delta_Y(I\Del I)=(F\Del_{\C}F)((Y^1\Del Y^2)\Del_{\C}(I\Del Y^3))=(Y^1Z^1\Del Z^2Y^2)\Del_{\C}(\crta{Z^1}\Del\crta{Z^2}Y^3)
\simeq Y^1Z^1\Del Z^2Y^2\crta{Z^1}\Del\crta{Z^2}Y^3$, whereas $\Delta_X F(I\Del I)=\Delta_X(Z^1\Del Z^2)=(Z^1X^1\Del X^2)\Del_{\C}(I\Del X^3Z^2)\simeq
Z^1X^1\Del X^2\Del X^3Z^2$. Let $\gamma:\Delta_X\comp F\to (F\Del_{\C}F)\Delta_Y$ be the (obviously) $\C$-bimodule natural isomorphism given by the commuting diagram:
\begin{equation*} 
\scalebox{0.8}{
\bfig 
\putmorphism(0,500)(1,0)[UZ^1X^1\Del X^2\Del X^3Z^2V ` UY^1Z^1\Del Z^2Y^2\crta Z^1\Del\crta Z^2Y^3V ` \gamma(U\Del V)]{2700}1a
\putmorphism(0,0)(1,0)[UZ^1X^1\Del X^2\Del Z^2X^3V` UZ^1Y^1\Del \crta Z^1Z^2Y^2\Del\crta Z^2Y^3V ` \iso]{2700}1b
\putmorphism(-280,500)(0,1)[``\Id\Del\Id\Del\Phi_{X^3, Z^2}\ot V]{500}1r
\putmorphism(3000,500)(0,1)[``U\ot\Phi_{Z^1,Y^1}\Del\Phi_{\crta{Z^1},Z^2Y^2}\Del\Id]{500}{-1}l
\efig}
\end{equation*}
The bottom isomorphism is due to \equref{homogeneous}.
Then it is clear that $F: (\C\Del\C)_Y\to(\C\Del\C)_X$ defines an equivalence of quasi $\C$-coring categories, since $Z$ is invertible.

Finally we prove that $\omega$ is a group map. Take $[X], [Y]\in H^3(\C,\Inv)$.
We have: $\omega([X][Y])=\omega([XY])=[(\C\Del\C)_{XY}]$ and $\omega([X])\omega([Y])=[(\C\Del\C)_{X}][(\C\Del\C)_{Y}]=
[(\C\Del\C)_{X}\Del_{\C\Del\C}(\C\Del\C)_{Y}]$. The comultiplication functor for the latter coring category is given by \equref{Delta grande}, that is,
$D((I\Del I)\Del_{\C\Del\C}(I\Del I))=((X^1\Del X^2)\Del_{\C\Del\C}(Y^1\Del Y^2))\Del_{\C}((I\Del X^3)\Del_{\C\Del\C}(I\Del Y^3))\simeq
(X^1Y^1\Del X^2Y^2)\Del_{\C}(I\Del X^3Y^3)\simeq XY = \Delta_{XY}(I\Del I)$. Then it is clear that $\omega([X][Y])$ and $\omega([X])\omega([Y])$ determine the same
quasi coring category.
\qed\end{proof}

\begin{rem}
Given two coring categories $\M$ and $\N$ one would expect 
the counit functor on 
their product $\varepsilon:\M\Del_{\C\Del\C}\N\to\C$ to be 
given by $\varepsilon(M\Del N)=\varepsilon_{\M}(M)\varepsilon_{\N}(N)$. However, it is not clear how to prove the compatibility of the counit and the comultiplication
functors on $\M\Del_{\C\Del\C}\N$. This is even not clear in such a simple case as for the 3-cocycle twisted coring categories $(\C\Del\C)_X$.

Though, we do have that the map $\omega$ restricted to the set $H^3(\C,\Inv)$ produces proper $\C$-coring categories (the class of the coring $[(\C\Del\C)_X]$ - with the counit functor -
does not depend on the representative of the class of $X$ in the cohomology group). Indeed, with the notations as in the above proof we have that
$\varepsilon_X\comp F(U\Del V)=\varepsilon_X(UZ^1\Del Z^2V)=UZ^1\vert X\vert^{-1} Z^2V$
and $\varepsilon_Y(U\Del V)=U\vert Y\vert^{-1}V=U\vert X\vert^{-1}\vert Z\vert V$, so define the (obviously) $\C$-bimodule natural
isomorphism $\delta: \varepsilon_X\comp F\to\varepsilon_Y$ to be $\delta(U\Del V)=U\ot\Phi_{Z^1, \vert X\vert^{-1}}\ot Z^2V$.
\end{rem}

\bigskip

The map $\omega$ from \prref{connection} makes the following diagram commutative (up to an anti-group switch):
$$
\bfig
\putmorphism(100,470)(1,0)[H^3(\C, \Inv)`H^2(\C, \dul\Pic)`\alpha_3]{800}1a
\putmorphism(70,470)(1,-1)[``\omega]{380}1l
\putmorphism(150,470)(1,-1)[`AzQCor(\C)`]{380}0l
\putmorphism(700,235)(1,1)[``\chi]{100}1r
\efig
\qquad
\bfig
\putmorphism(100,470)(1,0)[[X]`(\C\Del\C, m(X))`\alpha_3]{800}2a
\putmorphism(70,470)(1,-1)[``\omega]{380}2l
\putmorphism(110,470)(1,-1)[`((\C\Del\C)_X, \Delta_X).`]{380}0l
\putmorphism(700,235)(1,1)[``\chi]{100}2r
\efig
$$
Here $\chi$ is the map from \coref{Brauer} and $\chi([(\M, \Delta)])=[(\M, \alpha(\tilde\Delta))]$, with $\alpha$ from \equref{alfa-delta}. To see this,
recall from \equref{3H} that $\tilde{\Delta}(M_2)=M_{(1)3}\Del_{\C^{\ot 3}} M_{(2)1}$, where $M_{(1)}\Del_{\C^{\ot 2}} M_{(2)}=\Delta(M)$.
Then it is immediate that $\tilde{\Delta}_X=m(X)$, so it is indeed a $\C^{\Del 3}$-module equivalence. Let us compute $\alpha(\tilde\Delta_X)$.
First of all observe that $\tilde\Delta_X((Y\Del Z)_2)=(YX^1\Del X^2\Del I)\Del_{\C^{\ot 3}} (I\Del I\Del X^3Z)$, then by \equref{alfa-delta} we have:
$$\begin{array}{rl}
\alpha^{-1}(\tilde\Delta_X)(I^{\Del 3})\hskip-1em&\hspace{-0,2cm} =(\tilde\Delta_X\Del_{\C^{\Del 3}}(\C\Del\C)_2^{op})\comp\crta\coev_{(\C\Del\C)_2}(I^{\Del 3})\\
&=
(\tilde\Delta_X\Del_{\C^{\Del 3}}(\C\Del\C)_2^{op})(\displaystyle\bigoplus_{\substack{i\in J\\j\in J}} (V_i\Del I\Del V'_j)\Del_{\C^{\ot 3}} (W_i\Del I\Del W'_j))\\
&=\displaystyle\bigoplus_{\substack{i\in J\\j\in J}} (V_iX^1\Del X^2\Del I)\Del_{\C^{\ot 3}} (I\Del I\Del X^3 V'_j)\Del_{\C^{\ot 3}} (W_i\Del I\Del W'_j)\\
&\simeq \displaystyle\bigoplus_{\substack{i\in J\\j\in J}} (V_iX^1\Del X^2\Del X^3 V'_j)\Del_{\C^{\ot 3}} (W_i\Del I\Del W'_j)\\
&\simeq m(X)(\displaystyle\bigoplus_{\substack{i\in J\\j\in J}} (V_i\Del I\Del V'_j)\Del_{\C^{\ot 3}} (W_i\Del I\Del W'_j)) \\
& \stackrel{\ev_{(\C\Del\C)_2}}{\simeq} m(X)(\displaystyle\bigoplus_{\substack{i\in J\\j\in J}} \u\Hom_{\C}(V_i, W_i)\Del I\Del \u\Hom_{\C}(V'_j, W'_j)) \\
& \simeq m(X)
\end{array}$$
where the last identity is due to \coref{coev inverse ev}. 
Then $\alpha(\tilde\Delta_X)\simeq m((X)^{-1})$, and we obtain $\chi([((\C\Del\C)_X, \Delta_X)])=[((\C\Del\C)_X, m(X^{-1}))]$.

\section{The full group of Azumaya quasi coring categories} \selabel{Full group}

So far we have studied module categories over a symmetric finite tensor category $\C$. In the next two subsections we will investigate which relations
we have if we take another symmetric finite tensor category $\D$ into a consideration. 

\subsection{The colimit over symmetric finite tensor categories}

Let $\F:\C\to \D$ be a functor between symmetric finite tensor categories.
Then the induced functor $\F':\dul{AzQCor}(\C)\to \dul{AzQCor}(\D)$ is given by $\F'(\M)=\M\Del_{\C^{\Del 2}}\D\Del\D$. The comultiplication
functor on the latter is induced by those on $\M$ and $\D\Del\D$. That is,
$\Delta: \M\Del_{\C^{\Del 2}}\D\Del\D\to (\M\Del_{\C^{\Del 2}}\D\Del\D)\Del_{\D}(\M\Del_{\C^{\Del 2}}\D\Del\D)$ is given by:
$\Delta(M\Del_{\C^{\Del 2}}(D\Del D'))=(M_{(1)}\Del_{\C^{\Del 2}}(D\Del I))\Del_{\D}(M_{(2)}\Del_{\C^{\Del 2}}(I\Del D'))$ for $M\in\M, D,D'\in\D$.
It is direct to check that $\Delta$ is $\C\Del\C$-balanced, as $\C$ and $\D$ are symmetric. The left $\D$-module category structure on
$\M\Del_{\C^{\Del 2}}\D\Del\D$ is induced by that on $\D\Del\D$. The functor $\F'$ is such that the following diagram commutes:
$$\bfig
\putmorphism(0, 400)(1, 0)[\dul{AzQCor}(\C)`\phantom{\dul{Z}^2(\C,\dul{\Pic})}`\HH]{900}1a
\putmorphism(60, 400)(0, -1)[\phantom{\dul{Cor}^*(\C)}`\phantom{[M, Y]\ot M}`\F']{400}1l
\putmorphism(900, 400)(0, -1)[\dul{Z}^2(\C,\dul{\Pic})`\phantom{\dul{Z}^1(T/R,\dul{\Pic}).}`\F^*]{400}1r
\putmorphism(0, 0)(1, 0)[\dul{AzQCor}(\D)`\dul{Z}^2(\D,\dul{\Pic})`\HH]{900}1a
\efig
$$
where $\HH$ is the monoidal isomorphism from \equref{cats Az and coc} and $\F^*$ was defined in \equref{fun ind cat}.
After taking the corresponding Grothendieck groups and their quotients, we obtain a commutative diagram:
$$\bfig
\putmorphism(0, 400)(1, 0)[AzQCor(\C)`\phantom{H^1(S/R,\dul{\Pic})}`\cong]{900}1a
\putmorphism(60, 400)(0, -1)[\phantom{\dul{Cor}^*(\C)}`\phantom{[M, Y]\ot M}`\F'']{400}1l
\putmorphism(900, 400)(0, -1)[H^2(\C,\dul{\Pic})`\phantom{H^1(T/R,\dul{\Pic}).}`\F_*]{400}1r
\putmorphism(0, 0)(1, 0)[AzQCor(\D)`H^2(\D,\dul{\Pic}).`\cong]{900}1a
\efig
$$
This is to say that the group isomorphisms in \coref{Brauer} induces an isomorphism of functors
$$AzQCor(\bullet/vec)\cong H^2(\bullet/vec,\dul{\Pic}):\ FTC_{symm}\to\dul{Ab}$$
where the second functor is that from \equref{H2 functor} for $n=2$. Here we identify $AzQCor(\C)=AzQCor(\C/vec)$ for any $\C\in FTC_{symm}$.
So for the colimits we have:
\begin{equation}\eqlabel{colimit Hn}
\colim~ AzQCor(\bullet/vec)\cong \colim~ H^2(\bullet/vec,\dul{\Pic}).
\end{equation}

\subsection{The full group of Azumaya quasi coring categories} \sslabel{Full group}

In this subsection we will introduce the full group of Azumaya quasi coring categories and we will prove that it is the colimit of the {\em relative} groups
$AzQCor(\C/vec)$. For these constructions \prref{coh maps dul} is of great importance.

We define an {\em Azumaya quasi coring category over $vec$} as a pair $(\C,\M)$, where $\C$ is a symmetric finite tensor category, and $\M$ is an Azumaya quasi
$\C$-coring category. Given two Azumaya quasi coring categories $(\C,\M)$ and $(\D,\N)$ over $vec$, a morphism between them 
is a pair $(F,\F)$, where $F:\C\to\D$ is an equivalence of symmetric tensor categories and $\F:\ \M\to\N$ is a $\C$-bimodule category equivalence that preserves
the comultiplication functors.
Let $\dul{AzQCor}(vec)$ be the category of Azumaya quasi coring categories over $vec$. In what follows, the (canonical) Azumaya quasi coring categories of the form
$\delta_1(\N)$ for $\N\in\dul{\Pic}(\C)$ and $\C$ a symmetric finite tensor category, we will call {\em elementary}. 

\medskip

We start by an easy to prove result:

\begin{lma} 
Let $\C$ and $\D$ be tensor categories. If $\M\in\C\x\Mod$ and $\N\in\D\x\Mod$, then $\M\Del\N\in\C\Del\D\x\Mod$.
\end{lma}

\begin{lma} 
Suppose that $\C$ and $\D$ are braided tensor categories.
If $\M$ is an Azumaya quasi $\C$-coring category and $\N$ is an Azumaya quasi $\D$-coring category, then $\M\Del\N$ is an Azumaya quasi
$\C\Del\D$-coring category.

An analogous statement holds true for invertible coring categories.
\end{lma}

\begin{proof}
Since $\C$ and $\D$ are braided, so is $\C\Del\D$ (e.g. 
\cite[Section 2.1, equation (8)]{Femic2}).
We have that $\M\Del\N$ is indeed an invertible $\C\Del\D$-bimodule category:
$$(\M\Del\N)\Del_{\C\Del\D}(\M^{op}\Del\N^{op})\stackrel{\leref{Weq}}{\simeq} (\M\Del_{\C}\M^{op})\Del(\N\Del_{\D}\N^{op})\simeq\C\Del\D.$$
The functor $\Delta$ defined through the composition:
$$ \bfig
\putmorphism(0, 200)(1, 0)[\M\Del\N` (\M\Del_{\C}\M) \Del (\N\Del_{\D}\N)` \Delta_{\M}\Del\Delta_{\N}]{1250}1a
\putmorphism(-550, 0)(1, 0)[\phantom{(\M\Del\N` (\M\Del_{\C}\M) \Del (\N\Del\N)}`
  (\M\Del\N) \Del_{\C\Del\D}(\M\Del\N)`\simeq]{1850}1a
\efig
$$
is $\C\Del\D$-bilinear, because $\Delta_{\M}$ is $\C$-bilinear and $\Delta_{\N}$ is $\D$-bilinear.
For $M\in\M$ and $N\in\N$ it is $\Delta(M\Del N)=(M_{(1)}\Del N_{(1)}) \Del_{\C\Del\D} (M_{(2)}\Del N_{(2)})$.
Then it is clear that $\M\Del\N$ equipped with $\Delta$ is a quasi $\C\Del\D$-coring category.

Let $\tilde\Delta$ denote the corresponding functor for $\M\Del\N$. To see that it is an equivalence, we compute:
$$\begin{array}{rl}
\tilde\Delta((M\Del N)_2)\hskip-1em&\hspace{-0,2cm}=(M\Del N)_{(1)3}\Del_{(\C\Del\D)^{\Del 3}}(M\Del N)_{(2)1} \\
&\simeq(M_{(1)}\Del N_{(1)})_3\Del_{(\C\Del\D)^{\Del 3}}(M_{(2)}\Del N_{(2)})_1 \\
&\simeq(M_{(1)3}\Del N_{(1)3})\Del_{(\C\Del\D)^{\Del 3}}(M_{(2)1}\Del N_{(2)1}) \\
&\simeq\tilde\Delta_{\M}(M_2)\Del\tilde\Delta_{\N}(N_2).
\end{array}$$
Since both $\tilde\Delta_{\M}$ and $\tilde\Delta_{\N}$ are equivalences and the bifunctor $\Del$ is biexact, we conclude that $\tilde\Delta$ is an
equivalence, too, proving that $\M\Del\N$ is Azumaya.

If $\varepsilon_{\M}$ and $\varepsilon_{\N}$ are the counit functors for $\M$ and $\N$ respectively, then $\varepsilon:=\varepsilon_{\M}\Del\varepsilon_{\N}$
is clearly the counit functor for $\M\Del\N$.
\qed\end{proof}

\bigskip

Take $(\M,\alpha), (\N,\beta)\in \dul{Z}^2(\C,\dul{\Pic})$ such that $[(\M,\alpha)]=[(\N,\beta)]$ in $H^2(\C,\dul{\Pic})$.
Let $F: \C\to \D$ be a symmetric tensor functor and consider two further functors:
$\Lambda,\Psi: \C\to \C\Del\D$ given by:
$$\Lambda(X)=I_{\C}\Del F(X)\qquad\textnormal{and}\qquad\Psi(X)=X\Del I_{\D}$$
for $X\in\C$, here $I_{\C}, I_{\D}$ are the unit objects in $\C$ and $\D$, respectively. They are symmetric monoidal functors.
We may rewrite them as: $\Lambda=(\Id_{\C}\Del F)e_1$ and $\Psi=-\Del\Id_{\D}$. Then clearly if $\C=\D$ and $F=\Id_{\C}$, we have:
$\Lambda=e_1$ and $\Psi=e_2$, where the functors $e_1,e_2: \C\to \C\Del\C$ are those from \equref{e's}.
Applying \prref{coh maps dul} to the functors $\Lambda$ and $\Psi$ we obtain:
$$[\Lambda^*(\M,\alpha)]=[\Psi^*(\N,\beta)]$$
in $H^2(\C\Del \D/vec,\dul{\Pic})$, where $\Lambda^*, \Psi^*: \dul{Z}^2(\C,\dul{\Pic})\to\dul{Z}^2(\C\Del\D,\dul{\Pic})$ are the induced functors as
in \equref{fun ind cat}. We apply $\chi^{-1}$ from \coref{Brauer} to both sides and we get:
$$[(\M\Del_{\C^{\Del 2}} {}_{\Lambda^{\Del 2}}(\C\Del\D)^{\Del 2}]=[(\N\Del_{\C^{\Del 2}}{}_{\Psi^{\Del 2}}(\C\Del\D)^{\Del 2}]$$
in $AzQCor(\C\Del \D)$. By the definition of $\Lambda$ and $\Psi$ this can be rewritten as:
\begin{equation*} 
[\C\Del\C\Del(\M\Del_{\C^{\Del 2}} {}_{F^{\Del 2}}(\D^{\Del 2}))]=[\N\Del\D\Del\D].
\end{equation*}
If we multiply the inverse of one of them with the other one we get the unit, hence the quasi coring categories:
$$
(\C\Del\C\Del(\M^{op}\Del_{\C^{\Del 2}} {}_{F^{\Del 2}} (\D\Del\D))) \Del_{\C\Del\C\Del\D\Del\D} \left( \N\Del\D\Del\D\right)$$
\begin{equation} \eqlabel{2fold appl of coh Prop}
\simeq \N\Del \M^{op}\Del_{\C\Del\C} {}_{F^{\Del 2}}(\D^{\Del 2}) 
\end{equation}
and similarly:
\begin{equation} \eqlabel{1fold appl of coh Prop}
\N^{op}\Del(\M\Del_{\C\Del\C} {}_{F^{\Del 2}}(\D^{\Del 2}))
\end{equation}
are elementary.

\medskip

We will apply this idea in the results that follow. Observe that for an Azumaya quasi $\C$-coring category $(\M,\Delta)$ and its corresponding
$(\M,\alpha)\in \dul{Z}^2(\C,\dul{\Pic})$, the inverse of the latter in $Z^2(\C,\dul{\Pic})$ is represented by $(\M^{op}, (\alpha^{op})^{-1})$,
and its corresponding quasi coring category is $(\M^{op},(\Delta^{op})^{-1})$.

\begin{prop} \prlabel{5.2}
Let $\M$ be an Azumaya quasi $\C$-coring category. Then $\M\Del \M^{op}$ is an elementary quasi coring category over $vec$.
\end{prop}

\begin{proof}
Consider $\HH(\M)=(\M,\alpha)\in \dul{Z}^2(\C,\dul{\Pic})$, where $\HH$ is from \equref{cats Az and coc}. The assertion follows from \equref{2fold appl of coh Prop}
with $\N=\M, \C=\D$ and $F=\Id_{\C}$.
\qed\end{proof}

Let $(\C,\M)$ and $(\D,\N)$ be Azumaya quasi coring categories over $vec$. We say that $\M$ and $\N$ are {\em Brauer equivalent} (notation: $\M\sim \N$)
if there exist elementary quasi coring categories $\E_1$ and $\E_2$ over $vec$ such that $\M\Del \E_1\simeq \N\Del \E_2$ as (Azumaya) invertible quasi coring categories over $vec$.
Since the Deligne tensor product of two elementary quasi coring categories is elementary, it is easy to show that $\sim$ is an equivalence relation.
Let $AzQCor(vec)$ be the set of Brauer equivalence classes of equivalence classes of Azumaya quasi coring categories over $vec$.
From \prref{5.2} it follows that $AzQCor(vec)$ is an (abelian) group under the operation induced by the Deligne tensor product $\Del$, with unit element $[vec]$ and
the inverse $[(\M^{op},(\Delta^{op})^{-1})]$ of $[(\M,\Delta)]$. Observe that it is abelian since the monoidal category $(\Cc_k, \Del, vec)$ is symmetric.

\begin{lma}\lelabel{5.5}
Let $\M$ and $\E$ be Azumaya quasi $\C$-coring categories so that $\E=\Can(\Ll;\C)$ is elementary, for some $\Ll\in\dul{\Pic}(\C)$.
Then the Azumaya quasi coring categories $\M\Del_{\C^{\Del 2}}\E$ and $\M$ are Brauer equivalent.
\end{lma}

\begin{proof}
For $\E$ elementary we have $\HH(\E)=(\Ll^{op}\Del \Ll,\lambda_{\Ll})$, and if $\HH(\M)=(\M,\alpha)$, we have:
$$[(\M\Del_{\C^{\Del 2}}\E,\alpha\Del_{\C^{\Del 3}}\lambda_{\Ll})]=[(\M,\alpha)]$$
in $H^2(\C,\dul{\Pic})$. Let now $\N=\M\Del_{\C^{\Del 2}}\E, \C=\D$ and $F=\Id_{\C}$ in \equref{2fold appl of coh Prop}.
Then $(\M\Del_{\C^{\Del 2}}\Ee)\Del \M^{op}=\Pp$ is an elementary quasi coring category, and
$$(\M\Del_{\C^{\Del 2}}\Ee)\Del \M^{op}\Del\M=\Pp\Del \M.$$
By \prref{5.2} we have that $\M\Del\M^{op}$ is elementary, so it follows that $\M\Del_{\C^{\Del 2}}\Ee\sim \M$.
\qed\end{proof}

\begin{lma}\lelabel{5.5a}
Let $F: \C\to \D$ be a functor between two symmetric finite tensor categories. If $\M$
is an Azumaya quasi $\C$-coring category, then $\M\sim F'(\M)=\M\Del_{\C^{\Del 2}} {}_{F^{\Del 2}}\Can(\D;\D)$.
\end{lma}

\begin{proof}
We set $\HH(\M)=(\M,\alpha)$ as before, now putting $\N=\M$ in \equref{1fold appl of coh Prop} we have that
$\M^{op}\Del (\M\Del_{\C^{\Del 2}} {}_{F^{\Del 2}}\Can(\D;\D))=\Ee$ is an elementary quasi $\C\Del \D$-coring category. We then have:
$$\M\Del \M^{op}\Del (\M\Del_{\C^{\Del 2}}\Can(\D;\D))\simeq\M\Del\Ee,$$
being $\M\Del\M^{op}$ 
elementary, 
we may deduce that $\M\sim F'(\M)=\M\Del_{\C^{\Del 2}} {}_{F^{\Del 2}}\Can(\D;\D)$.
\qed\end{proof}

\begin{prop}\prlabel{relative and full}
Let $\C$ be a symmetric finite tensor category. There is a well-defined group monomorphism:
$$i_{\C}: AzQCor(\C)\to AzQCor(vec),~~i_{\C}([\M])=[\M].$$
If $F:\C\to\D$ is a functor between two symmetric finite tensor categories, then there is a commutative diagram:
$$
\bfig
\putmorphism(380,0)(0,-1)[`AzQCor(\D).`\tilde{F}]{450}1l
\putmorphism(380,0)(1,0)[AzQCor(\C)`AzQCor(vec)`i_{\C}]{850}1a
\putmorphism(1300,-30)(-2,-1)[``]{850}{-1}r
\putmorphism(1300,-50)(-2,-1)[``i_{\D}]{935}0r
\efig
$$
\end{prop}

\begin{proof}
Let us first make sure that $i_{\C}$ is well-defined. For that purpose take $[\M]=[\N]$ in
$AzQCor(\C)=K_0[\dul{AzQCor}(\C)]/\Can(\C)$. That means that for an elementary quasi $\C$-coring category $\E$
we have $\M\Del_{\C^{\Del 2}}\N^{op}\simeq\Ee$. Multiplying this in the group $AzQCor(\C)$ by $\N$, i.e. tensoring the equation over $\C\Del\C$ by $\N$, we obtain
$\M\simeq\N\Del_{\C^{\Del 2}}\Ee$. By \leref{5.5} this is Brauer equivalent to $\N$ in
$AzQCor(vec)$, thus $i_{\C}$ is well-defined. Let us now show that $i_{\C}$ is a group
homomorphism. Let $\M$ and $\N$ be two Azumaya quasi $\C$-coring categories. Then by \prref{5.2}
the quasi $\C\Del \C$-coring category $\M^{op}\Del \M=\Ee_1$, and clearly the quasi $\C$-coring category $\M\Del_{\C^{\Del 2}}\M^{op}=\Ee_2$,
are both elementary. Now we find:
\begin{eqnarray*}
\M\Del \N&\sim & (\M\Del \N)\Del_{\C^{\Del 4}} (\M^{op}\Del \M)\\
&\simeq & (\M\Del_{\C^{\Del 2}} \M^{op})\Del (\N\Del_{\C^{\Del 2}} \M)\\
&\sim &\N\Del_{\C^{\Del 2}} \M\simeq \M\Del_{\C^{\Del 2}} \N
\end{eqnarray*}
where the first identity holds by \leref{5.5}, the second identity is due to \leref{Weq} and the third one holds because $vec$ is trivially an elementary quasi coring category.
Thus, in $AzQCor(vec)$ we have:
$$i_{\C}([\M\Del_{\C^{\Del 2}} \N]) = [\M\Del \N] = i_{\C}([\M])i_{\C}([\N]).$$
This proves that $i_{\C}$ is a group map. It is clearly injective.

Finally, from \leref{5.5a} it follows that $i_{\C}([\M])=[\M]=[\M\Del_{\C^{\Del 2}} {}_{F^{\Del 2}}(\D\Del\D)]=(i_{\D}\circ \tilde{F})([\M])$.
\qed\end{proof}

\begin{thm}
We have the following isomorphism of groups:
$$AzQCor(vec)\cong \colim~ AzQCor(\bullet/vec)\cong \colim~ H^2(\bullet/vec,\dul{\Pic}).$$
\end{thm}

\begin{proof}
From \prref{relative and full} and the definition of the colimit we have that there is a unique map
$$i:\ \colim~ AzQCor(\bullet/vec)\to AzQCor(vec).$$
Let $A$ be an arbitrary abelian group, and let $\beta_{\C}: AzQCor(\C)\to A$ be a collection of maps such that
$$\beta_{\D}\circ \widetilde{F}=\beta_{\C},$$
for every functor of symmetric finite tensor categories $F: \C\to \D$.

We define the map
$$\beta: AzQCor(vec)\to A$$
as follows. An element $X$ of $AzQCor(vec)$ is represented by an Azumaya quasi $\C$-coring category $\M$ for some symmetric finite tensor category $\C$.
Then define
$$\beta(X)=\beta_{\C}([\M]).$$
To show that it is well-defined, take an Azumaya quasi $\D$-coring category $\N$ such that $[\N]=[\M]$. 
Then
$$\M\Del (\D\Del\D)\sim \M\sim \N\sim \N\Del (\C\Del\C)$$
and from the injectivity of $i_{\C\Del\D}$ (see \prref{relative and full}) we obtain that
$[\M\Del (\D\Del\D)]=[\N\Del (\C\Del\C)]$ in $AzQCor(\C\Del\D)$, hence
$$\beta_{\C}([\M])=\beta_{\C\Del\D}([\M\Del (\D\Del\D)])=\beta_{\C\Del\D}([\N\Del (\C\Del\C)])=\beta_{\D}([\N]),$$
so $\beta$ is well-defined. It is constructed so that the diagrams
$$
\bfig
\putmorphism(-30, 500)(2, -1)[\phantom{AzQCor(\C)}`\phantom{P\cong P \Del I}`]{850}1r
\putmorphism(-30, 450)(2, -1)[\phantom{AzQCor(\C)}`\phantom{P\cong P \Del I}`\beta_{\C}]{850}0l
\putmorphism(770, 500)(0, -1)[`A`\beta]{420}1r
\putmorphism(0, 500)(1, 0)[AzQCor(\C)`AzQCor(vec)`i_{\C}]{850}1a
\efig
$$
commute for all $\C\in FTC_{symm}$. To prove that $\beta$ is unique with such a property, assume that there exists
$\gamma: AzQCor(vec)\to A$ with $\gamma i_{\C}=\beta_{\C}$ for all symmetric finite tensor categories $\C$.
Let $X=[\M]\in AzQCor(vec)$ for some symmetric finite tensor category $\C$ and an Azumaya quasi $\C$-coring category $\M$. We then
have: $\gamma(X)=\gamma i_{\C}([\M])=\beta_{\C}([\M])=\beta i_{\C}([\M])=\beta(X)$, hence $\gamma=\beta$. This means that
$AzQCor(vec)$ satisfies the required universal property of a colimit. Finally, apply \equref{colimit Hn}.
\qed\end{proof}


Observe that all the results in this subsection are valid for Azumaya quasi- as well for proper coring categories, so we have:
$$AzCor^*(vec)\cong \colim~ AzCor^*(\bullet/vec)$$
for the corresponding groups of Azumaya coring categories.

\subsection*{Acknowledgments} The work of the author was partially supported by the Mathematical Institute of the Serbian Academy of Sciences and Arts (MI SANU),
Serbia. 

\bibliographystyle{amsalpha}

\end{document}